\documentclass[11pt]{elsarticle}
\usepackage{mathtools,amssymb,amsthm}
\usepackage{graphicx}   
\usepackage{multirow}  
\usepackage{verbatim}
\usepackage{booktabs}
\usepackage{bm}   
\usepackage[left=2.5cm,right=2.5cm,top=2.5cm,bottom=2.5cm]{geometry}  
\usepackage{lineno,hyperref}
\usepackage{appendix}
\usepackage[justification=centering]{caption}  
\usepackage{color}

\usepackage{amsmath}
\usepackage{makecell}
\usepackage{tabularx}
\usepackage{hyperref}

\usepackage{amssymb}
\usepackage{mathrsfs}

\usepackage{graphicx}

\allowdisplaybreaks[4]  

\newtheorem{theorem}{Theorem}[section]
\newtheorem{assumption}[theorem]{Assumption}
\newtheorem{lemma}[theorem]{Lemma}

\newtheorem{remark}[theorem]{Remark}

\newtheorem{example}[theorem]{Example}

\numberwithin{equation}{section}

\renewcommand{\le}{\leq}
\renewcommand{\ge}{\geq}

\begin{document}
	
	\begin{frontmatter}
		\title{Uniform-in-time strong convergence rates of fully discrete approximations for stochastic Cahn--Hilliard equations with multiplicative noise}
		
		\author{Jiaqin He \fnref{label}}
		\ead{hejiaqin@seu.edu.cn}
		
		\author{Nan Deng \fnref{label}}
		\ead{deng_nan@seu.edu.cn}
		
		\author{Shuhang Zhang \fnref{label}}
		\ead{zhangshuhang@seu.edu.cn}
		
		\author{Wanrong Cao\corref{cor1} \fnref{label}}
		\ead{wrcao@seu.edu.cn}

		\cortext[cor1]{Corresponding author. }
		\address[label]{School of Mathematics, Southeast University, Nanjing 210096, P.R.China.}
		
		\begin{abstract}
			This paper investigates the uniform-in-time strong convergence rates of a fully discrete approximation for the stochastic Cahn--Hilliard equation driven by multiplicative noise in spatial dimensions $d\in\{1,2,3\}$. The proposed scheme combines a spectral Galerkin method in space with a backward Euler scheme in time. The main analytical difficulties arise from the state-dependent stochastic perturbation, the absence of a global monotonicity structure for the nonlinear term, and the fourth-order nature of the Cahn--Hilliard operator. In particular, these features make the derivation of uniform $L^{\infty}$-moment estimates highly nontrivial in three dimensions. For the continuous equation, by utilizing the It\^{o} formula to $\|u\|^p$ and introducing the energy functional $\mathcal{E}(u(t))$, we derive the uniform moment boundedness of the solution. At the fully discrete level, we develop discrete energy estimates and close the required high-order moment bounds through an induction argument. Based on these regularity estimates, we deduce uniform-in-time strong convergence rates for the fully discrete scheme. Moreover, we prove the existence and uniqueness of invariant measures for both the exact dynamics and the fully discrete numerical dynamics. Numerical experiments are provided to confirm the theoretical findings.

			\vspace{.5em}
			
			\vspace{.5em}

		\end{abstract}
		
	\end{frontmatter}

	\section{Introduction}

	    The stochastic Cahn--Hilliard equation is a fundamental phase-field model originating from the description of phase separation in binary alloys subject to random fluctuations \cite{cahn1958,scarpa18}. The Allen--Cahn and Cahn--Hilliard equations are two prototypical phase-field models derived from the same free-energy functional but endowed with different gradient-flow structures: the former is an $L^2$-gradient flow, whereas the latter is an $H^{-1}$-gradient flow. In recent years, the long-time behavior and numerical approximation of the stochastic Allen--Cahn equation have been extensively investigated; see, e.g., \cite{chen2020,brehier2022,wang2024,liu2023,liu2024,liu2025}. By contrast, analogous results for the stochastic Cahn--Hilliard equation remain comparatively limited, primarily due to its fourth-order spatial operator, non-globally Lipschitz polynomial nonlinearity, and the presence of multiplicative noise.

     	Denote $\mathcal{D}:=(0,1 )^d$ with $d \in \{1,2,3\}$ and let $H:=L^{2}(\mathcal{D})$. In this paper, we consider the stochastic Cahn--Hilliard equation driven by multiplicative noise
    	\begin{equation}\label{Cahn-Hilliard equation}
	    	\left\{\begin{aligned}
		    	&du(t)+A(Au(t)+F(u(t)))dt = G(u(t))dW(t), \quad t >0, \\
		    	&u(0) = u_0 , 
	    	\end{aligned}\right.
    	\end{equation}
	    where $A=-\Delta:D(A)\subset H\to H$ denotes the negative Laplacian operator with homogeneous Neumann boundary conditions, and $W(t)$ is a Q-Wiener process defined on a filtered probability space $(\Omega, \mathcal{F}, \left \{ \mathcal{F}_{t} \right \}_{t\ge 0},\mathbb{P})$. The nonlinear term $F$ is the Nemytskii operator associated with $f(u)=u^{3}-u$, namely, $F(u)(x)=f(u(x)), x \in \mathcal{D}$. In accordance with the conservative structure of the Cahn--Hilliard dynamics, we consider a mass-conserving diffusion coefficient
	    $G:H^{-1}\to\mathcal{L}_2^{-1}$, where $\mathcal{L}_2^{-1}$ denotes the space of Hilbert--Schmidt operators taking values in the zero-mean space $\dot H^{-1}$; the precise definition and assumptions on $G$ are
	    given in Section~\ref{main results}. Consequently, $G(u)\,\mathrm dW(t)$ has zero spatial mean for every $u\in H^{-1}$, so the stochastic forcing preserves the total mass of the solution.

	    In the additive-noise case, corresponding to $G=I$ in Eq. \eqref{Cahn-Hilliard equation}, a standard and effective approach for the stochastic Cahn--Hilliard (SCH) equation is to exploit the additive structure by separating the stochastic convolution from the remaining nonlinear dynamics. More precisely, the solution is often decomposed as 
	    $$
	    	u(t)=v(t)+\mathcal{O}_t, \quad
	    	\mathcal{O}_t=\int_{0}^{t}S(t-s)dW(s),
	    $$
	    where $\mathcal{O}_t$ denotes the stochastic convolution associated with the linearized equation. This decomposition converts the original SPDE into a random evolution equation for $v(t)$, with the rough stochastic forcing separated from the nonlinear dynamics. The regularity and moment estimates of $\mathcal{O}_t$ can then be established independently and subsequently combined with the one-sided Lipschitz structure of the nonlinear drift to control the nonlinear terms. Based on this decomposition strategy, Kov\'acs et al.~\cite{kovacs2011} and Furihata et al.~\cite{furihata2018} first established strong convergence for spatial finite element semi-discretizations and fully discrete implicit Euler schemes, respectively. Later developments refined these results by deriving strong convergence rates for various fully discrete approximations, including spectral Galerkin methods combined with implicit Euler schemes~\cite{qi2024}, accelerated implicit Euler schemes~\cite{cui2018}, and explicit tamed exponential Euler schemes~\cite{lin2024,cai2023}. Weak convergence estimates have also been obtained in subsequent works~\cite{caimeng2023,brehier2025}. In the multiplicative-noise setting, however, this strategy is no longer directly applicable: the noise intensity depends on the evolving state of the system, and thus the stochastic forcing cannot be isolated as a precribed stochastic convolution independent of the nonlinear dynamics.

    	Compared with the finite-time theory, results concerning the long-time behavior of the SCH equation remain relatively limited. In the additive noise case, Da Prato et al.~\cite{da1996} established the
    	well-posedness and the existence and uniqueness of an invariant measure for the SCH equation. More recently, Deng et al.~\cite{Deng2025} further investigated the long-time behavior from $H_{\alpha}^{-1}$ to the more regular state space $H_{\alpha}$, and proposed a fully discrete scheme based on a tamed exponential Euler method in time and a finite difference discretization in space. Moreover, uniform-in-time strong convergence rates of the fully discrete approximation were established.

	    For SCH equations with bounded and Lipschitz continuous diffusion coefficients, for instance, $|g(u)|\le C$, Hong~\cite{hong2024} and Cui~\cite{cui2026} studied the fully discrete schemes and established optimal strong convergence rates in one and three spatial dimensions, respectively.

	    For the one-dimensional stochastic Cahn--Hilliard equation driven by multiplicative noise, Cui et al.~\cite{cui2023} considered a diffusion coefficient $G$ given by the Nemytskii operator $G(u)(x)=g(u(x))$,
	    where $g$ satisfies the sublinear growth condition $|g(r)|\le C(1+|r|^{a}),  a<1$. They established the well-posedness of the equation and derived strong convergence rates for its spectral Galerkin approximation.
	    A key ingredient in their analysis is the decomposition $u(t)=v(t)+\mathcal{O}_t$,  where $v(t)$ solves a random deterministic equation and $\mathcal{O}_t$ denotes a solution-dependent stochastic convolution. The $H^{-1}$-energy estimates for $v(t)$ yields
	    \begin{equation*}
	    	\|v(t)\|_{\dot{H}^{-1}}^{2}+\int_{0}^{t}\|\nabla v(s)\|^{2}ds+\int_{0}^{t}\|v(s)\|_{L^{4}}^{4}ds
	    	\le C\Big ( 1+\int_{0}^{t}\|\mathcal{O}_s\|_{L^{4}}^{4}ds \Big ),
	    \end{equation*}
	    whereas, for $0\le t\le T$, the stochastic convolution satisfies
	    \begin{equation*}
	    	\mathbb{E}\left[\|\mathcal{O}_t\|_{L^{p}}^{q} \right]\le C(T)\Big (1+ \mathbb{E}\Big[\Big ( \int_{0}^{t}\|v(s)\|_{L^{p}}^{4a}ds\Big )^{\frac{q}{4}} \Big]\Big ).
	    \end{equation*}
	    The sublinear growth condition $a<1$ is essential for closing the moment estimates, since it allows the terms generated by the multiplicative noise to be absorbed into the available dissipative bounds. The factorization method and the smoothing properties of the semigroup $S(t)$ yield sufficiently high spatial regularity of $\mathcal{O}_t$. Meanwhile, the one-dimensional Gagliardo--Nirenberg inequality
	    $$
	    \|\varphi\|_{L^6}
	    \le C\big(
	    \|A\varphi\|^{\frac{1}{6}}\|\varphi\|^{\frac{5}{6}}
	    +\|\varphi\|
	    \big), \quad \varphi \in H^2,
	    $$
	    provides the required $L^6$-control of the remaining component $v$. Then the authors derive moment bounds for $u(t)$ in $L^\infty$. A fully discrete approximation combining spectral Galerkin method in space with an drift-implicit Euler scheme in time was subsequently investigated in~\cite{cui2020}. However, the resulting moment and error estimates involve constants that depend on the terminal time $T$. Moreover, the corresponding Gagliardo--Nirenberg estimates in higher spatial dimensions involve larger interpolation exponents and therefore provide insufficient control to close the same estimates. Consequently, the one-dimensional argument cannot be directly extended either to uniform-in-time analysis or to higher-dimensional problems.

	    To the best of our knowledge, the uniform-in-time strong convergence of numerical approximations for the SCH equation driven by multiplicative noise, especially in higher spatial dimensions, has not yet been systematically investigated. It is therefore natural to compare this problem with the corresponding long-time numerical analysis for stochastic Allen--Cahn equations, where substantial progress has recently been made.
	    
	    For stochastic Allen--Cahn equations driven by multiplicative noise, fully discrete schemes with favorable long-time stability properties have recently been developed; see, e.g., \cite{liu2023, liu2024, liu2025,cailiu2026}. A crucial step in their analysis is the uniform-in-time moment boundedness of the mild solution in $H^{1}$. Combining the dissipative structure
	    \[
	    \left\langle F(u), Au \right\rangle
	    =
	    \left\langle \nabla F(u), \nabla u \right\rangle
	    \ge -\|\nabla u\|^{2}
	    \]
	    with the Itô formula applied to $\|u(t)\|_{1}^{p}$ and the appropriate Lipschitz and linear growth conditions on the diffusion coefficient yields uniform-in-time moment estimates in the $H^{1}$-norm. Subsequently, by using the mild formulation of the solution and the smoothing properties of the semigroup $S(t)$, the authors further derived uniform-in-time moment estimates in the $L^\infty$-norm. These estimates constitute a key ingredient in establishing the long-time stability and uniform-in-time strong convergence of the numerical approximations.

     	We consider the stochastic Cahn--Hilliard equation driven by multiplicative noise in spatial dimensions $d\in\{1,2,3\}$, and propose a fully discrete approximation based on a spectral Galerkin method in space and a backward Euler discretization in time. The long-time analysis of such approximations presents several substantial difficulties. First, since the noise intensity depends on the evolving state of the system, the decomposition strategy used in the additive-noise case is no longer directly applicable, which necessitates a simultaneous treatment of the noise and nonlinear terms in establishing uniform moment bounds. Second, the multiplicative noise considered here is structurally different from the Nemytskii-type noise studied in~\cite{cui2023}. The analysis in \cite{cui2023} relies on the sublinear growth $\alpha<1$ and on one-dimensional Gagliardo--Nirenberg inequalities. By contrast, we consider mass-conserving diffusion coefficients that take values in the zero-mean space and are allowed to exhibit linear growth. Consequently, the arguments in~\cite{cui2023} cannot be directly extended to Eq.~\eqref{Cahn-Hilliard equation}, particularly in three spatial dimensions, where Sobolev embeddings provide substantially weaker control of the nonlinear terms than in lower-dimensional settings. Finally, for SCH equations driven by multiplicative noise, the argument developed in \cite{liu2023} cannot be applied directly. Unlike in the stochastic Allen--Cahn setting, the nonlinear drift in the SCH equation takes the form $AF(u)$. Consequently, when applying the It\^{o} formula to the $H^{1}$-energy, one encounters the term
     	\begin{align*}
     		|\langle A^{1/2}F(u),A^{3/2}u\rangle|
     		\leq\varepsilon\|A^{3/2}u\|^2+C_{\varepsilon}\left(1+\|u\|_{L^\infty}^{4}\right)\|\nabla u\|^2.
     	\end{align*}
     	The appearance of the $L^{\infty}$-norm prevents the resulting $H^{1}$-energy estimate from being closed without an additional uniform-in-time $L^{\infty}$-moment bound. 
     	Therefore, a direct application of the It\^{o} formula yields, at most, uniform-in-time moment bounds in $H$ for the solution in Eq.~\eqref{Cahn-Hilliard equation}. Consequently, establishing uniform-in-time $H^{1}$ and $L^{\infty}$-moment bounds remains one of the main technical challenges in the long-time analysis of three-dimensional stochastic Cahn--Hilliard equations driven by multiplicative noise.

	   The aim of this work is to establish uniform-in-time strong convergence rates for the proposed fully discrete scheme. We first combine the It\^{o}  formula applied to $\|u(t)\|^{p}$ with the dissipativity condition in Assumption \ref{dissipative condition} to derive a uniform-in-time moment bound for $u(t)$ in $H$. We then lift the regularity estimates from $H$ to higher spatial regularity $H^1$, for which the energy functional
       $$
	    \mathcal{E}(u(t))=\int _{\mathcal{D}}\Big [ \frac{1}{2}\left | \nabla u(t,x)\right |^2 +\frac{1}{4}\left(u^2(t,x)-1\right)^2 \Big ]dx
	   $$ 
       severs as a fundamental tool. For the fully discrete implicit scheme, however, the argument is substantially different, since the It\^{o}  formula is not available in the discrete setting. We first employ discrete energy estimates to obtain the uniform-in-time second-moment bound $\|u_{m}^{N}\|_{L^{2}(\Omega,H)}\le C$. A considerably more delicate analysis is then required: by using mathematical induction, we obtain the uniform-in-time boundedness  $\|u_{m}^{N}\|_{L^{p}(\Omega,H)}\le C$ for $p\ge 2$. Similarly, through the energy functional $\mathcal{E}(u_{m}^{N})$ and mathematical induction, we derive uniform-in-time moment bounds for the numerical solution in $L^\infty$, namely $\|u_{m}^{N}\|_{L^{p}(\Omega,L^{\infty})}\le C$. Building on these uniform moment estimates, we then prove the uniform strong convergence rates of the numerical scheme. Our main result, stated in Theorem \ref{main theorem}, shows that, under Assumptions \ref{diffusion coeffient} and \ref{dissipative condition}, for $0<\tau <\min\big\{\tau_0^*,\tau_1^*, \frac{1}{\lambda _{1}^{2}},\frac{\lambda_{1}^{2}-\lambda_{1}-5L_{2}}{144L_2^2(\lambda_{1}^{2}+6)}\big\} $, it holds that
	   \begin{equation*}
		   \mathbb{E}[\|u(t_{m})-u_{m}^N\|^2  ]\le C\big(\lambda _{N+1}^{-2}+\tau\big). 
       \end{equation*}

	  The remainder of this paper is organized as follows. In Section~\ref{main results}, we introduce the fully discrete scheme for Eq.\eqref{Cahn-Hilliard equation} and state the main results. Section~\ref{section-uniform-moment-bounds} establishes uniform-in-time moment bounds for both the exact and numerical solutions. Section~\ref{ergodicity} establishes the existence and uniqueness of invariant measures for both the Markov process associated with Eq.~\eqref{Cahn-Hilliard equation} and the corresponding numerical Markov chain. In Section~\ref{strong convergence}, we derive the uniform-in-time strong convergence rates of the fully discrete scheme. Finally, Section~\ref{numerical experiments} presents numerical experiments that illustrate and validate the theoretical results.

    \section{Preliminaries and main results }\label{main results}
    
    In this section, we introduce the notation and assumptions used throughout the paper, formulate the fully discrete approximation of \eqref{Cahn-Hilliard equation}, and state the main results. We denote by $C$ the generic positive constant which may be different from line to line but is always independent of the discretization parameters and the time variable $t$. Sometimes we use $C(\cdot)$ to emphasize its dependence on the parameters in brackets.
    
    \subsection{Notation and Assumptions}
    
    Given two separable Hilbert spaces $(U, \left \langle\cdot,\cdot   \right \rangle_{U}, \left \| \cdot  \right \| _{U})$ and $(V, \left \langle\cdot,\cdot   \right \rangle_{V}, \left \| \cdot  \right \| _{V})$. Let $\mathcal{L}(U;V)$ be the Banach space of all bounded linear operators from $U$ to $V$ with the norm $\left \| \cdot  \right \| _{\mathcal{L}(U,V)}$. Denote by $\mathcal{L}_2(U;V)\subset\mathcal{L}(U;V)$ the space consisting of all Hilbert-Schmidt operators from $U$ into $V$, equipped with norms and inner products respectively given by
    \begin{equation}
    	\lVert T \rVert _{\mathcal{L}_2(U,V)}^{2}=\sum_{j\in \mathbb{N}}{\lVert T\psi_j \rVert}_{V}^2, \  \  \  \    \left< S,T \right> _{\mathcal{L}_2(U,V)}=\sum_{j\in \mathbb{N}}{(S\psi _j,T\psi _j)_{V}},\notag
    \end{equation}
    where the above definitions are independent of specific choice of orthonormal basis $\left\{\psi _j \right\} _{j=1}^{\infty}$ of $U$. 
    To simplify the notation, we often write $\mathcal{L}(U)$ and $\mathcal{L}_2(U)$ instead of $\mathcal{L}(U,U)$ and $\mathcal{L}_2(U,U)$, respectively. If $S\in \mathcal{L}(U)$ and $T\in \mathcal{L}_2(U,V)$, then the following inequalities hold,
    \begin{equation}
    	\lVert ST \rVert _{\mathcal{L}_2(U,V)}\le \lVert S \rVert _{\mathcal{L}(U)}\lVert T \rVert _{\mathcal{L}_2(U,V)} .\notag
    \end{equation}

    For any $p\ge 1$, we denote by $L^p(\mathcal{D})$ the Banach space of real-valued, $p$-times integrable functions endowed with the the associated norm $\left \| \cdot \right \|_{L^{p} }$. For brevity, we denote $H:=L^2(\mathcal{D})$ equipped with the norm $\lVert \cdot \rVert$ and inner product $\left \langle \cdot ,\cdot \right \rangle $. Let $(\Omega, \mathcal{F}, \left \{ \mathcal{F}_{t} \right \}_{t\ge 0},\mathbb{P})$ be a filtered probability space. Given a Banach space $\left ( \mathcal{X}, \|\cdot \|_{\mathcal{X}} \right )$, denote $L^p(\Omega ;\mathcal{X})$ the space of $\mathcal{X}$-valued $p$-times integrable random variables with the norm $\lVert \cdot \rVert _{L^p(\Omega ;\mathcal{X})}=\left( \mathbb{E}\left [  \lVert \cdot  \rVert_{\mathcal{X}} ^p\right ] \right) ^{\frac{1}{p}}$.

    Throughout this paper, we define $\dot{H}= \{ v\in H: \int _{\mathcal{D}}vdx=0 \}$ and an orthogonal projector $\mathcal{P}:H\to\dot{H}$ by
    \begin{equation*}
    	\mathcal{P}v=v-\left|\mathcal{D}\right|^{-1}\int _{\mathcal{D}}vdx.
    \end{equation*}
    We define $A=-\bigtriangleup$, the negative of the Neumann Laplacian with domain of definition  
    \begin{equation}
    	\mathrm{dom}(A)= \{ v\in H^{2}(\mathcal{D}):\frac{\partial v}{\partial n} =0 \ \text{on} \  \partial \mathcal{D}     \}.\notag 
    \end{equation}
    It is known that $A$ is a positive definite, self-adjoint and unbounded linear operator on $\dot{H}$ with compact inverse. There exists a family of orthonormal basis $\{\phi_{k}\}_{k=0}^{\infty }$ of $H$ with corresponding eigenvalues $\left \{\lambda_{k}\right \}_{k=0}^{\infty }$ such that  	
    \begin{equation*}
    	A\phi_{k}=\lambda_{k}\phi_{k},\ 0=\lambda _{0}<\lambda _{1}\le \cdots \le\lambda _{k}\le \cdots,  \lambda_{k}\to \infty , \text{as}\ k\to \infty 
    \end{equation*}
    where $\phi_0=\left | \mathcal{D}  \right |^{-\frac{1}{2} }$ and $\{\phi_{k}\}_{k=1}^{\infty }$ forms an orthonormal basis of $\dot{H}$. 
    We define 
    \begin{equation*}
    	\left\|v\right\|_{\dot{H}^{\beta}}:= \Big( \sum_{j=1}^{\infty }\lambda _{j}^{\beta }\left|\left\langle v, \phi _j \right\rangle\right|^2 \Big)^{\frac{1}{2}}, \ \ \ 
    	\left \|v\right \|_{\beta}:=\left ( \left\|v\right\|_{\dot{H}^{\beta}}^2+\left| \left\langle v, \phi_0 \right\rangle \right |^2 \right )^{\frac{1}{2}}, \quad \alpha \in \mathbb{R},
    \end{equation*}
    and corresponding spaces, 
    \begin{equation*}
    	\dot{H}^{\beta} =D(A^{\frac{\beta}{2}})= \{ v\in \dot{H}: \left\|v\right\|_{\dot{H}^{\beta}}<\infty\}, \ \ \
    	H^{\beta}= \{ v\in H: \left\| v \right\|_{\beta}<\infty\}.
    \end{equation*}
     It is well-known that for integer $\beta \in \{1,2\}$, the norm $\left\| \cdot \right\|_{\beta}$ is equivalent on $H^{\beta}$ to the standard Sobolev norm $\left\| \cdot \right\|_{W^{\beta,2}}$. The analytical semigroup $S(t)=e^{-A^2t}$ is generated by the operator $-A^2$. The regularity properties of the semigroup $S(t)$ have been stated as follows in \cite[Lemma 2.1]{Deng2025}
     	\begin{align}
     		&\left \| A^\nu S(t)\varphi  \right \|\le C\Big ( t^{-\frac{\nu }{2} }e^{-\frac{\lambda _{1}^{2} }{2}t }\left \| \varphi  \right \|+k(\nu )\left | \left \langle \varphi ,\phi_0 \right \rangle  \right |   \Big ),  &&\nu \ge 0, \quad \varphi \in H,\label{semigroup_1}\\
     		&\|A^{-\eta}(S(t)-I)\varphi \|\le Ct^{\frac{\eta }{2}}\|\varphi \|, && \eta \in [0,2], \quad \varphi \in H, \label{semigroup_2}\\
     		&\left \| A^\mu S(t)(I-S(s))\varphi  \right \|\le C  t^{-\frac{\mu -\rho +2r}{2} }s^re^{-\frac{\lambda _{1}^{2} }{2}t }\left \| A^{\rho }\varphi\right \|,  &&\mu -\rho +2r \ge 0, r \in [0,1], \quad \varphi \in H^{2\rho}, \label{semigroup_3}
     	\end{align}
     where the function $k:[0,\infty)\to \mathbb{R}$ is defined by $k(0)=1$ and $k(x)=0$ for all $x>0$.

    To simply the presentation, we assume the covariance operator $Q$ is a self-adjoint and nonnegative definite operator on $\dot{H}$. We denote $ U_0=Q^{1/2}\dot{H}$ and denote by $\mathcal{L}_{2}^{\theta}:=HS(U_0;\dot{H}^{\theta})$ the space of Hilbert-Schmidt operators from $U_0$ to $\dot{H}^{\theta}$ with $\theta \in \left \{-1,0,1 \right \}$. Let $W(t)$ be a standard $\dot{H}$-valued $Q$-Wiener process which satisfies the series representation 
    \begin{equation*}
    	W(t)=\sum_{k=1}^{\infty}\sqrt{q_k}\phi_k\beta_{k}(t),
    \end{equation*}
    where $\left \{ \beta_{k}(t)\right \}_{k=1}^{\infty }$ is a sequence of mutually independent Brownian motions, and $\left \{ (q_{k}, \phi_k)  \right \}_{k=1}^{\infty }$ are the eigenpairs of $Q$.

    We next introduce several assumptions on the diffusion coefficient $G$, which will play a crucial role in the subsequent long-time analysis. Suppose for $u \in H^{-1}$, $G(u)$ is a Hilbert-Schmidt operator in $\mathcal{L}_{2}^{-1}$. This implies that the stochastic term is mass-conservative: regardless of the value of $u$, the multiplicative noise $G(u)dW(t)$ takes values in the zero-mean space $\dot{H}^{-1}$ and therefore does not change the total mass of the solution. We further impose the following Lipschitz and growth conditions on $G$.
    
   \begin{assumption}\label{diffusion coeffient}
   	\noindent(i) Global Lipschitz continuity: There exist constants $L_1,L_{2} >0$ such that
   		\begin{align}
   			\|G(u)-G(v)\|_{\mathcal{L}_2^0}^2 &\le L_1 \|u-v\|^2, 
   			&& u,v \in H, \label{global Lipschitz in H}\\
   			\|G(u)-G(v)\|_{\mathcal{L}_2^{-1}}^2 
   			&\le L_{2} \|u-v\|_{-1}^2, && u,v \in H^{-1}.\label{global Lipschitz in H^-1}
   		\end{align}
 
   	\noindent(ii) Growth and spatial regularity: There exist constants $L_{3},L_{4},L_5,L_6,L_7,L_8>0$ such that
   		\begin{align}
   			&\|G(u)\|_{\mathcal{L}_2^0}^2 \le L_{3} \|u\|^2 + L_{4}, 
   			&& u \in H, \label{linear growth in H}\\
   			&\|G(u)\|_{\mathcal{L}_2^{1}}^2 
   			\le L_5 \|u\|_{1}^2 + L_6, 
   			&& u \in H^1,\label{linear growth in H^1}\\
   			&\sum_{k=1}^{\infty}\|G(u)Q^{1/2}\phi_k\|_{\dot{H}^{\frac{d}{4}}}^2 
   			\le L_7 \|u\|_{L^4}^2 + L_8, 
   			&& u \in L^4. \label{linear growth in L^4}
   		\end{align}
   \end{assumption}
     
     To obtain uniform-in-time estimates, we further impose the following condition on the relevant Lipschitz and growth constants in Assumption~\ref{diffusion coeffient}. It ensures that the dissipation is strong enough to control the effect of the noise. Let $Z$ denote the standard Gaussian random variable. Define $\mathfrak{m}_{r}:= \mathbb{E}|Z|^r=\frac{2^{\frac{r}{2}}\Gamma(\frac{r+1}{2})}{\sqrt{\pi } }$. Let $C_{d,4}>0$ be the Sobolev embedding constant from $H^{\frac{d}{4}}\hookrightarrow L^4$.

     \begin{assumption}[Dissipativity condition]\label{dissipative condition}
     	 Assume that the following dissipativity conditions hold:
     		\begin{align}
     			\noindent(i) \quad 	& \max \Big \{5L_{2}, 59L_3 \Big \} < \lambda_1(\lambda_1 - 1), \label{dissipative for H} \\
     			\noindent(ii) \quad 	& \max\{11, 20p-29\}(L_{5}+4C_{d,4}^2L_{7})+ \frac{1}{4p}\mathfrak{m}_{4p}(\sqrt{8}C_{d,4}^2L_7)^{2p} < \lambda_1^2 - 1,\quad 1 \le p \le 16.\label{dissipative for H^1}
     		\end{align}

     \end{assumption}
     
    \begin{remark}
    	The conditions \eqref{global Lipschitz in H}-\eqref{linear growth in H^1} are standard in the numerical analysis of SPDE with multiplicative noise; see, e.g., \cite{liu2023,liuqiao2021}. Condition \eqref{linear growth in L^4} is additionally imposed to control the diffusion coefficient in $L^4$, which is needed in the uniform-in-time estimates for the energy functional $\mathcal{E}(u)$ defined in \eqref{energy func}. 
    \end{remark}
    
   \begin{example} 
   	\begin{enumerate}
   		\item[(i)] \textit{Additive noise.}
   		In the additive-noise case $G=I$, the noise is assumed to take values in the zero-mean space $\dot{H}$, and is therefore mass-conservative. In this case, the conditions \eqref{global Lipschitz in H}--\eqref{linear growth in L^4} reduce to
   		$\|A^{\frac{1}{2}}Q^{\frac{1}{2}}\|_{\mathcal{L}_2(\dot{H})}< \infty$.
   		
   		\item[(ii)] \textit{Multiplicative noise.}
        For $\alpha\ge \frac{d}{8}$, we define  
      	$$
      	G(u)Q^{1/2}\phi_k
    	=
    	A^{-\alpha}\mathcal{P}\big( uQ^{1/2}\phi_k\big),
    	\qquad k\ge1 ,
    	$$
    	and assume that 
    	$$
          \sum_{k=1}^{\infty} q_k\|\phi_k\|_{W^{1,\infty}}^2<\infty . 
        $$
    	A direct computation shows that 
     	\begin{equation*}
   	 	\begin{aligned}
   			\|G(u)-G(v)\|_{\mathcal L_2^0}^2
   			=\sum_{k=1}^{\infty}q_{k}
   			\|A^{-\alpha}\mathcal{P}\big( (u-v)\phi_k\big)\|^2        
   			\le\lambda_{1}^{-2\alpha}
   			\sum_{k=1}^{\infty}q_{k}\|(u-v)\phi_k\|^2                         
   			\le L_{1}\|u-v\|^2,
   		\end{aligned}
    	\end{equation*}
    	with $L_{1}=\lambda_{1}^{-2\alpha  }(\sum_{k=1}^{\infty}q_{k}\|\phi_k\|_{L^\infty}^2)$. 
    	Similarly, through the product estimate
    	\begin{equation*}
    		\begin{aligned}
   			\|u\phi_k\|_{-1}\le \|u\|_{-1}\|\phi_k\|_{W^{1,\infty}}, \quad
   			\|u\phi_k\|_{1}\le \|u\|_{1}\|\phi_k\|_{W^{1,\infty}},
   	    	\end{aligned}
    	\end{equation*}
     	we can deduce that 
    	\begin{equation*}
   		\begin{aligned}
   			\|G(u)-G(v)\|_{\mathcal L_2^{-1}}^2\le L_{2}\|u-v\|_{-1}^2, \quad
   	        \|G(u)\|_{\mathcal L_2^{1}}^2\le L_{5}\|u\|_{1}^2,
   		\end{aligned}
   	    \end{equation*}
    	with $L_{2}=L_{5}=\lambda_{1}^{-2\alpha }(\sum_{k=1}^{\infty}q_{k}\|\phi_k\|_{W^{1,\infty}}^2)$ and $L_{6}=0$. Since $G(0)=0$, the global Lipschitz estimate also gives $\|G(u)\|_{\mathcal L_2^0}^2\le L_1\|u\|^2$,
    	and hence one may take $L_3=L_1$ and $L_4=0$.
    	Finally, the H\"{o}lder inequality yields that
    	\begin{equation*}
   		\begin{aligned}
   			\sum_{k=1}^{\infty}\|G(u)Q^{1/2}\phi_k\|_{\dot{H}^{\frac{d}{4}}}^2
   			\le\sum_{k=1}^{\infty}q_{k}\|A^{-\alpha}\mathcal{P}\big( u\phi_k\big)\|_{\dot{H}^{\frac{d}{4}}}^2
   			\le\sum_{k=1}^{\infty}q_{k}\|u\phi_k\|^2
   			\le L_7\|u\|_{L^4}^2.
   		\end{aligned}
    	\end{equation*} 		    
        with $L_7=\lambda_{1}^{\frac{d}{4}-2\alpha }(\sum_{k=1}^{\infty}q_{k}\|\phi_k\|_{L^4}^2)$.
        \end{enumerate}
     \end{example}

  Under our assumptions, Theorem~2.9 and the continuous-dependence result in Theorem~2.10 of Scarpa~\cite{scarpa18} yield a unique global strong solution in the sense of Definition~2.5 therein. Applying the variation-of-constants formula to the corresponding integral identity gives
  \begin{equation}\label{mild solution}
  	u(t)=S(t)u_0-\int_0^t S(t-s)AF(u(s))\,\mathrm{d}s
  	+\int_0^t S(t-s)G(u(s))\,\mathrm{d}W(s).
  \end{equation}
   Conversely, every mild solution in the same regularity class satisfies the integral formulation of Definition~2.5 in \cite{scarpa18}. Hence, uniqueness follows from Theorem~2.10. Therefore, Eq.~\eqref{Cahn-Hilliard equation} admits a unique global mild solution.

  \subsection{The fully discrete scheme}
  
  Let $\tau$ denote the temporal step size and $t_m=m\tau$ with $m \ge 0$. The projection operator $P_N:H\to H_N$ is defined by $P_Nx=\sum_{k=0}^{N}\left \langle x,\phi_k \right \rangle \phi_k$ for $\forall x\in H$. By the Parseval identity, we have
  \begin{equation}\label{projection_operator}
  	\left \| \left ( I-P_N\right )v \right \| \le \lambda _{N+1}^{-\frac{\alpha}{2}}\left \|v \right \|_{\alpha}, \quad  \forall v\in H^{\alpha },\ \alpha > 0.
  \end{equation}
  The fully discrete numerical scheme combining a spectral Galerkin method in space with a drift-implicit Euler scheme in time, is defined as  
  \begin{equation}\label{full-discrete scheme}
  	\left\{\begin{aligned}
  		&u_{m+1}^{N} - u_{m}^{N}
  		= -\tau A^2 u_{m+1}^{N}
  		- \tau A P_N F(u_{m+1}^{N})
  		+ P_N G(u_{m}^{N})\triangle_m W, \\	
  		&u^N_0 = P_Nu_0 , 
  	\end{aligned}\right.
  \end{equation}
  where $\triangle_m W:=W(t_{m+1})-W(t_m)$. 
  By introducing the operator $ S_{\tau,N}:=(I+\tau A^2)^{-1}P_N$ and the family operators $\{ S_{\tau,N}^{m}\}_{m\ge 1}$:
  \begin{equation*}
  	S_{\tau,N}^{m}v=(1+\tau A^2)^{-m}P_Nv=\sum_{k=0}^{N}(1+\tau \lambda_{k}^{2})^{-m}\left \langle v,\phi_{k}\right \rangle\phi_{k},  \quad  \forall  v\in H ,
  \end{equation*}
  the numerical solution admits a unique mild representation:
  \begin{equation}\label{full mild solution}
  	\begin{aligned}
  		u_{m+1}^{N}=S_{\tau,N}^{m+1}u_{0}^{N}-\tau\sum_{i=0}^{m}S_{\tau,N}^{m+1-i}A P_N F(u_{i+1}^{N})+\sum_{i=0}^{m}S_{\tau,N}^{m+1-i}P_N G(u_{i}^{N})\triangle_i W.
  	\end{aligned}
  \end{equation}

  \subsection{Main results}
  We now state the main results of the paper. We first establish uniform-in-time moment bounds and temporal regularity for the exact solution, followed by the corresponding uniform moment bounds for the fully discrete approximation. Based on these stability estimates, we derive the uniform-in-time strong convergence rate of the fully discrete
  scheme.

   \begin{theorem}\label{continous uniform moment}
  	Let $2\le p\le 24$ and $u_0\in L^{6p}( \Omega;H^{\gamma})$ for $\gamma \in \big(\max\{1, \frac{d}{2}\}, 2\big )$. Suppose Assumptions \ref{diffusion coeffient} and \ref{dissipative condition} hold. There exists a constant $C$ such that
  	\begin{equation}\label{uniform stability}
  		\sup_{t\geq0}\|u(t)\|_{L^p(\Omega;H^{\gamma})}\le C .
  	\end{equation}
  	If, in addition, assume that $u_{0}\in L^{6p}(\Omega; H^{2})$, then there exists a constant $C$ such that 
  	\begin{equation}\label{time regularity}
  		\|u(t)-u(s)\|_{L^{p}(\Omega;H)} \le C(t-s)^{\frac{1}{2}}.
  	\end{equation}
  \end{theorem}
   The proof of Theorem \ref{continous uniform moment} is given in Section~\ref{section-uniform-moment-bounds}, with some technical details deferred to \ref{detailed proof of Theorem 2.4}.

  \begin{theorem}\label{all regularity of full discrete}
  	Let $2\le p\le 8$ and fix $\gamma \in \big(\max\{1, \frac{d}{2}\}, 2\big )$.  Suppose that $u_0\in L^{64} ( \Omega;H^{\gamma})$, Assumptions \ref{diffusion coeffient} and \ref{dissipative condition} hold. Then, for $0<\tau<\min\big\{\tau_0^*,\tau_1^*\big\}$, it holds that
  	\begin{equation*}
  		\begin{aligned}
  			\sup_{m\geq0}\|u_{m}^{N}\|_{L^{p}(\Omega; H^{\gamma })} \le C. 
  		\end{aligned} 		
  	\end{equation*}
  \end{theorem}
  \begin{proof} 
  	We prove the uniform-in-time $H^1$-moment bounds for the numerical solution $u_{m}^{N}$ in Section \ref{section-uniform-moment-bounds}. For each fixed
  	$\gamma<2$, combining this bound with the smoothing
  	properties for $S_{\tau,N}^m$ in Lemma~\ref{semigroup full discrte}
  	and proceeding as in the proof of
  	Theorem~\ref{continous uniform moment} yields the asserted
  	$H^\gamma$-moment estimate. The constant $C$ is
  	independent of $N$, $\tau$, and $m$, but is generally not uniform
  	as $\gamma\uparrow2$. We omit the repetitive details.
  \end{proof}

  These continuous and discrete stability estimates constitute the main ingredients of the subsequent error analysis and lead to a uniform-in-time strong convergence rate for the fully discrete approximation.
  
   \begin{theorem}\label{main theorem}
  	Let $u_{0} \in L^{144}(\Omega;H^2)$. Suppose that Assumptions \ref{diffusion coeffient} and \ref{dissipative condition} hold. Then, for every
  	$N\in\mathbb N$ and $0<\tau <\min\big\{\tau_0^*,\tau_1^*, \frac{1}{\lambda _{1}^{2}},\frac{\lambda_{1}^{2}-\lambda_{1}-5L_{2}}{144L_2^2(\lambda_{1}^{2}+6)}\big\} $, there exists a constant $C>0$ such that 
  	\begin{equation*}
  		 \mathbb{E}[\|u(t_{m})-u_{m}^N\|^2  ]\le C\big(\lambda _{N+1}^{-2}+\tau\big). 
  	\end{equation*}
  \end{theorem}
  The proof of Theorem~\ref{main theorem} is given in
  Section~\ref{strong convergence}.

   \section{Uniform-in-time moment estimates}\label{section-uniform-moment-bounds}
  
  \subsection{Uniform moment bounds for the exact solution}
  In this subsection, we establish the uniform-in-time moment bounds for the exact solution $u(t)$. We first combine the It\^o formula with the
  dissipative structure of the Cahn--Hilliard equation to derive uniform
  moment bounds in $H$. To obtain the higher spatial estimates required
  in Theorem~\ref{continous uniform moment}, we introduce the shifted
  energy functional
  \begin{equation}\label{energy func}
  	\begin{aligned}
  		\mathcal{E}(u(t))=\int _{\mathcal{D}}\Big [ \frac{1}{2}\left | \nabla u(t,x)\right |^2 +\frac{1}{4}\left(u^2(t,x)-1\right)^2 \Big ]dx+C^{\ast },
  	\end{aligned}
  \end{equation}
  where $C^{\ast }$ is chosen sufficiently large to ensure that $\|G(u(t))\|_{\mathcal{L}_{2}^{0}}^2\le \frac{1}{(p-1)\lambda_{1}}\mathcal{E} (u(t))$. 
  We then establish uniform-in-time estimates for
  $\mathcal{E}(u(t))$, which constitute a key ingredient in the proof of
  Theorem~\ref{continous uniform moment}.

    \begin{lemma}\label{uniform boundness in L^2}
    	Let $2 \le p \le 72$ and $u_0 \in L^{p}(\Omega;H)$. Suppose that Assumptions \ref{diffusion coeffient} and \ref{dissipative condition} hold. Then the mild solution $u(t)$ for Eq.\eqref{Cahn-Hilliard equation} satisfies 
    	\begin{equation*}
    		\begin{aligned}
    			\sup_{t\geq0}\left \| u(t) \right \|_{L^p(\Omega; H)} \le C.
    		\end{aligned}	
    	\end{equation*}

    \end{lemma}
    \begin{proof}
    Applying the It\^o formula to $\|u\|^p$, we obtain
    \begin{equation}\label{Ito formula for u^p}
    	\begin{aligned}
    		d\|u\|^p
    		=& -p \|u\|^{p-2} \|Au\|^2 dt
    		- p \|u\|^{p-2} \langle u, AF(u) \rangle dt 
    		+ \frac{p}{2} \|u\|^{p-2} \|G(u)\|_{\mathcal{L}_2^0}^2 dt\\
    		&+ \frac{p(p-2)}{2} \|u\|^{p-4}
    		\sum_{k=1}^\infty \langle u, G(u)Q^{\frac{1}{2}}\phi_k \rangle^2 dt+ p \|u\|^{p-2} \langle u, G(u)dW \rangle \\
    		=&:I_1+I_2+I_3+I_4+I_5.
    	\end{aligned}
    \end{equation}
    Taking the inner product in $H$ on both sides of Eq.\eqref{Cahn-Hilliard equation} with $\phi _{0}$, we obtain $\left \langle u(t), \phi_{0} \right \rangle=\left \langle u_0, \phi _{0}\right \rangle$. Consequently, for the term $I_1$, it holds that
    \begin{equation*}
    	\begin{aligned}
    		&-\|A u\|^2 \le -\lambda_1^2 \|u\|^2+\lambda_1^2\left \langle u_0, \phi _{0}\right \rangle^2, \\
    		&I_1\le -p\lambda_1^2\|u\|^{p}+p\lambda_1^2\left \langle u_0, \phi _{0}\right \rangle^2\|u\|^{p-2}
    		\le \left ( -p\lambda_1^2+\varepsilon\right )\|u\|^{p}+C\left (p,\varepsilon \right) \left \langle u_0, \phi _{0}\right \rangle^{p}.
    	\end{aligned}
    \end{equation*}	 
    For the term $I_2$, by using the integration by parts formula yields
    \begin{equation*}
    	\begin{aligned}
    	    I_2&= -3p \|u\|^{p-2} \|u \nabla u\|^2dt
    		+ p \|u\|^{p-2} \|\nabla u\|^2dt.
    	\end{aligned}
    \end{equation*}	
    Furthermore, for the term $I_3$ and $I_4$, applying \eqref{linear growth in H} from Assumption \ref{diffusion coeffient} gives
    \begin{equation*}
    	\begin{aligned}
    		I_3 &\le \frac{p}{2} \|u\|^{p-2} (L_{3}\|u\|^2+L_{4})dt
    		\le \frac{p}{2}L_{3}\|u\|^p dt+\frac{p}{2}L_{4}\|u\|^{p-2}dt,\\
    		I_4 &\le \frac{p(p-2)}{2} \|u\|^{p-2}\|G(u)\|_{\mathcal{L}_2^0}^2 dt
    		\le \frac{p(p-2)}{2}L_{3}\|u\|^pdt +\frac{p(p-2)}{2}L_{4}\|u\|^{p-2}dt.
    	\end{aligned}
    \end{equation*}
    For $\gamma_{1}>0$, applying the product rule to $e^{\gamma_1 t}\|u(t)\|^p$ and integrating over $[0,t]$, we obtain
    \begin{equation}\label{integral of Ito formula for u^p}
    	\begin{aligned}
    		e^{\gamma_{1}t}\|u(t)\|^p-\|u_{0}\|^p=&-p\int_{0}^{t}e^{\gamma_{1}r}\|u(r)\|^{p-2}\left [\|Au(r)\|^2 +\langle u(r), AF(u(r)) \rangle\right ]dr \\
    		&+\gamma_{1}\int_{0}^{t}e^{\gamma_{1}r}\|u(r)\|^{p}dr+ \frac{p}{2}\int_{0}^{t}e^{\gamma_{1}r}\|u(r)\|^{p-2}\|G(u(r))\|_{\mathcal{L}_2^0}^2dr\\
    		&+ \frac{p(p-2)}{2}\int_{0}^{t}e^{\gamma_{1}r}\|u(r)\|^{p-4}\sum_{k=1}^\infty \langle u(r), G(u(r))Q^{\frac{1}{2}}\phi_k \rangle^2dr \\
    		&+p\int_{0}^{t}e^{\gamma_{1}r}\|u(r)\|^{p-2}\langle u(r), G(u(r))dW(r) \rangle.
    	\end{aligned}  	
    \end{equation}
    Taking expectations on both sides of \eqref{integral of Ito formula for u^p}, and using the classical localization argument based on stopping times, see ~\cite{doob1953}, we obtain
    \begin{equation*}
    	\mathbb{E}\Big[\int_{0}^{t}e^{\gamma_{1}r}\|u(r)\|^{p-2}\langle u(r), G(u(r))dW(r) \rangle \Big ]=0.
    \end{equation*}
     Utilizing $\|\nabla u\|^2 \le \frac{1}{2\lambda_{1}} \|A u\|^2 + \frac{\lambda_{1}}{2} \|u\|^2$, and combining the above inequalities into \eqref{integral of Ito formula for u^p}, we obtain
    \begin{equation*}
    	\begin{aligned}
    		\mathbb{E}[\|u(t)\|^p]\le& e^{-\gamma_{1}t}\mathbb{E}[\|u_0\|^p]-\gamma_{1}^{\ast}
    		\int_{0}^{t}e^{-\gamma_{1}(t-r)}\mathbb{E}[\|u(r)\|^{p}]dr\\
    		&+C(p,\varepsilon)\mathbb{E}[\left \langle u_0,\phi_0 \right \rangle ^{p}]\int_{0}^{t}e^{-\gamma_{1}(t-r)}dr
    		+C(p,L_4,\varepsilon)\int_{0}^{t}e^{-\gamma_{1}(t-r)}dr,
    	\end{aligned}
    \end{equation*}	
    with $\gamma_{1}^{\ast}:=p\big(\lambda_{1}^{2}-\lambda_{1}-\frac{p-1}{2}L_{3} \big)-2\varepsilon-\gamma _1$. 
    By \eqref{dissipative for H} from Assumption \ref{dissipative condition}, we can choose $\varepsilon$ and $\gamma_{1}$ to be sufficiently small such that $\gamma_{1}^{\ast}>0$. 
    Therefore, it follows that 
    \begin{equation*}
    	\mathbb{E}[\|u(t)\|^p]\le C.
    \end{equation*}
    The proof is complete.
    	 
    \end{proof}

   To derive uniform-in-time moment bounds for the exact solution $u(t)$ in $H^1$, we consider the shifted energy functional $\mathcal{E}$ defined in \eqref{energy func}. A direct calculation shows that, for $u,v,w\in H^1$,
    \begin{equation}\label{frechet deri}
    	\left \langle D\mathcal{E}(u), v \right \rangle =\left \langle Au+F(u),v \right \rangle ,\quad  D^2\mathcal{E}(u)(v,w)=\left \langle \nabla v, \nabla w \right \rangle +\left \langle {f}'(u)v,w  \right \rangle,
    \end{equation}
    where $D\mathcal{E}(u)$ and $D^2\mathcal{E}(u)$ denote the first and second-order Fr\'{e}chet derivatives of $\mathcal{E}$ at $u$.

    \begin{lemma}\label{uniform estimate of u(t) in H^1}
    	Let $2 \le p \le 72$ and $u_0 \in L^{2p}(\Omega;H^1)$. Suppose Assumptions \ref{diffusion coeffient} and \ref{dissipative condition} hold. Then there exists a positive constant $C$ such that
    	\begin{equation}
    		\sup_{t\geq0}\mathbb{E}\big [ \left ( \mathcal{E}(u(t)) \right )^{\frac{p}{2}}  \big ]\le C .
    	\end{equation}
    \end{lemma}
    \begin{proof}
    	By applying the It\^{o} formula to  $\left ( \mathcal{E}(u(t)) \right )^p$ and using the properties in \eqref{frechet deri}, we have
    	\begin{equation}\label{ito_enengy}
    		\begin{aligned}
    			d\left (\mathcal{E} (u(t)) \right )^p=&-p\left (\mathcal{E} (u(t)) \right )^{p-1}\big \|A^{\frac{1}{2}}(Au(t)+F(u(t)))\big \|^2dt\\
    			&+p\left (\mathcal{E} (u(t)) \right )^{p-1}\left \langle Au(t)+F(u(t)),G(u(t))dW(t) \right \rangle\\
    			&+\frac{p}{2}\left (\mathcal{E} (u(t)) \right )^{p-1} \sum_{k=1}^{\infty}D^2\mathcal{E}(u(t))\big ( G(u(t))Q^{\frac{1}{2}}\phi_{k}, G(u(t))Q^{\frac{1}{2}}\phi_{k}\big)dt\\
    			&+\frac{p(p-1)}{2}\left (\mathcal{E} (u(t)) \right )^{p-2}\sum_{k=1}^{\infty}\big \langle Au(t)+F(u(t)),G(u(t))Q^{\frac{1}{2}}\phi_{k}\big \rangle^{2}dt\\
    			=:&I_1+I_2+I_3+I_4.
    		\end{aligned}
    	\end{equation}
    	For the first term $I_1$,  we can easily obtain 
    	\begin{equation*}
    		\begin{aligned}
    			I_1 \le -p\lambda _{1}\left (\mathcal{E} (u(t)) \right )^{p-1}\|Au(t)+F(u(t))\|_{\dot{H}}^2dt.
    		\end{aligned}
    	\end{equation*}
    	Let $C_{d,4}>0$ be the Sobolev embedding constant from $H^{\frac{d}{4}}\hookrightarrow L^4$.
    	For the term $I_3$, through \eqref{linear growth in H^1} and \eqref{linear growth in L^4} in Assumption \ref{diffusion coeffient} and the H\"{o}lder inequality, we get
    		\begin{align*}
    			 I_3=&\frac{p}{2}\left (\mathcal{E} (u(t)) \right )^{p-1} \sum_{k=1}^{\infty}D^2\mathcal{E}(u(t))\big ( G(u(t))Q^{\frac{1}{2}}\phi_{k}, G(u(t))Q^{\frac{1}{2}}\phi_{k}\big)dt\\
    			=&\frac{p}{2}\left (\mathcal{E} (u(t)) \right )^{p-1} \sum_{k=1}^{\infty}\left ( \big \|\nabla G(u(t))Q^{\frac{1}{2}}\phi_{k}\big\|^2+\big\langle {f}'(u(t))G(u(t))Q^{\frac{1}{2}}\phi_{k},G(u(t))Q^{\frac{1}{2}}\phi_{k}\big\rangle \right )dt\\
    			\le& \frac{p}{2}\left (\mathcal{E} (u(t)) \right )^{p-1}\Big ( \left \|G(u(t))\right \|_{\mathcal{L}_{2}^{1}}^{2}+\sum_{k=1}^{\infty}\left \|{f}'(u(t))\right \|\big \| G(u(t))Q^{\frac{1}{2}}\phi_{k}\big\|_{L^4}^{2}\Big )dt\\
    			\le& \frac{p}{2}\left (\mathcal{E} (u(t)) \right )^{p-1}\Big ( L_5\|u\|_{1}^2+L_6+\big(3C_{d,4}^2\|u\|_{L^4}^2+C_{d,4}^2\big)\sum_{k=1}^{\infty}\big \| G(u(t))Q^{\frac{1}{2}}\phi_{k}\big\|_{\dot{H}^{\frac{d}{4}}}^{2}  \Big )dt \\
    			\le &\frac{p}{2}\left (\mathcal{E} (u(t)) \right )^{p-1}\Big ( L_5\|\nabla u\|^2+(3C_{d,4}^2L_7+\varepsilon )\|u\|_{L^4}^4+L_{5}\left \langle u_0,\phi _0 \right \rangle^{2}+C(\varepsilon,C_{d,4},L_6,L_7, L_8)\Big )dt\\
    			\le &  p\max\left \{L_5+\varepsilon, 6C_{d,4}^2L_7+3\varepsilon \right \}\left (\mathcal{E} (u(t)) \right )^{p}dt+C(p,\varepsilon,L_5)\left \langle u_0,\phi _0 \right \rangle^{2p}+C(p,\varepsilon,C_{d,4},L_6,L_7,L_8)dt
    		\end{align*}
    	 In the definition of the energy functional $\mathcal{E} (u(t))$ in \eqref{energy func}, by taking $C^{\ast }$ large enough, it holds that 
    	\begin{equation*}
    		\begin{aligned}
    			\|G(u(t))\|_{\mathcal{L}_{2}^{0}}^2 \le L_{3}\|u(t)\|^2+L_{4} 
    			\le \frac{1}{(p-1)\lambda_{1}}\mathcal{E} (u(t)).
    		\end{aligned}
    	\end{equation*}
    	Then, for the term $I_4$, we can obtain
    	\begin{equation*}
    		\begin{aligned}
    			I_4=&\frac{p(p-1)}{2}\left (\mathcal{E} (u(t)) \right )^{p-2}\sum_{k=1}^{\infty}\big \langle \mathcal{P}(Au(t)+F(u(t)))+(I-\mathcal{P})(Au(t)+F(u(t))),G(u(t))Q^{\frac{1}{2}}\phi_{k}\big \rangle^{2}dt\\
    			\le  &\frac{p(p-1)}{2}\left (\mathcal{E} (u(t)) \right )^{p-2}\|\mathcal{P}(Au(t)+F(u(t)))\|^2\|G(u(t))\|_{\mathcal{L}_{2}^{0}}^2dt\\
    			\le &\frac{p}{2\lambda_{1}}\left (\mathcal{E} (u(t)) \right )^{p-1}\|Au(t)+F(u(t))\|_{\dot{H}}^2dt,
    		\end{aligned}
    	\end{equation*}
    	Similar to the proof in Lemma \ref{uniform boundness in L^2}, for $\gamma_{2}>0$, applying the product rule to $e^{\gamma_2 t}\mathcal{E}(u(t))^p$ and integrating over $[0,t]$, we substitute the above estimates into the resulting identity to obtain
    	\begin{equation}\label{integral for energy functional}
    		\begin{aligned}
    			e^{\gamma_{2}t}\left(\mathcal{E} (u(t))\right)^p \le&\left(\mathcal{E} (u_0)\right)^p-p\big(\lambda _{1}-\frac{1}{2\lambda_{1}}\big)
    			\int_{0}^{t}e^{\gamma_{2}r}\left(\mathcal{E} (u(r))\right)^{p-1}\|Au(r)+F(u(r))\|_{\dot{H}}^2dr\\
    			&+\left ( p\max\left \{L_5+\varepsilon, 6L_7+3\varepsilon\right\}+\gamma_{2} \right )\int_{0}^{t}e^{\gamma_{2}r}\left(\mathcal{E} (u(r))\right)^{p}dr\\
    			&+C(p,\varepsilon,L_5)\left \langle u_0,\phi _0 \right \rangle^{2p}\int_{0}^{t}e^{\gamma_{2}r}dr
    			+C(\varepsilon,L_6,L_7,L_8)\int_{0}^{t}e^{\gamma_{2}r}dr\\
    			&+p\int_{0}^{t}e^{\gamma_{2}r}\left (\mathcal{E} (u(r)) \right )^{p-1}\left \langle Au(r)+F(u(r)),G(u(r))dW(r) \right \rangle.
    		\end{aligned}
    	\end{equation} 
    	By simple calculations, we have
    	\begin{equation}\label{estimate of derivative Eu}
    		\begin{aligned}
    			\left \|\mathcal{P}(Au(t)+F(u(t)))\right\|_{\dot{H}^{-1}}^{2} &=\left \langle A^{-1}(Au(t)+F(u(t))), Au(t)+F(u(t))\right \rangle\\
    			&=\left \langle Au(t),u(t) \right \rangle+2\left \langle \mathcal{P}u(t),F(u(t)) \right \rangle+\left \langle F(u(t)), A^{-1}F(u(t)) \right \rangle\\
    			&\ge \left \| \nabla u(t) \right \|^2+2\left \|u(t)\right \|_{L^4}^4-2\left \|u(t)\right \|^2-2\left \langle (I-\mathcal{P})u(t),F(u(t)) \right \rangle \\
    			&\ge (\left \| \nabla u(t) \right \|^2+\frac{1}{2}\left \|u(t)\right \|_{L^4}^4-\left \|u(t)\right \|^2)+\frac{3}{2}\left \|u(t)\right \|_{L^4}^4-\left \|u(t)\right \|^2\\
    			&-\Big(\left \|u(t)\right \|_{L^4}^4+\frac{27}{16}\left | \left \langle u_0,\phi_0 \right \rangle \right | ^4-2\left | \left \langle u_0,\phi_0 \right \rangle \right |^2\Big)\\
    			&\ge 2\mathcal{E} (u(t))-\frac{27}{16}\left | \left \langle u_0,\phi_0 \right \rangle \right | ^4-(2C^{\ast }+1).
    		\end{aligned}
    	\end{equation}
    	Taking expectations on both sides of \eqref{integral for energy functional}, and using the localization argument based on stopping times, we obtain
    	\begin{equation*}
    		\mathbb{E}\Big[\int_{0}^{t}e^{\gamma_{2}r}\left (\mathcal{E} (u(r)) \right )^{p-1}\left \langle Au(r)+F(u(r)),G(u(r))dW(r) \right \rangle\Big ]=0.
    	\end{equation*}
    	Therefore, after simplification, we obtain
    	\begin{equation*}
    		\begin{aligned}
    			 \mathbb{E}[(\mathcal{E} (u(t)))^p] \le& e^{-\gamma_{2}t}\mathbb{E}[(\mathcal{E}(u_0))^p]
    			-\gamma_{2}^{\ast}\int_{0}^{t}e^{-\gamma_{2}(t-r)}\mathbb{E}[(\mathcal{E} (u(r)))^p]dr\\
    			&+C(p,\varepsilon,L_5)\mathbb{E}[\left | \left \langle u_0,\phi_0 \right \rangle \right |^{4p}]\int_{0}^{t}e^{-\gamma_{2}(t-r)}dr+C(p,\varepsilon,C_{d,4}, C^{\ast },L_6,L_7,L_8)\int_{0}^{t}e^{-\gamma_{2}(t-r)}dr,
    		\end{aligned}
    	\end{equation*}
    	where $\gamma_{2}^{\ast}=p\Big(2\lambda_{1}^{2}-1-\max\left \{L_5+2\varepsilon, 6C_{d,4}^2L_7+4\varepsilon\right\}\Big)-\gamma_{2}$. Through \eqref{dissipative for H^1} in Assumption \ref{dissipative condition}, we  can choose
    	$\varepsilon$ and $\gamma_{2}$ small enough to ensure that $\gamma_{2}^{\ast}>0$. The desired estimate then follows immediately, and the proof is complete.

    \end{proof}

    \subsection{Uniform moment bounds for the fully discrete scheme}
    In this subsection, we establish uniform-in-time moment bounds for the proposed fully discrete numerical solution $u_m^N$. In contrast to the continuous setting, the Itô formula is not directly available at the discrete level. We first derive a uniform second-moment estimate in $H$ by taking the inner product of the fully discrete equation \eqref{full-discrete scheme} with $u_{m+1}^N$. This estimate serves as the base estimate for a mathematical induction argument, which subsequently yields uniform bounds for higher-order moments of $u_m^N$.

\begin{lemma}\label{p-moment uniform in H}
	Let $2 \le p \le 32$ and assume that $u_0\in L^{32}(\Omega;H)$. Suppose that Assumptions \ref{diffusion coeffient} and \ref{dissipative condition} hold. Define $\tau_0^*:=\displaystyle\inf_{2\le p\le32}\tau_0(p,L_3)$. Then, for every $0<\tau<\tau_0^*$, it holds that
	\begin{equation*}
		\sup_{m\geq0}\|u_{m}^{N}\|_{L^{p}(\Omega;H)}\le C.
	\end{equation*}
\end{lemma}
 \begin{proof}
 	To establish the uniform boundedness of $\|u_{m+1}^{N}\|_{L^{p}(\Omega;H)}$, we divide the proof into three steps.\\
 	\textbf{Step 1.} We begin with the case $p=2$.
 	Taking the inner product in $H$ on both sides of \eqref{full-discrete scheme} with $u_{m+1}^{N}$ yields
    \begin{equation*}
    	\begin{aligned}
    		\frac{1}{2}\big(\|u_{m+1}^{N}\|^2 - \|u_{m}^{N}\|^2 + \|u_{m+1}^{N}-u_{m}^{N}\|^2\big)
    		= &-\tau \|Au_{m+1}^{N}\|^2
    		-\tau \langle AF(u_{m+1}^{N}), u_{m+1}^{N}\rangle \\
    		&+ \langle G(u_{m}^{N})\triangle_m W, u_{m+1}^{N}-u_{m}^{N} \rangle + \langle G(u_{m}^{N})\triangle_m W, u_{m}^{N} \rangle .
    	\end{aligned}
    \end{equation*}
 	By the formula of integration by parts and the Cauchy-Schwarz inequality, it holds that
 	\begin{equation*}
 		\begin{aligned}
 			-\tau \langle AF(u_{m+1}^{N}), u_{m+1}^{N}\rangle
 			&=-3\tau \|u_{m+1}^{N}\nabla u_{m+1}^{N}\|^2+\tau\|\nabla u_{m+1}^{N}\|^2 ,\\
 			\langle G(u_{m}^{N})\triangle_m W, u_{m+1}^{N}-u_{m}^{N} \rangle
 			&\le \frac{1}{2}\|u_{m+1}^{N}-u_{m}^{N}\|^2+\frac{1}{2}\|G(u_{m}^{N})\triangle_m W\|^2.
 		\end{aligned}
 	\end{equation*}
 	Collecting the above estimates, utilizing the inequalities $\|\nabla u_{m+1}^{N}\|^2 \le \frac{1}{2\lambda_{1}} \|A u_{m+1}^{N}\|^2 + \frac{\lambda_{1}}{2} \|u_{m+1}^{N}\|^2$ and $-\|A u_{m+1}^{N}\|^2 \le -\lambda_1^2 \|u_{m+1}^{N}\|^2+\lambda_1^2\left \langle u_0, \phi _{0}\right \rangle^2$ yields
 	\begin{equation}\label{case p=2}
 		\begin{aligned}
 			&\frac{1}{2}\big(\|u_{m+1}^{N}\|^2 - \|u_{m}^{N}\|^2\big)+\lambda_{1}(\lambda_{1}-1)\tau\|u_{m+1}^{N}\|^2-\lambda_1^2\tau \left \langle u_0, \phi _{0}\right \rangle^2 \\
 			\le&\frac{1}{2}\|G(u_{m}^{N})\triangle_m W\|^2+\langle G(u_{m}^{N})\triangle_m W, u_{m}^{N} \rangle 
 			=:I_1+I_2.
 		\end{aligned}
 	\end{equation}
 	Taking expectations on both sides of \eqref{case p=2} and using \eqref{linear growth in H} in Assumption \ref{diffusion coeffient} yields
 	\begin{equation*}
 		\begin{aligned}
 			\Big( \frac{1}{2}+\lambda_{1}(\lambda_{1}-1)\tau \Big )\mathbb{E}[\|u_{m+1}^{N}\|^2] 
 			\le \Big(\frac{1}{2}+\frac{1}{2}L_{3}\tau\Big )\mathbb{E}[\|u_{m}^{N}\|^2] +\lambda_1^2\tau \mathbb{E}[\left\langle u_0, \phi _{0}\right\rangle^2]+\frac{1}{2}L_{4}\tau.
 		\end{aligned}
 	\end{equation*}
 	 Set $\kappa_{2}=\mathbb{E}[\left\langle u_0, \phi _{0}\right\rangle^{2}]$. By simple calculations, we have
 	\begin{equation*}
 		\begin{aligned}
 			\mathbb{E}[\|u_{m+1}^{N}\|^2]
 			\le& \left ( \frac{1+L_{3}\tau}{1+2\lambda_{1}(\lambda_{1}-1)\tau} \right )^{m+1}\mathbb{E}[\|u_{0}^{N}\|^2]+\left ( \frac{L_{4}\tau+2\kappa_{1}\lambda_{1}^2\tau}{1+2\lambda_{1}(\lambda_{1}-1)\tau} \right )\sum_{i=0}^{m}\left ( \frac{1+L_{3}\tau}{1+2\lambda_{1}(\lambda_{1}-1)\tau} \right )^{i}.
 		\end{aligned}
 	\end{equation*}
 	It follows from \eqref{dissipative for H} in Assumption \ref{dissipative condition} that $\frac{1+L_{3}\tau}{1+2\lambda_{1}(\lambda_{1}-1)\tau} <1$. Applying the inequality 
 	$a^{m} \le e^{-(1-a)m}$  for all  $a \in (0,1)$ and $m \ge 0$, we obtain
 	\begin{equation*}
 		\mathbb{E}[\|u_{m+1}^{N}\|^2] \le e^{-\frac{2\lambda_{1}(\lambda_{1}-1)-L_{3} }{1+2\lambda_{1}(\lambda_{1}-1)\tau}t_{m+1}}\mathbb{E}\big [ \left \| u_{0}^{N}\right \|^2  \big ]+\frac{L_{4}+2\kappa_{2}\lambda_{1}^2}{2\lambda_{1}(\lambda_{1}-1)-L_{3}}=:K_1  \le C.
 	\end{equation*}
 	This completes the proof of uniform moment boundedness for the case $p=2$. \\

 	\noindent\textbf{Step 2.} We show the case for $p=4$.  Multiplying the inequality \eqref{case p=2} by $\|u_{m+1}^{N}\|^2$, and applying the identity $(a-b)a=\frac{1}{2}\left ( a^2-b^2+(a-b)^2 \right )$ together with the Young inequality yields
 	\begin{equation*}
 		\begin{aligned}
 			&\frac{1}{4}\Big[\|u_{m+1}^{N}\|^4 - \|u_{m}^{N}\|^4+\left (\|u_{m+1}^{N}\|^2-\|u_{m}^{N}\|^2\right )^2 \Big]
 			+\lambda_{1}(\lambda_{1}-1)\tau\|u_{m+1}^{N}\|^4-\lambda_1^2\tau \left \langle u_0, \phi _{0}\right \rangle^2\|u_{m+1}^{N}\|^2\\
 			\le &\left ( I_1+I_2 \right )\left (\|u_{m+1}^{N}\|^2-\|u_{m}^{N}\|^2\right )+\left ( I_1+I_2 \right )\|u_{m}^{N}\|^2\\
 			\le &\left ( I_1+I_2 \right )^2+\frac{1}{4}\left (\|u_{m+1}^{N}\|^2-\|u_{m}^{N}\|^2\right )^2+\left ( I_1+I_2 \right )\|u_{m}^{N}\|^2.
 		\end{aligned}
 	\end{equation*}
 	After a straightforward simplification, we rewrite the above inequality as follows
 	\begin{equation}\label{case p=4}
 		\begin{aligned}
 			&\frac{1}{4}\big[\|u_{m+1}^{N}\|^4 - \|u_{m}^{N}\|^4 \big]
 			+\left ( \lambda_{1}(\lambda_{1}-1)\right )  \tau\|u_{m+1}^{N}\|^4-\lambda_{1}^2\tau \left \langle u_0, \phi _{0}\right \rangle^2\|u_{m+1}^{N}\|^2\\
 			\le &\left ( I_1+I_2 \right )^2+\left ( I_1+I_2 \right )\|u_{m}^{N}\|^2
 		\end{aligned}
 	\end{equation}
 	Using the Burkh\"{o}lder-Davis-Gundy inequality and \eqref{linear growth in H} in Assumption \ref{diffusion coeffient}, we deduce that
 	\begin{equation*}
 		\begin{aligned}
 			\mathbb{E}[I_1^2]\le& \frac{1}{4}c_4\tau ^{2} \mathbb{E}\big [ \| G(u_{m}^{N})\|_{\mathcal{L}_{2}^{0}}^{4}\big]
 			\le \frac{1}{2}c_4L_{3}^2\tau ^{2}\mathbb{E}\big [ \| u_{m}^{N}\|^{4}\big]+\frac{1}{2}c_4L_{4}^2\tau ^{2}\\
 			\le &18L_{3}^2\tau ^{2}\mathbb{E}\big [ \| u_{m}^{N}\|^{4}\big]+18L_{4}^2\tau ^{2},
 		\end{aligned}
 	\end{equation*}
 	where $c_{4}$ denotes the constant in the Burkh\"{o}lder-Davis-Gundy inequality and $c_{p}=\big(\frac{p(p-1)}{2}\big)^{\frac{p}{2}}$.
 	By the It\^{o}-isometry formula and \eqref{linear growth in H} in Assumption \ref{diffusion coeffient}, it can be verified that
 	\begin{equation*}
 		\begin{aligned}
 			\mathbb{E}[I_2^2]=&\mathbb{E}\Big[ \Big(\int_{t_m}^{t_{m+1}}\left \langle u_{m}^{N}, G(u_{m}^{N})dW(s)\right\rangle\Big)^2 \Big]\\
 			\le& \mathbb{E}\Big[ \int_{t_m}^{t_{m+1}}\| u_{m}^{N}\|^{2}\| G(u_{m}^{N})\|_{\mathcal{L}_{2}^{0}}^{2}ds\Big]\\
 			\le& (L_{3}+\varepsilon) \tau\mathbb{E}\big [ \| u_{m}^{N}\|^{4}\big]+\frac{1}{4\varepsilon} L_{4}^2\tau.
 		\end{aligned}
 	\end{equation*}
 	Similarly, by the It\^{o}-isometry formula and \eqref{linear growth in H} in Assumption \ref{diffusion coeffient},
 	we have
 	\begin{equation*}
 		\begin{aligned}
 			\mathbb{E}[I_1\|u_{m}^{N}\|^{2}]=&\frac{1}{2}\mathbb{E}\Big[\Big\| \int_{t_m}^{t_{m+1}}\|u_{m}^{N}\|G(u_{m}^{N})dW(s)\Big\|^2 \Big]\\
 			=&\frac{1}{2}\tau \mathbb{E}\Big[\|u_{m}^{N}\|^2\|G(u_{m}^{N})\|_{\mathcal{L}_{2}^{0}}^{2}\Big]\\
 			\le &\Big(\frac{1}{2}L_{3}+\frac{1}{4}\varepsilon\Big)\tau\mathbb{E}\big [ \| u_{m}^{N}\|^{4}\big]+\frac{1}{4\varepsilon}L_{4}^2\tau.
 		\end{aligned}
 	\end{equation*}
 	Taking expectations on both sides of inequality \eqref{case p=4} and noting that the term $I_2\|u_{m}^{N}\|^2$ vanishes upon taking expectation, we substitute the above inequalities into \eqref{case p=4} and conclude that
 	\begin{equation}
 		\begin{aligned}
			\Big ( \frac{1}{4}+\left ( \lambda_{1}(\lambda_{1}-1)-\varepsilon \right )\tau \Big ) \mathbb{E}\big[\|u_{m+1}^{N}\|^4 \big]
			\le& \Big ( \frac{1}{4}+\Big( \frac{5}{2}L_{3}+\frac{9}{4}\varepsilon+36L_{3}^2\tau\Big)\tau\Big ) \mathbb{E}\big [ \| u_{m}^{N}\|^{4}\big]\\
			&+\frac{1}{4\varepsilon}\lambda _{1}^{4}\tau\mathbb{E}[\left\langle u_0, \phi _{0}\right\rangle^{4}]+\frac{3}{4\varepsilon}L_{4}^2\tau+36L_{4}^2\tau ^{2}.
 		\end{aligned}
 	\end{equation}
 	Through \eqref{dissipative for H} in Assumption \ref{dissipative condition}, we can know that $\varepsilon:=\frac{2\lambda_{1}(\lambda_{1}-1)-5L_{3}}{10}>0$. 
 	Choose $\tau _{0}:=\frac{2\lambda_{1}(\lambda_1-1)-5L_3 }{360L_3^2}$. Then, for any $0<\tau <\tau_0$, it follows that $36L_{3}^2\tau < \varepsilon$. Let $\kappa_{4}=\mathbb{E}[\langle u_0, \phi _{0}\rangle^{4}]$. By following similar arguments to those in Step 1 and employing a numerical iteration, we obtain
 	\begin{equation*}
 		\begin{aligned}
    		\mathbb{E}[\|u_{m+1}^{N}\|^4] \le e^{-\gamma_{1}t_{m+1}}\mathbb{E}\big [ \| u_{0}^{N}\|^4  \big ]+ \frac{\varepsilon^{-1}\lambda_{1}^{4}\kappa_4+3\varepsilon^{-1}L_{4}^2+144L_{4}^2\tau}{4\lambda_{1}(\lambda_{1}-1)-10L_{3}-17\varepsilon}\le C,
 		\end{aligned}
 	\end{equation*}
 	where $\gamma_{1}=\frac{4\lambda_{1}(\lambda_{1}-1)-10L_{3}-17\varepsilon}{1+4(\lambda_{1}(\lambda_{1}-1)-\varepsilon)\tau }>0$.

 	\noindent\textbf{Step 3.} For $p=2^k$ with  $k \ge 3$, assuming that the following inequality holds 
 	\begin{equation}\label{case p=2^k}
 		\begin{aligned}
 			&\frac{1}{p}\big[\|u_{m+1}^{N}\|^p - \|u_{m}^{N}\|^p \big]+\left ( \lambda_{1}(\lambda_{1}-1)\right )  \tau\|u_{m+1}^{N}\|^p
 			-\lambda_{1}^2\tau \left \langle u_0, \phi _{0}\right \rangle^2\|u_{m+1}^{N}\|^{p-2}\\
 			\le &\left ( I_1+I_2 \right )\|u_{m}^{N}\|^{p-2}+(p-3)\left ( I_1+I_2 \right )^2\|u_{m}^{N}\|^{p-4}
 			+\sum_{i=1}^{k-2}\mathcal{R}_{i}^{k}(I_1+I_2)^{2^{i+1} }\|u_{m}^{N}\|^{p-2^{i+2}},
 		\end{aligned}
 	\end{equation}
 	where 
 	\begin{equation*}
 		\mathcal{R}_{i}^{k}=
 		\left\{\begin{aligned}
 			&\mathcal{R}_1^{k-1}+2^{k-1}(k-2)(2^{k-1}-3)^2,  &&i=1, \\
 			&2^{k-1}(k-2)(\mathcal{R}_{i-1}^{k-1})^2+\mathcal{R}_{i}^{k-1},  &&2\le i \le k-3,\\
 			&2^{k-1}(k-2)(\mathcal{R}_{k-3}^{k-1})^2, &&i=k-2.
 		\end{aligned}\right.
 	\end{equation*}
 	It can be verified that  when $k=3$, $\mathcal{R}_1^{2}=0$ and $\mathcal{R}_1^{3}=4$.  
 	We aim to prove that the same inequality remains valid for the exponent $2p=2^{k+1}$.
 	Multiplying the inequality \eqref{case p=2^k} by $\|u_{m+1}^{N}\|^p$ and applying the identity $(a-b)a=\frac{1}{2}\left ( a^2-b^2+(a-b)^2 \right )$ together with the Young inequality yields
 		\begin{align*}
 			&\frac{1}{2p}\big[\|u_{m+1}^{N}\|^{2p} - \|u_{m}^{N}\|^{2p}+\left(\|u_{m+1}^{N}\|^{p}-\|u_{m}^{N}\|^{p}\right)^2\big]
 			+\left ( \lambda_{1}(\lambda_{1}-1)\right )  \tau\|u_{m+1}^{N}\|^{2p}
 			-\lambda_{1}^2\tau \left \langle u_0, \phi _{0}\right \rangle^2\|u_{m+1}^{N}\|^{2p-2}\\
 			\le& \left (I_1+I_2\right )\|u_{m}^{N}\|^{p-2}\left(\|u_{m+1}^{N}\|^{p}-\|u_{m}^{N}\|^{p}\right)+\left ( I_1+I_2 \right )\|u_{m}^{N}\|^{2p-2}\\
 			&+(p-3)\left (I_1+I_2\right )^2\|u_{m}^{N}\|^{p-4}\left(\|u_{m+1}^{N}\|^{p}-\|u_{m}^{N}\|^{p}\right)+(p-3)\left (I_1+I_2\right )^2\|u_{m}^{N}\|^{2p-4}\\
 			&+\Big( \sum_{i=1}^{k-2}\mathcal{R}_{i}^k(I_1+I_2)^{2^{i+1} }\|u_{m}^{N}\|^{p-2^{i+2}}\Big)\left(\|u_{m+1}^{N}\|^{p}-\|u_{m}^{N}\|^{p}\right)+\Big( \sum_{i=1}^{k-2}\mathcal{R}_{i}^k(I_1+I_2)^{2^{i+1} }\|u_{m}^{N}\|^{p-2^{i+2}}\Big)\|u_{m}^{N}\|^{p} \\
 			\le& \frac{1}{4p}\left(\|u_{m+1}^{N}\|^{p}-\|u_{m}^{N}\|^{p}\right)^2+p\left (I_1+I_2\right )^2\|u_{m}^{N}\|^{2p-4}+\left ( I_1+I_2 \right )\|u_{m}^{N}\|^{2p-2}\\
 			&+\frac{1}{4p(k-1)}\left(\|u_{m+1}^{N}\|^{p}-\|u_{m}^{N}\|^{p}\right)^2+p(k-1)(p-3)^2\left (I_1+I_2\right )^4\|u_{m}^{N}\|^{2p-8}\\
 			&+(p-3)\left (I_1+I_2\right )^2\|u_{m}^{N}\|^{2p-4}
 			+\frac{k-2}{4p(k-1)}\left(\|u_{m+1}^{N}\|^{p}-\|u_{m}^{N}\|^{p}\right)^2\\
 			&+\frac{p(k-1)}{k-2}\Big( \sum_{i=1}^{k-2}\mathcal{R}_{i}^k(I_1+I_2)^{2^{i+1} }\|u_{m}^{N}\|^{p-2^{i+2}}\Big)^2+\Big( \sum_{i=1}^{k-2}\mathcal{R}_{i}^k(I_1+I_2)^{2^{i+1} }\|u_{m}^{N}\|^{p-2^{i+2}}\Big)\|u_{m}^{N}\|^{p}
 		\end{align*}
 	After a straightforward simplification, we rewrite the above inequality as follows
 	\begin{equation}\label{case p=2^k+1}
 		\begin{aligned}
 			&\frac{1}{2p}\big[\|u_{m+1}^{N}\|^{2p} - \|u_{m}^{N}\|^{2p}\big]
 			+\left ( \lambda_{1}(\lambda_{1}-1)\right )  \tau\|u_{m+1}^{N}\|^{2p}
 			-\lambda_{1}^2\tau \left \langle u_0, \phi _{0}\right \rangle^2\|u_{m+1}^{N}\|^{2p-2}\\
 			\le &\left ( I_1+I_2 \right )\|u_{m}^{N}\|^{2p-2}+(2p-3)\left (I_1+I_2\right )^2\|u_{m}^{N}\|^{2p-4}+\sum_{i=1}^{k-1}\mathcal{R}_{i}^{k+1}(I_1+I_2)^{2^{i+1} }\|u_{m}^{N}\|^{2p-2^{i+2}}.
 		\end{aligned}
 	\end{equation}
 	Assuming that the inequality \eqref{case p=2^k} holds for the exponent $p$, the above argument yields inequality \eqref{case p=2^k+1} for the exponent $2p$, thereby completing the induction step. We next use the inequality \eqref{case p=2^k+1} as the starting point to establish the uniform-in-time bounds for $\|u_{m+1}^{N}\|^{2p}$. Noting that $\mathbb{E}\left[I_1\right]\sim \tau$ and $\mathbb{E}\left[I_2^2\right]\sim \tau$, we concentrate on the dominant terms, namely $I_1\|u_{m}^{N}\|^{2p-2}$ and $I_2^2\|u_{m}^{N}\|^{2p-4}$. 
 	By the It\^{o}-isometry formula, \eqref{linear growth in H} in Assumption \ref{diffusion coeffient} and the Young inequality, it can be verified that
 	\begin{equation*}
 		\begin{aligned}
 			\mathbb{E}[I_1\|u_{m}^{N}\|^{2p-2}]=&\frac{1}{2}\mathbb{E}\Big[\Big\| \int_{t_m}^{t_{m+1}}\|u_{m}^{N}\|^{p-1}G(u_{m}^{N})dW(s)\Big\|^2 \Big]
 			=\frac{1}{2}\tau \mathbb{E}\Big[\|u_{m}^{N}\|^{2p-2}\|G(u_{m}^{N})\|_{\mathcal{L}_{2}^{0}}^{2}\Big]\\
 			\le &\big( \frac{1}{2}L_{3}+\varepsilon \big) \tau\mathbb{E}\big [ \| u_{m}^{N}\|^{2p}\big]
 			+p^{-1}\varepsilon^{1-p}L_{4}^{p}\tau.
 		\end{aligned}
 	\end{equation*}
 	Since $I_2\|u_{m}^{N}\|^{2p-2}$ defines a real-valued martingale, the expectation of this term vanishes. Similarly, by the It\^{o}-isometry formula and \eqref{linear growth in H} in Assumption \ref{diffusion coeffient},
 	we have
 	\begin{equation*}
 		\begin{aligned}
 			\mathbb{E}[I_2^2\|u_{m}^{N}\|^{2p-4}]=&\mathbb{E}\Big[\Big( \int_{t_m}^{t_{m+1}}\left \langle \|u_{m}^{N}\|^{p-2}u_{m}^{N}, G(u_{m}^{N})dW(s)\right\rangle \Big)^2 \Big]\\
 			\le& \mathbb{E}\Big[ \int_{t_m}^{t_{m+1}}\| u_{m}^{N}\|^{2p-2}\| G(u_{m}^{N})\|_{\mathcal{L}_{2}^{0}}^{2}ds\Big]\\
 			\le& \big(L_{3}+\varepsilon \big) \tau\mathbb{E}\big [ \| u_{m}^{N}\|^{2p}\big]
 			+p^{-1}\varepsilon^{1-p}L_{4}^{p}\tau.
 		\end{aligned}
 	\end{equation*}
 	As $\mathbb{E}\big [I_{1}^{2^{i+1}}\big]\sim \tau^{2^{i+1}}$ and $\mathbb{E}\big [I_{2}^{2^{i+1}}\big]\sim \tau^{2^{i}}$, we can obtain the following inequality  
 	\begin{equation*}
 		\begin{aligned}
 			(4p-6)\mathbb{E}\left[I_{1}^{2}\|u_{m}^{N}\|^{2p-4}\right ]  
 			+\mathbb{E}\Big[\sum_{i=1}^{k-1}\mathcal{R}_{i}^{k+1}(I_1+I_2)^{2^{i+1} }\|u_{m}^{N}\|^{2p-2^{i+2}}\Big]
 			\le C(p,L_3)\tau ^{2}\mathbb{E}\left[\|u_{m}^{N}\|^{2p}\right ]+C(p,L_4)\tau ^{2}.
 		\end{aligned}
 	\end{equation*}
 	By substituting the above inequalities into \eqref{case p=2^k+1}, we conclude that
 	\begin{equation}
 		\begin{aligned}
 			&\Big ( \frac{1}{2p}+\bigl ( \lambda_{1}(\lambda_{1}-1)-\varepsilon \bigl)\tau \Big ) \mathbb{E}\big[\|u_{m+1}^{N}\|^{2p} \big]\\
 			\le &\Big ( \frac{1}{2p}+\Big( (4p-\frac{11}{2})L_{3}+(4p-5)\varepsilon +C(p,L_3)\tau\Big)\tau \Big ) \mathbb{E}\big [ \| u_{m}^{N}\|^{2p}\big] \\
 			&+(4p-5)p^{-1}\varepsilon^{1-p}L_{4}^{p}\tau
 			+p^{-1}\varepsilon^{1-p}\lambda_{1}^{2p}\mathbb{E}[\left \langle u_0,\phi_0\right \rangle^{2p}]\tau+C(p,L_4)\tau^2.
 		\end{aligned}
 	\end{equation}
 	From \eqref{dissipative for H} in Assumption \ref{dissipative condition}, we know that $\varepsilon :=\frac{2\lambda_{1}(\lambda_{1}-1)-(8p-11)L_{3} }{8p}>0$.
 	Set $\tau _{0}:=\frac{2\lambda_{1}(\lambda_1-1)-(8p-11)L_3 }{8pC(p,L_3)}$. Then, for any $0<\tau<\tau_0$, it follows that $C(p,L_3)\tau<\varepsilon$. Let $\kappa_{2p}=\mathbb{E}[\langle u_0, \phi _{0}\rangle^{2p}]$. Similarly, we can apply a numerical iteration to deduce that 
 	\begin{equation*}
 		\begin{aligned}
 			\mathbb{E}[\|u_{m+1}^{N}\|^{2p}] \le e^{-\gamma_{2}t_{m+1}}\mathbb{E}\big [ \| u_{0}^{N}\|^{2p} \big ]+ \frac{(8p-10)p^{-1}\varepsilon^{1-p}L_{4}^{p}+2p^{-1}\varepsilon^{1-p}\lambda _{1}^{2p}\kappa_{2p}+C(p,L_4)\tau}{2\lambda_{1}(\lambda_{1}-1)-(8p-11)L_{3}-(8p-6)\varepsilon} \le C,
 		\end{aligned}
 	\end{equation*}
 	where $\gamma_{2}=\frac{2\lambda_{1}(\lambda_{1}-1)-(8p-11)L_{3}-(8p-6)\varepsilon}{\frac{1}{p}+2(\lambda_{1}(\lambda_{1}-1)-\varepsilon)\tau} >0$.
 	Thus, the proof is complete. 
 	
 \end{proof}

 Analogous to the approach used for the continuous equation, we leverage the energy functional analysis method to establish the uniform moment boundedness of the mild solution $u_{m}^{N}$ in $H^1$. Recall the energy functional $\mathcal{E}(u_{m}^{N})$ defined as 
 \begin{equation}\label{full energy func}
 	\begin{aligned}
 		\mathcal{E}(u_{m}^{N})=\frac{1}{2}\|\nabla u_{m}^{N} \|^2 +\frac{1}{4}\left\|(u_{m}^{N})^2-1 \right\|^2.
 	\end{aligned}
 \end{equation}
 
 \begin{lemma}\label{uniform estimate of u_m^N in H^1}
 	Let $2 \le p \le 32$ and assume that $u_{0}\in L^{64}(\Omega;H^1)$. Suppose that Assumptions \ref{diffusion coeffient} and \ref{dissipative condition} hold. Define $\tau_1^*:=\displaystyle\inf_{2\le p\le32}\tau_1(p,C_{d,4}, L_5,L_7)$. Then, for every $0<\tau<\tau_1^*$, it holds that
 	\begin{equation}
 		\sup_{m\ge0}\mathbb{E}\big [ \mathcal{E}^{\frac{p}{2}}(u_{m+1}^{N}) \big ]\le C .
 	\end{equation}
 \end{lemma}
 \begin{proof}Similar to the proof in Lemma \ref{p-moment uniform in H}, the proof of the uniform boundedness of $\mathbb{E}\bigl[\mathcal{E}^{\frac{p}{2}}(u_{m+1}^{N})\bigr]$ is divided into three steps.\\
 	\textbf{Step 1.} We begin with the case $p=2$.
 	Taking the inner product in $H$ on both sides of \eqref{full-discrete scheme} with $Au_{m+1}^{N}+F(u_{m+1}^{N})$ yields
 	\begin{equation}\label{full energy}
 		\begin{aligned}
 			&\left\langle u_{m+1}^{N}-u_{m}^{N}, Au_{m+1}^{N}+F(u_{m+1}^{N})\right\rangle +\tau \big \|A^{\frac{1}{2}}(Au_{m+1}^{N}+F(u_{m+1}^{N}))\big\|^2\\
 			=&\left\langle P_NG(u_{m}^{N})\Delta_m W, Au_{m+1}^{N}+F(u_{m+1}^{N})-Au_{m}^{N}-F(u_{m}^{N})\right\rangle
 			+\left\langle P_NG(u_{m}^{N})\Delta_m W, Au_{m}^{N}+F(u_{m}^{N})\right\rangle.
 		\end{aligned}
 	\end{equation}
  Using the identity $(a-b)a=\frac{1}{2}\left ( a^2-b^2+(a-b)^2 \right )$ and the integration by parts formula, we have
  \begin{equation*}
  	\begin{aligned}
  		\left\langle u_{m+1}^{N}-u_{m}^{N}, Au_{m+1}^{N}\right\rangle =&\left\langle \nabla (u_{m+1}^{N}-u_{m}^{N}),\nabla u_{m+1}^{N}\right\rangle\\
  		=&\frac{1}{2}\left(\|\nabla u_{m+1}^{N}\|^2-\|\nabla u_{m}^{N}\|^2+\|\nabla (u_{m+1}^{N}-u_{m}^{N})\|^2 \right),\\
  		\left\langle u_{m+1}^{N}-u_{m}^{N}, F(u_{m+1}^{N})\right\rangle=&\Big\langle u_{m+1}^{N}-u_{m}^{N},\frac{1}{2}\big ( (u_{m+1}^{N})^{2}-1\big )\left(u_{m+1}^{N}+u_{m}^{N}+u_{m+1}^{N}-u_{m}^{N}\right)\Big\rangle\\
  		=&\frac{1}{2}\Big\langle (u_{m+1}^{N})^2-(u_{m}^{N})^2,(u_{m+1}^{N})^{2}-1\Big\rangle 
  		+\frac{1}{2}\Big\langle (u_{m+1}^{N}-u_{m}^{N})^2,(u_{m+1}^{N})^{2}-1\Big\rangle \\
  		=&\frac{1}{4}\left(\left\|(u_{m+1}^{N})^{2}-1\right\|^2-\left\|(u_{m}^{N})^{2}-1\right\|^2+\left\|(u_{m+1}^{N})^{2}-(u_{m}^{N})^{2}\right\|^2\right)\\
  		&+\frac{1}{2}\left (\left \|u_{m+1}^{N}(u_{m+1}^{N}-u_{m}^{N})\right\|^2-\left\|u_{m+1}^{N}-u_{m}^{N}\right\|^2\right ).
  	\end{aligned}
  \end{equation*}
 	By the Cauchy-Schwarz inequality, the Young inequality, and the integration by parts formula, one can easily deduce that
 		\begin{align*}
 			&\left\langle G(u_{m}^{N})\Delta_m W, Au_{m+1}^{N}-Au_{m}^{N}\right\rangle+\left\langle P_NG(u_{m}^{N})\Delta_m W, F(u_{m+1}^{N})-F(u_{m}^{N})\right\rangle \\
 			=&\left\langle \nabla G(u_{m}^{N})\Delta_m W,\nabla(u_{m+1}^{N}-u_{m}^{N}) \right \rangle+\left\langle P_NG(u_{m}^{N})\Delta_m W, u_{m}^{N}\big((u_{m+1}^{N})^{2}-(u_{m}^{N})^{2}\big)\right\rangle \\
 			&+\left\langle P_NG(u_{m}^{N})\Delta_m W, \big((u_{m+1}^{N})^{2}-1\big)(u_{m+1}^{N}-u_{m}^{N})\right\rangle\\
 			\le &\left\|\nabla G(u_{m}^{N})\Delta_m W\right\|^2+\frac{1}{4}\|\nabla (u_{m+1}^{N}-u_{m}^{N})\|^2+\left\|u_{m}^{N}P_NG(u_{m}^{N})\Delta_m W\right\|^2 \\
 			&+\frac{1}{4}\left\|(u_{m+1}^{N})^{2}-(u_{m}^{N})^{2}\right\|^2 +\frac{1}{2}\left\|u_{m+1}^{N}P_NG(u_{m}^{N})\Delta_m W\right\|^2
 			+\frac{1}{2}\left \|u_{m+1}^{N}(u_{m+1}^{N}-u_{m}^{N})\right\|^2 \\ &+\frac{1}{2}\left\|G(u_{m}^{N})\Delta_m W\right\|^2
 			+\frac{1}{2}\left \|u_{m+1}^{N}-u_{m}^{N}\right\|^2.
 		\end{align*}
 	Substituting the above inequalities into \eqref{full energy}, we obtain
 	\begin{equation}\label{middle full energy with p=1}
 		\begin{aligned}
 			&\mathcal{E}(u_{m+1}^{N})-\mathcal{E}(u_{m}^{N})+\frac{1}{4}\|\nabla (u_{m+1}^{N}-u_{m}^{N})\|^2+\tau \big \|A^{\frac{1}{2}}(Au_{m+1}^{N}+F(u_{m+1}^{N}))\big\|^2\\
 			\le &\left\|\nabla G(u_{m}^{N})\Delta_m W\right\|^2+\left\|u_{m}^{N}P_NG(u_{m}^{N})\Delta_m W\right\|^2
 			+\frac{1}{2}\left\|u_{m+1}^{N}P_NG(u_{m}^{N})\Delta_m W\right\|^2\\
 			&+\frac{1}{2}\left\|G(u_{m}^{N})\Delta_m W\right\|^2
 			+\left \|u_{m+1}^{N}-u_{m}^{N}\right\|^2+\left\langle P_NG(u_{m}^{N})\Delta_m W, Au_{m}^{N}+F(u_{m}^{N})\right\rangle .
 		\end{aligned}
 	\end{equation}
 	Analogous to the preceding method, by taking the inner product in $H$ on both sides of \eqref{full-discrete scheme} with $A^{-1}(u_{m+1}^{N}-u_{m}^{N})$, we can deduce 
 	\begin{equation}\label{middle prperty}
 		\begin{aligned}
 			\frac{5}{8}\big \| A^{-\frac{1}{2}}(u_{m+1}^{N}-u_{m}^{N})\big \|^2
 			\le -\tau (\mathcal{E}(u_{m+1}^{N})-\mathcal{E}(u_{m}^{N}))+\big \| A^{-\frac{1}{2}}G(u_{m}^{N})\Delta_{m}W \big \|^2.
 		\end{aligned}
 	\end{equation}
 	Utilizing the inequality \eqref{middle prperty} and the Young inequality, we obtain
 	\begin{equation*}
 		\begin{aligned}
 			\|u_{m+1}^{N}-u_{m}^{N}\|^2\le&\frac{1}{4}\|\nabla (u_{m+1}^{N}-u_{m}^{N})\|^2+\| A^{-\frac{1}{2}}(u_{m+1}^{N}-u_{m}^{N})\big \|^2\\
 			\le &\frac{1}{4}\|\nabla (u_{m+1}^{N}-u_{m}^{N})\|^2-\frac{8}{5}\tau(\mathcal{E}(u_{m+1}^{N})-\mathcal{E}(u_{m}^{N}))+\frac{8}{5}\big \| A^{-\frac{1}{2}}G(u_{m}^{N})\Delta_{m}W \big \|^2.
 		\end{aligned}
 	\end{equation*}
 	Substituting the above inequality into \eqref{middle full energy with p=1}, we can rewrite \eqref{middle full energy with p=1}  as follows
 	\begin{equation}\label{full energy with p=1}
 		\begin{aligned}
 			&\Big (1+\frac{8}{5}\tau \Big ) \big (\mathcal{E}(u_{m+1}^{N})-\mathcal{E}(u_{m}^{N})\big )+\tau \big \|A^{\frac{1}{2}}(Au_{m+1}^{N}+F(u_{m+1}^{N}))\big\|^2\\
 			\le &\left\|\nabla G(u_{m}^{N})\Delta_m W\right\|^2+\left\|u_{m}^{N}P_NG(u_{m}^{N})\Delta_m W\right\|^2
 			+\frac{1}{2}\left\|u_{m+1}^{N}P_NG(u_{m}^{N})\Delta_m W\right\|^2\\
 			&+\frac{1}{2}\left\|G(u_{m}^{N})\Delta_m W\right\|^2
 			+\frac{8}{5}\big \| A^{-\frac{1}{2}}G(u_{m}^{N})\Delta_{m}W \big \|^2+\left\langle P_NG(u_{m}^{N})\Delta_m W, Au_{m}^{N}+F(u_{m}^{N})\right\rangle \\
 			=:&J_1+J_2+J_3+J_4+J_5+J_6.
 		\end{aligned}
 	\end{equation}
 	Let $C_{d,4}>0$ be the Sobolev embedding constant from $H^{\frac{d}{4}}\hookrightarrow L^4$.
 	It follows from the It\^{o}-isometry formula, \eqref{linear growth in H^1}-\eqref{linear growth in L^4} in Assumption \ref{diffusion coeffient} and the H\"{o}lder inequality that
 		\begin{align}\label{J1-J6 estimate}
 			\mathbb{E}\left [ J_1\right ]=&\tau\mathbb{E}\big [\|G(u_{m}^{N})\|_{\mathcal{L}_{2}^{1}}^{2}\big ]
 			\le L_5\tau\mathbb{E}\left [\|\nabla u_{m}^{N}\|^{2}\right]+L_{5}\tau \mathbb{E}[ \|u_{0}\|^{2}]+L_6\tau ,\notag\\
 			\mathbb{E}\left [ J_2\right ]=&\mathbb{E}\Big[\Big\|\int_{t_m}^{t_{m+1}}u_{m}^{N}P_NG(u_{m}^{N})dW(s)\Big\|^2\Big]
 			=\tau\mathbb{E}\left[\left\|u_{m}^{N}P_NG(u_{m}^{N})\right\|_{\mathcal{L}_{2}^{0}}^{2}\right ]\notag\\
 			\le& C_{d,4}^2\tau\mathbb{E}\Big[\left\|u_{m}^{N}\right\|_{L^4}^{2}\sum_{k=1}^{\infty}\|G(u_{m}^{N})Q^{\frac{1}{2}}\phi_{k}\|_{\dot{H}^{\frac{d}{4}}}^{2} \Big]\\ \notag
 			\le& (C_{d,4}^2L_7+\varepsilon) \tau\mathbb{E}\big[\left\|u_{m}^{N}\right\|_{L^4}^{4}\big]+\frac{1}{4\varepsilon}C_{d,4}^4L_{8}^{2}\tau,  \\
 			\mathbb{E}\left [ J_4+J_5\right ]=&\frac{1}{2}\tau\mathbb{E}\big[\|G(u_{m}^{N})\|_{\mathcal{L}_{2}^{0}}^{2}\big]+\frac{8}{5}\tau\mathbb{E}\big[\|G(u_{m}^{N})\|_{\mathcal{L}_{2}^{-1}}^{2}\big ]\notag\\
 			\le &\frac{21}{10}\tau\mathbb{E}\big[\|G(u_{m}^{N})\|_{\mathcal{L}_{2}^{0}}^{2}\big]
 			\le\frac{21}{10}L_{3}\tau\mathbb{E}[\|u_{m}^{N}\|^{2}]+\frac{21}{10}L_{4}\tau\notag.
 		\end{align}
 	Let $Z$ denote the standard Gaussian random variable. Define $\mathfrak{m}_{r}:= \mathbb{E}|Z|^r=\frac{2^{\frac{r}{2}}\Gamma(\frac{r+1}{2})}{\sqrt{\pi } }$.
 	Since $u_m^N$ is $\mathcal F_{t_m}$-measurable and $\Delta_mW$
 	is independent of $\mathcal F_{t_m}$, using the Young inequality,
 	the tower property, the conditional fourth-moment formula for
 	Gaussian random variables, and the Sobolev embedding
 	$H^{\frac{d}{4}}\hookrightarrow L^4$, we obtain
 	
 	\begin{equation}\label{J_3 estimate}
 		\begin{aligned}
 			\mathbb{E}\left [ J_3\right ]
 			\le&\frac{1}{4}\tau\mathbb{E}[\|u_{m+1}^{N}\|_{L^4}^{4}]
 			+\frac{1}{4\tau}\mathbb{E}[\mathbb{E}[\|P_NG(u_{m}^{N})\Delta_m W\|_{L^4}^{4}|\mathcal{F}_{t_m} ]]\\
 			=&\frac{1}{4}\tau\mathbb{E}[\|u_{m+1}^{N}\|_{L^4}^{4}]+\frac{\mathfrak{m}_{4}}{4}\tau\mathbb{E}\Big[\int _{\mathcal{D}}\Big ( \sum_{k=1}^{\infty} \big |P_NG(u_m^N)Q^{\frac{1}{2}}\phi_k(x)\big |^2\Big )^2dx \Big]\\
 			\le&\frac{1}{4}\tau\mathbb{E}[\|u_{m+1}^{N}\|_{L^4}^{4}]+\frac{\mathfrak{m}_{4}}{4}C_{d,4}^4\tau\mathbb{E}\Big[\Big ( \sum_{k=1}^{\infty} \|P_NG(u_m^N)Q^{\frac{1}{2}}\phi_k\|_{\dot{H}^{\frac{d}{4}}}^2\Big )^2\Big]\\
 			\le&\frac{1}{4}\tau\mathbb{E}[\|u_{m+1}^{N}\|_{L^4}^{4}]+\frac{\mathfrak{m}_{4}}{2}C_{d,4}^4 L_7^2\tau\mathbb{E}[\|u_m^N\|_{L^4}^4]+\frac{\mathfrak{m}_{4}}{2}C_{d,4}^4 L_8^2\tau
 		\end{aligned}
 	\end{equation}
 	Using the estimate in \eqref{estimate of derivative Eu}, we have
 	\begin{equation}\label{full discrete for energy basis}
 		\big\|A^{\frac{1}{2}}(Au_{m}^{N}+F(u_{m}^{N}))\big\|^2 \ge 2\lambda_{1}^{2}\mathcal{E}(u_{m+1}^{N})-\frac{27}{16}\lambda_{1}^{2}\|u_0\| ^4-\lambda_{1}^{2}.
 	\end{equation}
 	By the definition of $\mathcal{E}(u_{m}^{N})$ in \eqref{full energy func}, we have $\|u_m^N\|_{L^4}^4 \le 4\mathcal{E}(u_{m}^{N})+2\|u_m^N\|^2$.
 	Taking expectations, combining the inequalities \eqref{J1-J6 estimate}, \eqref{J_3 estimate} and \eqref{full discrete for energy basis} into \eqref{full energy with p=1} gives
 	\begin{equation*}
 		\begin{aligned}
 			&\Big (1+\frac{8}{5}\tau \Big ) \mathbb{E}\big[ \big (\mathcal{E}(u_{m+1}^{N})-\mathcal{E}(u_{m}^{N})\big )\big]
 			+\big(2\lambda_{1}^{2}-1\big)\tau\mathbb{E}\big[\mathcal{E}(u_{m+1}^{N})\big]\\
 			\le &\big(2L_5+4C_{d,4}^2L_7+2\mathfrak{m}_{4}C_{d,4}^4L_7^2+4\varepsilon\big) \tau\mathbb{E}\big[\mathcal{E}(u_{m}^{N})\big] +\big(\frac{21}{10}L_{3}+2C_{d,4}^2L_7+\mathfrak{m}_{4}C_{d,4}^4L_7^2+2\varepsilon\big) \tau\mathbb{E}[\|u_{m}^{N}\|^{2}]\\
 			&+\frac{1}{2}\tau\mathbb{E}[\|u_{m+1}^{N}\|^{2}]+\frac{27}{16}\lambda_{1}^{2}\tau \mathbb{E}\big[\|u_0\|^4\big ]
 			+L_{5}\tau \mathbb{E}[ \|u_{0}\|^{2}]+\big(\frac{1}{4\varepsilon}+\frac{\mathfrak{m}_{4}}{2}\big)C_{d,4}^4L_{8}^{2}\tau +\big( \frac{21}{10}L_{4}+L_6+\lambda_{1}^{2} \big)\tau.
 		\end{aligned}
 	\end{equation*}
 	Through \eqref{dissipative for H^1} in Assumption \ref{dissipative condition}, we have $\varepsilon:=\frac{2\lambda_{1}^{2}-1-2L_{5}-4C_{d,4}^2L_{7}-2\mathfrak{m}_{4}C_{d,4}^4L_{7}^{2}}{5}>0$. Then, we can deduce that
 	\begin{equation*}
 		\frac{1+\frac{8}{5}\tau+\big(2L_5+4C_{d,4}^2L_7+2\mathfrak{m}_{4}C_{d,4}^4L_{7}^{2}+4\varepsilon \big)\tau }{1+\frac{8}{5}\tau +(2\lambda_{1}^{2}-1)\tau }<1 .
 	\end{equation*}
 	Set $\kappa_{4}=\mathbb{E}[\|u_0\|^{4}]$ and $\kappa_{2}=\mathbb{E}[\|u_0\|^{2}]$. Given that $\|u_{m+1}^{N}\|_{L^{p}(\Omega;H)}\le C$ has been established in Lemma \ref{p-moment uniform in H}, it follows that there exists a constant $K_{2}>0$ such that $\mathbb{E}[\|u_{m}^{N}\|^{2}]\le K_{2}$.  Then, by applying numerical iteration, we have
    \begin{equation*}
	   \begin{aligned}
		\mathbb{E}[\mathcal{E}(u_{m}^{N})] \le e^{-\beta_{1}t_{m+1}}\mathbb{E}\big [ \mathcal{E}(u_{0}^{N})\big]
		+\frac{C(\varepsilon,L_3,L_7,C_{d,4})K_2+\frac{27}{16}\lambda_{1}^{2}\kappa_{4}+L_{5}\kappa_{2}+C(\varepsilon,L_4,L_6,L_8,C_{d,4})}{2\lambda_{1}^{2}-1-\big(2L_5+4L_7+2L_{7}^{2}+4\varepsilon\big)} \le C
     	\end{aligned}
     \end{equation*}
     where $\beta_{1}=\frac{2\lambda_{1}^{2}-1-\big(2L_5+4C_{d,4}^2L_7+6C_{d,4}^4L_{7}^{2}+4\varepsilon\big)}{1+\frac{8}{5}\tau +(2\lambda_{1}^{2}-1)\tau}>0$.

 	\noindent \textbf{Step 2.} We show the case for $p=4$. Multiplying the inequality \eqref{full energy with p=1} by $\mathcal{E}(u_{m+1}^{N})$ and applying the identity $(a-b)a=\frac{1}{2}\left ( a^2-b^2+(a-b)^2 \right )$ together with the Young inequality yields
 	\begin{equation}\label{full energy with p=2}
 		\begin{aligned}
 			&\frac{1}{2}\Big (1+\frac{8}{5}\tau \Big ) \big (\mathcal{E}^2(u_{m+1}^{N})-\mathcal{E}^2(u_{m}^{N})\big )
 			+\tau \big \|A^{\frac{1}{2}}D\mathcal{E}(u_{m+1}^{N})\big\|^2\mathcal{E}(u_{m+1}^{N})\\
 			\le& \left(J_1+J_2+J_4+J_5+J_6\right)\mathcal{E}(u_{m}^{N})+\frac{1}{2}\left(J_1+J_2+J_4+J_5+J_6\right)^2+J_3\mathcal{E}(u_{m+1}^{N}).
 		\end{aligned} 
 	\end{equation}
 	Noting that $\mathbb{E}\left[J_1+J_2+J_3+J_4+J_5\right]\sim \tau$ and $\mathbb{E}\left[J_6^2\right]\sim \tau$, we concentrate on the dominant terms, namely $\left(J_1+J_2+J_4+J_5\right)\mathcal{E}(u_{m}^{N})$, $J_6^2$ and $J_3\mathcal{E}(u_{m+1}^{N})$. 
 	Using the It\^{o}-isometry formula, \eqref{linear growth in H^1}-\eqref{linear growth in L^4} in Assumption \ref{diffusion coeffient}, the Sobolev embedding inequality $H^{\frac{d}{4}}\hookrightarrow L^4$ and the Young inequality, we have
 	\begin{equation*}
 		\begin{aligned}
 			\mathbb{E}\left[J_6^2\right]
 			\le &2\mathbb{E}\left[\|\nabla u_{m}^{N}\|^2\|\nabla G(u_{m}^{N})\Delta_{m}W\|^2\right]+2\mathbb{E}\left[\|u_{m}^{N}P_NG(u_{m}^{N})\Delta_{m}W\|^2\|(u_{m}^{N})^2-1\|^2\right] \\
 			\le &2\tau \mathbb{E}\left[\|\nabla u_{m}^{N}\|^2\| G(u_{m}^{N})\|_{\mathcal{L}_{2}^{1}}^2\right]
 			+2C_{d,4}^2\tau\mathbb{E}\Big[\|u_{m}^{N}\|_{L^4}^2\sum_{k=1}^{\infty}\|G(u_{m}^{N})Q^{\frac{1}{2}}\phi_{k}\|_{\dot{H}^{\frac{d}{4}}}^{2}\|(u_{m}^{N})^2-1\|^2\Big]\\
 			\le&2L_5\tau\mathbb{E}\left[\|\nabla u_{m}^{N}\|^4\right]+2L_5\tau\mathbb{E}\left[\|\nabla u_{m}^{N}\|^2\|u_{0}\|^2\right]+2L_6\tau\mathbb{E}\left[\|\nabla u_{m}^{N}\|^2\right]\\ &+2C_{d,4}^2L_7\tau\mathbb{E}\Big[\|u_{m}^{N}\|_{L^4}^4\|(u_{m}^{N})^2-1\|^2\Big]
 			+2C_{d,4}^2L_8\tau\mathbb{E}\Big[\|u_{m}^{N}\|_{L^4}^2\|(u_{m}^{N})^2-1\|^2\Big]\\
 			\le& \left(8L_5+\varepsilon \right )\tau\mathbb{E}\left[\mathcal{E}^{2}(u_{m}^{N})\right]+\Big ( 32C_{d,4}^2L_7+\varepsilon \Big)\tau \mathbb{E}\left [ \mathcal{E}^{2}(u_{m}^{N})\right ]+C(\varepsilon,C_{d,4},L_{7},L_{8})\tau \|u_{m}^{N}\|^4\\
 			&+C(\varepsilon,L_{5})\tau \|u_0\|^4+C(\varepsilon,C_{d,4},L_{6},L_{8})\tau\\
 			\le& \left(8L_5+32C_{d,4}^2L_7+2\varepsilon \right )\tau\mathbb{E}\left[\mathcal{E}^{2}(u_{m}^{N})\right]+C(\varepsilon,C_{d,4},L_{7},L_{8})\tau \|u_{m}^{N}\|^4\\
 			&+C(\varepsilon,L_{5})\tau \|u_0\|^4+C(\varepsilon,C_{d,4},L_{6},L_{8})\tau.
 		\end{aligned}
 	\end{equation*}
 	Similar to the estimates in \eqref{J1-J6 estimate}, by using \eqref{linear growth in H^1}-\eqref{linear growth in L^4} in Assumption \ref{diffusion coeffient}, the It\^{o}-isometry formula and H\"{o}lder inequality, we have
 		\begin{align}
 			\mathbb{E}\left [ J_1\mathcal{E}(u_{m}^{N})\right ]=&\tau\mathbb{E}\big [\|\nabla G(u_{m}^{N})\|_{\mathcal{L}_{2}^{0}}^{2}\mathcal{E}(u_{m}^{N})\big ]
 			\le (2L_5+\varepsilon)\tau\mathbb{E}\left [\mathcal{E}^2(u_{m}^{N})\right]+\frac{L_{5}^{2}}{2\varepsilon }\tau\|u_0\|^4+\frac{L_{6}^{2}}{2\varepsilon }\tau ,\notag\\
 			\mathbb{E}\left [ J_2\mathcal{E}(u_{m}^{N})\right ]=&\mathbb{E}\Big[\Big\|\int_{t_m}^{t_{m+1}}\sqrt{\mathcal{E}(u_{m}^{N})}u_{m}^{N}P_NG(u_{m}^{N})dW(s)\Big\|^2\Big]
 			=\tau\mathbb{E}\left[\left\|u_{m}^{N}P_NG(u_{m}^{N})\right\|_{\mathcal{L}_{2}^{0}}^{2}\mathcal{E}(u_{m}^{N})\right ]\notag\\
 			\le &C_{d,4}^2L_7\tau\mathbb{E}\big[\|u_{m}^{N}\|_{L^4} ^{4}\mathcal{E}(u_{m}^{N})\big]+C_{d,4}^2L_8\tau\mathbb{E}\big[\|u_{m}^{N}\|_{L^4}^{2}\mathcal{E}(u_{m}^{N})\big]\label{J1-J6 estimate p=2}\\  
 			\le &\big( 4C_{d,4}^2L_7+\varepsilon \big)\tau \mathbb{E}\big[\mathcal{E}^2(u_{m}^{N})\big]+C(\varepsilon,C_{d,4},L_{7},L_{8})\tau \|u_{m}^{N}\|^4+C(\varepsilon,C_{d,4},L_{8})\tau,  \notag\\
 			\mathbb{E}\left [ (J_4+J_5)\mathcal{E}(u_{m}^{N})\right ] \le &\frac{21}{10}\tau\mathbb{E}\big[\left\|G(u_{m}^{N})\right\|_{\mathcal{L}_{2}^{0}}^{2}\mathcal{E}(u_{m}^{N})\big ]\notag\\
 			\le & \frac{1}{2}\varepsilon\mathbb{E}\big[\mathcal{E}^2(u_{m}^{N})\big]
 			+C(\varepsilon,L_{3})\tau \|u_{m}^{N}\|^4+C(\varepsilon,L_{4})\tau.\notag
 		\end{align}
 		Since $u_m^N$ is $\mathcal F_{t_m}$-measurable and $\Delta_mW$
 		is independent of $\mathcal F_{t_m}$, using the Young inequality,
 		the tower property, and the Sobolev embedding
 		$H^{\frac{d}{4}}\hookrightarrow L^4$, we obtain
 		\begin{align}\label{J_3 estimate p=2}
 			\mathbb{E}\left [J_3\mathcal{E}(u_{m+1}^{N})\right ]
 			\le&\frac{1}{4}\tau\mathbb{E}[\|u_{m+1}^{N}\|_{L^4}^{4}\mathcal{E}(u_{m+1}^{N})]
 			+\frac{1}{4\tau}\mathbb{E}[\|P_NG(u_{m}^{N})\Delta_m W\|_{L^4}^{4}\mathcal{E}(u_{m+1}^{N})]\notag\\
 			\le &\big(\frac{3}{2}+\varepsilon\big)\tau\mathbb{E}\Big[\mathcal{E}^2(u_{m+1}^{N})\Big]+\frac{1}{16\varepsilon }\tau\mathbb{E}\left [ \|u_{m+1}^{N}\|^4 \right ] 
 			+\frac{1}{32\tau^3}\mathbb{E}[\mathbb{E}[\|P_NG(u_{m}^{N})\Delta_m W\|_{L^4}^{8}|\mathcal{F}_{t_m} ]]\notag \\
 			\le &\big(\frac{3}{2}+\varepsilon\big)\tau\mathbb{E}\Big[\mathcal{E}^2(u_{m+1}^{N})\Big]+\frac{1}{16\varepsilon }\tau\mathbb{E}\left [ \|u_{m+1}^{N}\|^4 \right ]
 			+\frac{\mathfrak{m}_{8}}{32}C_{d,4}^8\tau\mathbb{E}\Big[\Big( \sum_{k=1}^{\infty}\|G(u_{m}^{N})Q^{\frac{1}{2}}\phi_{k}\|_{\dot{H}^{\frac{d}{4}}}^{2}\Big)^4\Big]\notag\\
 			\le& \big(\frac{3}{2}+\varepsilon\big)\tau\mathbb{E}\Big[\mathcal{E}^2(u_{m+1}^{N})\Big]+8\mathfrak{m}_{8}C_{d,4}^8L_{7}^{4}\tau\mathbb{E}\Big[\mathcal{E}^2(u_{m}^{N})\Big]+\frac{1}{16\varepsilon }\tau\mathbb{E}\left [ \|u_{m+1}^{N}\|^4 \right ] \\
 			&+2\mathfrak{m}_{8}C_{d,4}^8L_{7}^{4}\tau\mathbb{E}\left [ \|u_{m}^{N}\|^4 \right ]+\frac{1}{4}\mathfrak{m}_{8}C_{d,4}^8L_{8}^{4}\tau. \notag
 		\end{align} 		
 	As $\mathbb{E}\big [J_{1}^{2}+J_{2}^{2}+J_{4}^{2}+J_{5}^{2}\big]\sim \tau^{2}$, we can obtain the following inequality
 	\begin{equation*}
 		\begin{aligned}
 			\frac{5}{2}\mathbb{E}\left [J_1^2+J_2^2+J_4^2+J_5^2\right ]\le C(p,C_{d,4},L_5,L_7)\tau^2\mathbb{E}\left [\mathcal{E}^2(u_{m}^{N})\right ]+C(p,C_{d,4},L_6,L_8)\tau^2.
 		\end{aligned}
 	\end{equation*}
 	Taking expectations on both sides of inequality \eqref{full energy with p=2} and noting that the term $J_6\mathcal{E}(u_{m}^{N})$ vanishes upon taking expectation, we substitute the above inequalities into \eqref{full energy with p=2} and conclude that
 	\begin{equation*}
 		\begin{aligned}
 			&\frac{1}{2}\Big (1+\frac{8}{5}\tau \Big )\mathbb{E}\left[\mathcal{E}^2(u_{m+1}^{N})-\mathcal{E}^2(u_{m}^{N})\right]
 			+(2\lambda_{1}^{2}-\frac{3}{2}-2\varepsilon)\tau\mathbb{E}\left[\mathcal{E}^2(u_{m+1}^{N})\right] \\ 
 			\le &\Big ( 22L_5+84C_{d,4}^2L_7+8L_{7}^{4}+8\mathfrak{m}_{8}C_{d,4}^8\varepsilon+C(p,C_{d,4},L_5,L_7)\tau\Big )\tau\mathbb{E}\big[\mathcal{E}^2(u_{m}^{N})\big]
 			+C(\varepsilon)\tau\mathbb{E}\left [\|u_{m+1}^{N}\|^{4}\right]\\
 			&+C(\varepsilon,L_{5})\tau\mathbb{E}[\|u_0\|^{8}]
 			+ C(\varepsilon,C_{d,4},L_{3},L_{7},L_{8})\tau\mathbb{E}[\|u_{m}^{N}\|^{4}]
 			+C(\varepsilon,C_{d,4},L_{4},L_{6},L_{8})(\tau+\tau^2). 
 		\end{aligned}
 	\end{equation*}
 	From \eqref{dissipative for H^1} in Assumption \ref{dissipative condition}, we know that $\varepsilon:=\frac{2\lambda_{1}^{2}-\frac{3}{2}-22L_5-84C_{d,4}^2L_7-8\mathfrak{m}_{8}C_{d,4}^8L_{7}^{4}}{12}>0$. Set $\tau_1:=\frac{2\lambda_{1}^{2}-\frac{3}{2}-22L_5-84C_{d,4}^2L_7-8\mathfrak{m}_{8}C_{d,4}^8L_{7}^{4}}{12C(p,C_{d,4},L_5,L_7)}$. Then, for any $0<\tau<\tau_1$, it follows that $C(p,C_{d,4},L_5,L_7)\tau <\varepsilon$.
    Set $\kappa_{8}=\mathbb{E}[\|u_0\|^{8}]$. By Lemma \ref{p-moment uniform in H}, it follows that there exists a constant $K_4>0$ such that $\mathbb{E}[\|u_{m+1}^{N}\|^{4}]\le K_{4}$. Similarly, employing a numerical iteration yields 
 	\begin{equation*}
 		\begin{aligned}
 			\mathbb{E}\left[\mathcal{E}^2(u_{m+1}^{N})\right]\le e^{-\beta_{2}t_{m+1}}\mathbb{E}\big [ \mathcal{E}^2(u_{0}^{N})\big]
 			+\frac{C(\varepsilon,C_{d,4},L_{3},L_{7},L_{8})K_4+C(\varepsilon,L_5)\kappa_{8}+C(\varepsilon,C_{d,4},L_{4},L_{6},L_{8}) }{2\lambda_{1}^{2}-\frac{3}{2}-22L_5-84L_7-8L_{7}^{4}-11\varepsilon} \le C,
 		\end{aligned}
 	\end{equation*}
 	where $\beta _{2}=\frac{2\lambda_{1}^{2}-\frac{3}{2}-22L_5-84C_{d,4}^2L_7-8\mathfrak{m}_{8}C_{d,4}^8L_{7}^{4}-11\varepsilon}{\frac{1}{2}\big (1+\frac{8}{5}\tau\big )+\big(2\lambda_{1}^{2}-\frac{3}{2}-2\varepsilon\big)\tau }>0$ .

 	\noindent\textbf{Step 3.} For $p=2^k$ with  $k \ge 2$, assuming that the following inequality holds
 	\begin{equation}\label{full energy with p=2^k}
 		\begin{aligned}
 			&\frac{1}{p}\Big (1+\frac{8}{5}\tau \Big ) \big (\mathcal{E}^p(u_{m+1}^{N})-\mathcal{E}^p(u_{m}^{N})\big )
 			+\tau \big \|A^{\frac{1}{2}}D\mathcal{E}(u_{m+1}^{N})\big\|^2\mathcal{E}^{p-1}(u_{m+1}^{N})\\
 			\le& \left(J_1+J_2+J_4+J_5+J_6\right)\mathcal{E}^{p-1}(u_{m}^{N})+\big(p-\frac{3}{2}\big)\left(J_1+J_2+J_4+J_5+J_6\right)^2\mathcal{E}^{p-2}(u_{m}^{N}) \\
 			&+\sum_{i=1}^{k-1}\mathcal{M}_i^k\left(J_1+J_2+J_4+J_5+J_6\right)^{2^{i+1}} \mathcal{E}^{p-2^{i+1}}(u_{m}^{N})+J_3\mathcal{E}^{p-1}(u_{m+1}^{N}),
 		\end{aligned} 
 	\end{equation}
 	where
 	\begin{equation*}
 		\mathcal{M}_{i}^{k}=
 		\left\{\begin{aligned}
 			&\mathcal{M}_1^{k-1}+2^{k-1}(k-1)\big(2^{k-1}-\frac{3}{2}\big)^2,  &&i=1, \\
 			&2^{k-1}(k-1)(\mathcal{M}_{i-1}^{k-1})^2+\mathcal{M}_{i}^{k-2},  &&2\le i \le k-2,\\
 			&2^{k-1}(k-1)(\mathcal{M}_{k-2}^{k-1})^2, &&i=k-1.
 		\end{aligned}\right.
 	\end{equation*}
 	 It can be verifies that when $k=2$, $\mathcal{M}_1^1 =0$ and $\mathcal{M}_2^1 =\frac{1}{2}$. We aim to prove that the same inequality remains valid for the exponent $2p=2^{k+1}$. Similar to the method of Step 3 in Lemma \ref{uniform estimate of u_m^N in H^1}, multiplying the inequality \eqref{full energy with p=2^k} by $\mathcal{E}^{p}(u_{m+1}^{N})$ and applying the identity $(a-b)a=\frac{1}{2}\left ( a^2-b^2+(a-b)^2 \right )$ together with the Young inequality yields
 	\begin{equation}\label{full energy with p=2^k+1}
 		\begin{aligned}
 			&\frac{1}{2p}\Big (1+\frac{8}{5}\tau \Big ) \big (\mathcal{E}^{2p}(u_{m+1}^{N})-\mathcal{E}^{2p}(u_{m}^{N})\big )
 			+\tau \big \|A^{\frac{1}{2}}D\mathcal{E}(u_{m+1}^{N})\big\|^2\mathcal{E}^{2p-1}(u_{m+1}^{N}) \\
 			\le& \left(J_1+J_2+J_4+J_5+J_6\right)\mathcal{E}^{2p-1}(u_{m}^{N})+\big(2p-\frac{3}{2}\big)\left(J_1+J_2+J_4+J_5+J_6 \right)^2\mathcal{E}^{2p-2}(u_{m}^{N}) \\
 			&+\sum_{i=1}^{k}\mathcal{M}_{i}^{k+1}\left(J_1+J_2+J_4+J_5+J_6\right)^{2^{i+1}} \mathcal{E}^{2p-2^{i+1}}(u_{m}^{N})+J_3\mathcal{E}^{2p-1}(u_{m+1}^{N}),
 		\end{aligned} 
 	\end{equation}
 	The preceding estimate completes the induction step from the exponent $p$ to $2p$. We next use inequality \eqref{full energy with p=2^k+1} as the starting point to establish the uniform-in-time bounds for $\mathcal{E}^{2p}(u_{m+1}^{N})$.
 	Since $u_m^N$ is $\mathcal F_{t_m}$-measurable and $\Delta_mW$
 	is independent of $\mathcal F_{t_m}$, using \eqref{linear growth in L^4} in Assumption \ref{diffusion coeffient}, the Young inequality,
 	the tower property, and the Sobolev embedding
 	$H^{\frac{d}{4}}\hookrightarrow L^4$, we obtain ( for $\varepsilon <\frac{1}{2p}$ )
 	\begin{equation*}
 		\begin{aligned}
 			\mathbb{E}\left [J_3\mathcal{E}^{2p-1}(u_{m+1}^{N})\right ]
 			\le&\frac{1}{4}\tau\mathbb{E}\Big[\|u_{m+1}^{N}\|_{L^4}^{4}\mathcal{E}^{2p-1}(u_{m+1}^{N})\Big]
 			+\frac{1}{4\tau}\mathbb{E}\Big[\|P_NG(u_{m}^{N})\Delta_m W\|_{L^4}^{4}\mathcal{E}^{2p-1}(u_{m+1}^{N})\Big]\\
 			\le&(1+\varepsilon)\tau\mathbb{E}\Big[\mathcal{E}^{2p}(u_{m+1}^{N})\Big]+C(p,\varepsilon )\tau \mathbb{E}[\|u_{m+1}^{N}\|^{4p}]+\frac{2p-1}{2p}\tau\mathbb{E}\Big[\mathcal{E}^{2p}(u_{m+1}^{N})\Big]\\
 			&+\frac{1}{2p}4^{-2p}\tau^{1-4p}\mathbb{E}\Big[\mathbb{E}\Big[\|P_NG(u_{m}^{N})\Delta_m W\|_{L^4}^{8p}|\mathcal{F}_{t_m}\Big]\Big]\\
 			\le&2\tau\mathbb{E}\Big[\mathcal{E}^{2p}(u_{m+1}^{N})\Big]+C(p,\varepsilon )\tau \mathbb{E}[\|u_{m+1}^{N}\|^{4p}]\\
 			&+\frac{1}{2p}4^{-2p}\mathfrak{m}_{8p}C_{d,4}^{8p}\tau\mathbb{E}\Big[\Big(\sum_{k=1}^{\infty}\|G(u_{m}^{N})Q^{\frac{1}{2}}\phi_{k}\|_{\dot{H}^{\frac{d}{4}}}^{2}\Big)^{4p}\Big]\\
 			\le& 2\tau\mathbb{E}\Big[\mathcal{E}^{2p}(u_{m+1}^{N})\Big]+\frac{1}{8p}\mathfrak{m}_{8p}(\sqrt{8}C_{d,4}^2L_7)^{4p}\tau\mathbb{E}\Big[\mathcal{E}^{2p}(u_{m}^{N})\Big]+C(p,\varepsilon )\tau \mathbb{E}[\|u_{m+1}^{N}\|^{4p}]\\
 			&+C(p,C_{d,4},L_{7})\mathbb{E}[\|u_{m}^{N}\|^{4p}]+C(p,C_{d,4},L_{8})\tau .
 		\end{aligned}
 	\end{equation*} 
 	As $\mathbb{E}\big [(J_{1}+J_2+J_4+J_5)^{2^{i+1}}\big]\sim \tau^{2^{i+1}}$ and $\mathbb{E}\big [J_{6}^{2^{i+1}}\big]\sim \tau^{2^{i}}$, we can obtain the following inequality  by choosing $\tau$ small enough
 	\begin{equation*}
 		\begin{aligned}
 			&\big(2p-\frac{3}{2}\big)\mathbb{E}\left[\left(J_1+J_2+J_4+J_5 \right)^2\mathcal{E}^{2p-2}(u_{m}^{N})\right ] \\
 			+&\mathbb{E}\Big[\sum_{i=1}^{k}\mathcal{M}_{i}^{k+1}\left(J_1+J_2+J_4+J_5+J_6\right)^{2^{i+1}} \mathcal{E}^{2p-2^{i+1}}(u_{m}^{N})\Big]\\
 			\le& C(p,C_{d,4},L_5,L_7)\tau^2\mathbb{E}\left [\mathcal{E}^{2p}(u_{m}^{N})\right ]+C(p,C_{d,4},L_6,L_8)\tau^2.
 		\end{aligned}
 	\end{equation*}
 	Take the expectation of \eqref{full energy with p=2^k+1} and note that the term $\left(J_1+J_2+J_3+J_4+J_5+J_6\right)\mathcal{E}^{2p-1}(u_{m}^{N})$ can be estimated similarly to \eqref{J1-J6 estimate p=2} and \eqref{J_3 estimate p=2}.
 	Define 
 	\begin{equation*}
 		\Lambda:=(80p-58)L_5+(320p-236)C_{d,4}^2L_7+\frac{1}{8p}\mathfrak{m}_{8p}(\sqrt{8}C_{d,4}^2L_7)^{4p}.
 	\end{equation*}
   By substitute the above inequalities into \eqref{full energy with p=2^k+1}, we have
 	\begin{equation*}
 		\begin{aligned}
 			&\frac{1}{2p}\Big (1+\frac{8}{5}\tau \Big )\mathbb{E}\left[\mathcal{E}^{2p}(u_{m+1}^{N})-\mathcal{E}^{2p}(u_{m}^{N})\right]
 			+(2\lambda_{1}^{2}-2-\varepsilon)\tau\mathbb{E}\left[\mathcal{E}^{2p}(u_{m+1}^{N})\right] \\ 
 			\le &\Big ( \Lambda +(20p-12)\varepsilon+C(p,C_{d,4},L_5,L_7)\tau\Big )\tau\mathbb{E}\big[\mathcal{E}^{2p}(u_{m}^{N})\big] 
 			+C(p,\varepsilon,L_5)\tau\mathbb{E}\left [\|u_{m+1}^{N}\|^{4p}+\|u_0\|^{8p}\right]\\
 			&+C(p,\varepsilon,C_{d,4},L_{3},L_{7},L_{8})\tau\mathbb{E}[\|u_{m}^{N}\|^{4p}]
 			+C(p,\varepsilon,C_{d,4}, L_{4}, L_6, L_8)(\tau+\tau^2).
 		\end{aligned}
 	\end{equation*}
 	From \eqref{dissipative for H^1} in Assumption \ref{dissipative condition}, we know that $\varepsilon := \frac{2\lambda_{1}^{2}-2-\Lambda}{20p}>0$. Set $\tau_1:=\frac{2\lambda_{1}^{2}-2-\Lambda}{20pC(p,C_{d,4},L_5,L_7)}$. Then, for any $0<\tau<\tau_1$, it follows that $C(p,C_{d,4},L_5,L_7)\tau<\varepsilon$. 
    Set $\kappa_{8p}=\mathbb{E}[\|u_0\|^{8p}]$.
 	By Lemma \ref{p-moment uniform in H}, it follows that there exists a constant $K_{4p}>0$ such that $\mathbb{E}[\|u_{m+1}^{N}\|^{4p}]\le K_{4p}$. Similarly, employing a numerical iteration yields 
 	\begin{equation*}
 		\begin{aligned}
 			\mathbb{E}\left[\mathcal{E}^{2p}(u_{m+1}^{N})\right]\le& e^{-\beta_{3}t_{m+1}}\mathbb{E}\big [ \mathcal{E}^{2p}(u_{0}^{N})\big]
 			+\frac{C(p,\varepsilon,C_{d,4},L_{3},L_5,L_{7},L_{8})(1+\kappa_{8p}+K_{4p})}{2\lambda_{1}^{2}-2-\Lambda-(20p-10)\varepsilon}\le C,
 		\end{aligned}
 	\end{equation*}
 	where $\beta _{3}=\frac{2\lambda_{1}^{2}-2-\Lambda-(20p-10)\varepsilon }{\frac{1}{2p}\big (1+\frac{8}{5}\tau\big )+\big(2\lambda_{1}^{2}-2-\varepsilon\big)\tau }>0$. Thus the proof is complete.

 \end{proof}

\section{Invariant measures and exponential ergodicity}\label{ergodicity}
In this section, we establish the existence and uniqueness of invariant measures for both the Markov process generated by Eq.~\eqref{Cahn-Hilliard equation} and the corresponding numerical Markov chain. Existence is obtained from the uniform-in-time moment estimates via the Krylov--Bogoliubov argument in \cite{hong2019invariant}, whereas uniqueness follows from the exponential contractivity.

Followed the work of Da Prato in \cite{daprato1996}, we introduce the space corresponding to $\alpha$: 
\begin{equation*}
	H_{-1}^{\alpha }:=\left \{ v\in H_{-1}: \left \langle v,\phi_0 \right \rangle=\alpha  \right \}. 
\end{equation*}
Notice that $H_{-1}^{\alpha}$ is an affine subspace of $H_{-1}$, and it is a linear subspace of $H_{-1}$ only when $\alpha=0$.
Let $u^x(t)$ denote the solution of Eq.~\eqref{Cahn-Hilliard equation} with initial condition $x\in H_{-1}^{\alpha}$. The associated Markov semigroup is defined by
\begin{equation*}
	P_t^\alpha\varphi(x)
	:=
	\mathbb E\big[\varphi(u^x(t))\big],
	\quad
	x\in H_{-1}^{\alpha},\quad
	\varphi\in B_b(H_{-1}^{\alpha}).
\end{equation*}
A probability measure $\pi_\alpha$ on
$\big(H_{-1}^{\alpha},\mathcal B(H_{-1}^{\alpha})\big)$ is called
invariant for $\{P_t^\alpha\}_{t\geq0}$ if
\begin{equation*}
	\int_{H_{-1}^{\alpha}}P_t^\alpha\varphi(x)\,
	\pi_\alpha(\mathrm dx)
	=
	\int_{H_{-1}^{\alpha}}\varphi(x)\,
	\pi_\alpha(\mathrm dx),
	\qquad t\geq0,
\end{equation*}
for every $t>0$ and $\varphi\in B_b(H_{-1}^{\alpha})$, where
$B_b(H_{-1}^{\alpha})$ denotes the space of all bounded
Borel-measurable functions on $H_{-1}^{\alpha}$. We further denote by
$C_b(H_{-1}^{\alpha})$ the space of all bounded continuous functions
on $H_{-1}^{\alpha}$. 

We first establish the exponential contraction property required for uniqueness. Let $u^x(t)$ and $u^y(t)$ be the solutions of Eq.~\eqref{Cahn-Hilliard equation} with initial conditions $x$ and $y$, respectively. Set $e(t)=u^{x}(t)-u^{y}(t)$ which satisfies 
\begin{equation}\label{initial error equation}
	\left\{\begin{aligned}
		&de(t)+A^2e(t)dt = -A(F(u^{x}(t))-F(u^{y}(t)))dt+(G(u^{x}(t))-G(u^{y}(t)))dW(t),  \\
		&e(0) = x-y . 
	\end{aligned}\right.
\end{equation}

\begin{lemma}\label{contraction property} 
	Suppose $x,y \in H_{-1}^{\alpha}$, Assumption \ref{diffusion coeffient} and \ref{dissipative condition} hold. Then, for every $0<\gamma_{1}<2\lambda_{1}^{2}-2\lambda_{1}-L_2$, it holds that
	\begin{equation}\label{exponential contraction invariant}
		\mathbb E[\|u^x(t)-u^y(t)\|_{-1}^{2}]
		\leq
		e^{-\gamma_{1}t}\mathbb E[\|x-y\|_{-1}^{2}], \quad t \ge 0.
	\end{equation}
    Consequently, $\{P_t^\alpha\}_{t\geq0}$ is Feller, i.e., if $\varphi\in C_b(H_{-1}^{\alpha})$, then $P_t^\alpha\varphi\in C_b(H_{-1}^{\alpha})$ for all $t>0$.
\end{lemma}
\begin{proof}
Since $x$ and $y$ belong to the same affine space $H_{-1}^{\alpha}$ and the dynamics preserve the spatial average, $e(t)$ has zero spatial average. Consequently, $A^{-1}e(t)$ is well defined.
For $\gamma_{1}>0$, applying the It\^o  formula to $\|e(t)\|_{-1}^{2}$ and integrating over $[0,t]$, we obtain
\begin{equation}\label{ergodicity H^{-1}}
	\begin{aligned}
		e^{\gamma_{1}t}\|e(t)\|_{-1}^2-\|x-y\|_{-1}^2=&
		-2\int_{0}^{t}e^{\gamma_{1}r}\|\nabla e(r)\|^2dr-2\int_{0}^{t}e^{\gamma_{1}r}\left \langle e(r),F(u^{x}(r))-F(u^{y}(r))\right \rangle dr\\
		&+\gamma_{1}\int_{0}^{t}e^{\gamma_{1}r}\|e(r)\|_{-1}^2dr+\int_{0}^{t}e^{\gamma_{1}r}\|G(u^{x}(r))-G(u^{y}(r))\|_{\mathcal{L}_{2}^{-1}}^{2}dr\\
		&+2\int_{0}^{t}e^{\gamma_{1}r}\left \langle A^{-1}e(r), (G(u^{x}(r))-G(u^{y}(r)))dW(r)\right \rangle.
	\end{aligned}
\end{equation}
Taking expectation on both sides of \eqref{ergodicity H^{-1}}, and using the classical localization argument based on stopping times, we have
\begin{equation*}
	\mathbb{E}\Big [ \int_{0}^{t}e^{\gamma_{1}r}\left \langle A^{-1}e(r), (G(u^{x}(r))-G(u^{y}(r)))dW(r)\right \rangle \Big ]=0
\end{equation*}
Using the monotonicity condition of $F$ and \eqref{global Lipschitz in H^-1} in Assumption \ref{diffusion coeffient}, together with $\|e(t)\|^2 \le \frac{1}{2\lambda _{1} }\|e(t)\|_{1}^2+\frac{\lambda _{1}}{2 }\|e(t)\|_{-1}^2$, we obtain 
\begin{equation*}
	\begin{aligned}
		\mathbb{E}[\|e(t)\|_{-1}^2]\le e^{-\gamma_{1}t}\mathbb{E}[\|x-y\|_{-1}^2]
		-\left ( 2\lambda_{1}^{2}-2\lambda_{1}-L_2-\gamma_{1}\right )\int_{0}^{t}e^{-\gamma_{1}(t-r)}\mathbb{E}[\|e(r)\|_{-1}^2]dr.
	\end{aligned}
\end{equation*}
We can choose $\gamma_{1}$ to be sufficiently small enough to ensure that $2\lambda_{1}^{2}-2\lambda_{1}-L_2-\gamma_{1}>0$.
Thus the proof is complete.
\end{proof}

Combining the uniform-in-time moment estimate established in
Lemma~\ref{uniform estimate of u(t) in H^1} with the exponential
contraction property in Lemma~\ref{contraction property}, we obtain
the following result.

\begin{theorem}\label{invariant measure exact solution}
	 Let $\alpha\in\mathbb R$ be fixed. Suppose that Assumptions
	\ref{diffusion coeffient} and \ref{dissipative condition} hold.
	Then the Markov semigroup $\{P_t^\alpha\}_{t\geq0}$ associated with
	Eq.~\eqref{Cahn-Hilliard equation} admits a unique invariant
	probability measure $\pi_\alpha$ on $H_{-1}^{\alpha}$.
\end{theorem}

\begin{proof}
	Fix a deterministic initial value
	$x\in H^1\cap H_{-1}^{\alpha}$. Consider the sequence $\{\mu_T\}_{T>0}$ of probability measure
	\begin{equation*}
		\mu_T(B)
		:=
		\frac{1}{T}\int_0^TP_t^\alpha\mathcal{X}_{B}(x)dt, \quad B\in \mathcal{B}(H_{-1}^{\alpha}), \ \ T>0. 
	\end{equation*}
	For $\mathcal{R}>0$, let $B:=\left \{ u \in H^1\cap H_{-1}^{\alpha}:\|u\|_1 \le \mathcal{R}\right \}\in \mathfrak{B}(H_{-1}^{\alpha})$.
	By the compact embedding $H^1\Subset H_{-1}$, the set $B$ is compact in $H_{-1}^{\alpha}$. Moreover, by the Chebyshev inequality and Lemma~\ref{uniform estimate of u(t) in H^1} yield
	\begin{equation*}
		\mu_T(B^c)=\frac{1}{T}\int_0^T\mathbb{P}(\|u^x(t)\|_1>\mathcal{R})dt \le\frac{1}{T\mathcal{R}^2}\int_0^T\mathbb{E}[\|u^x(t)\|_1^2]dt
		\le C\mathcal{R}^{-2} < \varepsilon.
	\end{equation*}
	Therefore, $\{\mu_T\}_{T>0}$ is tight on
	$H_{-1}^{\alpha}$.

	The contraction estimate in Lemma~\ref{contraction property}
	implies the continuity in probability of $u^x(t)$ with respect to
	$x$ in $H_{-1}^{\alpha}$. Consequently,
	$\{P_t^\alpha\}_{t\geq0}$ is Feller. The Krylov--Bogoliubov theorem
	therefore yields the existence of an invariant probability measure. Let $\mu_1$ and $\mu_2$ be two
	invariant probability measures on $H_{-1}^{\alpha}$. For every
	bounded Lipschitz function $\varphi$ on $H_{-1}^{\alpha}$, the contraction estimate in Lemma~\ref{contraction property} gives
	\begin{equation*}
		\begin{aligned}
			\left|P_t^\alpha\varphi(x)-P_t^\alpha\varphi(y)\right|
			\leq\min\left\{2\|\varphi\|_{\infty},
			e^{-\frac{\gamma_1t}{2}/}\operatorname{Lip}(\varphi)\|x-y\|_{-1}
			\right\}.
		\end{aligned}
	\end{equation*}
	Using the invariance of $\mu_1$ and $\mu_2$, we obtain
	\begin{equation*}
		\begin{aligned}
			\left|\mu_1(\varphi)-\mu_2(\varphi)\right|
			\leq\int_{H_{-1}^{\alpha}}\int_{H_{-1}^{\alpha}}
			\min\left\{
			2\|\varphi\|_{\infty},
			e^{-\gamma_1t/2}
			\operatorname{Lip}(\varphi)\|x-y\|_{-1}
			\right\}\mu_1(\mathrm dx)\mu_2(\mathrm dy).
		\end{aligned}
	\end{equation*}
	Letting $t\to\infty$ and applying the dominated convergence
	theorem, we obtain $\mu_1(\varphi)=\mu_2(\varphi)$.
	Since bounded Lipschitz functions determine probability measures
	on $H_{-1}^{\alpha}$, it follows that $\mu_1=\mu_2$. Hence the
	invariant probability measure is unique.

\end{proof}

An analogous argument applies to the numerical Markov chain. In this
case, existence follows from the uniform-in-time moment estimate established in Lemma \ref{p-moment uniform in H} with $p=2$, whereas uniqueness follows from the corresponding discrete contraction property.

\begin{theorem}\label{invariant measure numerical solution}
	Let $\alpha\in\mathbb R$ and $N\in\mathbb N$ be fixed. Suppose that Assumptions
	\ref{diffusion coeffient} and \ref{dissipative condition} hold. For every $0<\tau<1$, 
	the numerical Markov chain $\{u_m^N\}_{m\geq0}$ admits a unique
	invariant probability measure $\pi_\alpha^{\tau,N}$ on $H_{N,-1}^{\alpha}:=H_N\cap H_{-1}^{\alpha}$.
\end{theorem}
The proof is similar to that of Theorem~\ref{invariant measure exact solution} and is therefore omitted.

 \section{Uniform-in-time strong convergence rates}\label{strong convergence}
 
 In this section, we establish the uniform-in-time strong convergence rates of the fully discrete scheme. Our goal is to derive error estimates that are uniform with respect to both the discretization parameters and time variable $t$.

 To facilitate the error analysis, we first establish several stability and approximation properties of the discrete solution operator $S_{\tau,N}^{m}$. These estimates provide the required smoothing properties of the discrete solution operator and quantify its approximation error with respect to the continuous solution operator, which are essential for the subsequent convergence analysis.

 \begin{lemma}\label{semigroup full discrte}
 	For any $m>0$ and $\varphi \in H$, there exists $0<c<1$ such that
 	\begin{equation*}
 		\begin{aligned}
 			\|A^\nu S_{\tau,N}^m \varphi\|\le Ct_m^{-\frac{\nu }{2}}e^{-\frac{w(\nu)}{2} \lambda_1^2t_m}\| \varphi\| +k(\nu)|\langle \varphi, \phi_0 \rangle|, \qquad \nu \in [0,2), \\
 		\end{aligned}
 	\end{equation*}
 	where the function $w(x): [0,2) \to \mathbb{R}$ is defined by $w(x)=c$ when $\nu \in [0,1]$ and $w(x)=(2-\nu)c$ for $\nu \in (1,2)$.
 	
 \end{lemma}

 Denote by $E_{\tau,N}(t)=S(t)-S_{\tau,N}^{m+1}$ for $t \in(t_m, t_{m+1}]$. The following estimates for $E_{\tau,N}(t)$ will be frequently used in the error analysis.
 
 \begin{lemma}\label{error of semigrroup}
 	For $0<c<1$, the operator $E_{\tau,N}(t)$ satisfies the following estimates
 	\begin{align}
 		&\left\|E_{\tau,N}(t) \varphi\right\| \leq C\big(\lambda _{N+1}^{-\frac{\mu}{2} }+\tau^{\frac{\mu }{4}}\big) t^{-\frac{\mu -\nu }{4}}\|\varphi\|_{\nu },
 		&&0\le \nu \le \mu \le 4, \varphi \in H^{\nu }, \label{positive full semigroup}\\
 		&\left\|E_{\tau,N}(t) \varphi\right\| \leq C\big(\lambda _{N+1}^{-\frac{\mu}{2} }+\tau^{\frac{\mu}{4}}\big) t^{-\frac{\mu-\nu}{4}}e^{-\frac{c\lambda _{1}^{2} }{2}t }\|\varphi\|_\nu,
 		&&\mu \in [0,2], \nu \in [-2,0], \varphi \in H^\nu. \label{nonnegative full semigroup}\\
 		&\Big( \int_{0}^{t}\|E_{\tau,N}(r) \varphi\|_{\mathcal{L}(\dot{H})}^{2}dr \Big)^{\frac{1}{2}}\le C(\lambda _{N+1}^{-1}+\tau ^{\frac{1}{2}}), && \varphi \in H. \label{integral full semigroup}
 	\end{align}
 \end{lemma}
 
The proofs of the preceding lemmas are given in \ref{semigroup for full discrete} and \ref{solution operator difference}, respectively.

 For the subsequent error analysis, we introduce the auxiliary process
\begin{equation}\label{auxiliary full mild solution}
	\begin{aligned}
		\tilde{u}_{m+1}^{N}=S_{\tau,N}^{m+1}u_{0}^{N}-\tau\sum_{i=0}^{m}S_{\tau,N}^{m+1-i}A P_N F(u(t_{i+1}))+\sum_{i=0}^{m}S_{\tau,N}^{m+1-i}P_N G(u(t_{i}))\triangle_i W.
	\end{aligned}
\end{equation}
By the triangle inequality, the total error can then be bounded by
\begin{equation}\label{total error decomposition}
	\|u(t_{m+1})-u_{m+1}^{N}\|\le \|u(t_{m+1})-\tilde{u}_{m+1}^{N}\|+\|\tilde{u}_{m+1}^{N}-u_{m+1}^{N}\|.
\end{equation}
Let $2 \le p \le 24$ and fix $\gamma \in \big(\max\{1, \frac{d}{2}\}, 2\big )$. Suppose that $u_0\in L^{6p} ( \Omega;H^{\gamma})$, Assumptions \ref{diffusion coeffient} and \ref{dissipative condition} hold. Then there exists a constant $C>0$, independent of $N$, $\tau$, and $m$, such that, for every $N\in\mathbb N$, 
 \begin{equation}\label{auxilary solution stability}
 		\sup_{m\geq0}\|\tilde{u}_{m}^{N}\|_{L^{p}(\Omega; H^{\gamma})} \le C.		
 \end{equation}
This estimate follows from arguments analogous to those used in the proof of Theorem \ref{continous uniform moment}, together with the stability and smoothing properties of the discrete solution operator in Lemma \ref{semigroup full discrte}. The details are omitted.

Based on the properties of $S_{\tau,N}^{m+1}$ and $E_{\tau,N}(t)$ established in Lemmas \ref{semigroup full discrte} and \ref{error of semigrroup}, respectively, we first estimate the term $\|u(t_{m+1})-\tilde{u}_{m+1}^{N}\|$ in the following theorem. 

 \begin{theorem}\label{error estimate with auxilary solution}
 	Let $2 \le p\le 8$ and $u_{0} \in L^{144}(\Omega;H^2)$. Suppose that Assumptions \ref{diffusion coeffient} and \ref{dissipative condition} hold. Then we have
 	\begin{equation*}
 		\left \|u(t_{m+1})-\tilde{u}_{m+1}^{N}\right \|_{L^p(\Omega;H)} \le C\big(\lambda _{N+1}^{-1}+\tau^{\frac{1}{2}}\big). 
 	\end{equation*}
 \end{theorem}
 \begin{proof}
 	According to the definition of $u(t_{m+1})$ and $\tilde{u}_{m+1}^{N}$, it follows that
 	\begin{equation*}
 		\begin{aligned}
 			&\left \|u(t_{m+1})-\tilde{u}_{m+1}^{N}\right \|_{L^p(\Omega;H)}\\
 			\le&\left\|E_{\tau,N}(t_{m+1})u_0\right \|_{L^p(\Omega;H)} 
 			+
 			 \Big\|\int_{0}^{t_{m+1}}S(t_{m+1}-s)AP_NF(u(s))ds-\tau \sum_{i=0}^{m}S_{\tau,N}^{m+1-i}AF(u(t_{i+1}))\Big \|_{L^p(\Omega;H)}
 			\\
 			&+\Big\|\int_{0}^{t_{m+1}}S(t_{m+1}-s)G(u(s))dW(s)- \sum_{i=0}^{m}S_{\tau,N}^{m+1-i}P_NG(u(t_{i}))\triangle_i W \Big \|_{L^p(\Omega;H)}\\
        	=&:J_1+J_2+J_3.
    	\end{aligned}
      \end{equation*}
 	By \eqref{positive full semigroup} in Lemma \ref{error of semigrroup} with $\mu =\nu =2$, one can deduce that 
 	\begin{equation*}
 		J_1 \le C \big(\lambda _{N+1}^{-1 }+\tau^{\frac{1}{2}}\big)\|u_0\|_{L^p(\Omega;H^{2})}.
 	\end{equation*}
 	For the second term, we divide it into the following two parts:
 	\begin{equation*}
 		\begin{aligned}
 			J_2\le& \Big\|\sum_{i=0}^{m}\int_{t_i}^{t_{i+1}}E_{\tau,N}(t_{m+1}-s)AF(u(s))ds\Big \|_{L^p(\Omega;H)}\\
 			&+\Big\|\sum_{i=0}^{m}\int_{t_i}^{t_{i+1}}S_{\tau,N}^{m+1-i}P_NA(F(u(s))-F(u(t_{i+1})))ds\Big \|_{L^p(\Omega;H)} \\
 			=:&J_{21}+J_{22}.
 		\end{aligned}
 	\end{equation*}
 	Using \eqref{nonnegative full semigroup} in Lemma \ref{error of semigrroup} with $\mu=2, \nu=-1$ and Theorem \ref{continous uniform moment}, for $\gamma >\frac{d}{2}$, one can derive
 	\begin{equation*}
 		\begin{aligned}
 			J_{21}\le& \sum_{i=0}^{m}\int_{t_i}^{t_{i+1}}\big\|E_{\tau,N}(t_{m+1}-s)AF(u(s))\big \|_{L^p(\Omega;H)}ds\\
 			\le &C\big(\lambda _{N+1}^{-1}+\tau^{\frac{1}{2}}\big) \sup_{t>0} \|u(t)\|_{L^{3p}(\Omega;H^{\gamma})}^{3}
 			\sum_{i=0}^{m}\int_{t_i}^{t_{i+1}}(t_{m+1}-s)^{-\frac{3}{4}}e^{-\frac{c\lambda _{1}^{2}}{2}(t_{m+1}-s)}ds\\
 			\le& C\big(\lambda _{N+1}^{-1 }+\tau^{\frac{1}{2}}\big).
 		\end{aligned}
 	\end{equation*}
 	Through the H\"{o}lder regularity of $u(t)$ in Theorem \ref{continous uniform moment} and Lemma \ref{semigroup full discrte} with $\nu=1$, we have 
 		\begin{align*}
 			J_{22}\le&\sum_{i=0}^{m}\int_{t_i}^{t_{i+1}}\big\|S_{\tau,N}^{m+1-i}P_NA(F(u(s))-F(u(t_{i+1})))\big \|_{L^p(\Omega;H)}ds\\
 			\le& C\sum_{i=0}^{m}\int_{t_i}^{t_{i+1}}(t_{m+1}-t_i)^{-\frac{1}{2}}e^{-\frac{c\lambda _{1}^{2}}{4}(t_{m+1}-t_i)}\big\|F(u(s))-F(u(t_{i+1}))\big \|_{L^p(\Omega;H)}ds\\
 			\le& C\sum_{i=0}^{m}\int_{t_i}^{t_{i+1}}(t_{m+1}-t_i)^{-\frac{3}{4}}e^{-\frac{c\lambda _{1}^{2}}{4}(t_{m+1}-t_i)}
 			\big(1+\sup_{s>0} \|u(s)\|_{L^{3p}(\Omega;L^{\infty})}^2 \big)\|u(s)-u(t_{i+1})\|_{L^{3p}(\Omega;H)}ds\\
 			\le& C\tau^{\frac{1}{2}}.
 		\end{align*}
 	Similarly, for the third term, it can be divided into two parts
 	\begin{equation*}
 		\begin{aligned}
 			J_3\le& \Big\|\sum_{i=0}^{m}\int_{t_i}^{t_{i+1}}E_{\tau,N}(t_{m+1}-s)G(u(s))dW(s)\Big \|_{L^p(\Omega;H)}\\
 			&+\Big\|\sum_{i=0}^{m}\int_{t_i}^{t_{i+1}}S_{\tau,N}^{m+1-i}(G(u(s))-G(u(t_i)))dW(s)\Big \|_{L^p(\Omega;H)}
 			=:J_{31}+J_{32}.
 		\end{aligned}
 	\end{equation*}
 	Through the Burkh\"{o}lder-Davis-Gundy inequality, the H\"{o}lder inequality and \eqref{integral full semigroup} in Lemma \ref{error of semigrroup}, we obtain
 		\begin{align*}
 			J_{31}\le& \Big ( \sum_{i=0}^{m}\int_{t_i}^{t_{i+1}}\big\|E_{\tau,N}(t_{m+1}-s)G(u(s))\Big \|_{L^p(\Omega;\mathcal{L}_{2}^{0})}^2ds\big )^{\frac{1}{2}}\\ 
 			\le& \Big ( \sum_{i=0}^{m}\int_{t_i}^{t_{i+1}}\|E_{\tau,N}(t_{m+1}-s)\|^{2}_{\mathcal{L}(\dot{H})}\|G(u(s)) \|_{L^p(\Omega;\mathcal{L}_{2}^{0})}^2ds\big )^{\frac{1}{2}}\\
 			\le& C\Big ( \sum_{i=0}^{m}\int_{t_i}^{t_{i+1}}\|E_{\tau,N}(t_{m+1}-s)\|^{2}_{\mathcal{L}(\dot{H})}(\|u(s)\|_{L^p(\Omega;H)}^2+1)ds\big )^{\frac{1}{2}}\\
 			\le& C\Big (\int_{0}^{t_{m+1}}\|E_{\tau,N}(t_{m+1}-s)\|^{2}_{\mathcal{L}(\dot{H})} ds\big )^{\frac{1}{2}}
 			\le C\big(\lambda _{N+1}^{-1}+\tau^{\frac{1}{2}}\big).
 		\end{align*}
 	Furthermore, by the H\"{o}lder regularity of $u(t)$ in Theorem \ref{continous uniform moment} and Lemma \ref{semigroup full discrte} with $\nu=0$, we get
 	\begin{align*}
 			J_{32}\le& C\Big ( \sum_{i=0}^{m}\int_{t_i}^{t_{i+1}}\|S_{\tau,N}^{m+1-i}(G(u(s))-G(u(t_i)))\|_{L^p(\Omega;\mathcal{L}_{2}^{0})}^2ds\Big )^{\frac{1}{2}}\\ 
 			\le &C\Big ( \sum_{i=0}^{m}\int_{t_i}^{t_{i+1}}e^{-c\lambda _{1}^{2}(t_{m+1}-t_i)}\|G(u(s))-G(u(t_i))\|_{L^p(\Omega;\mathcal{L}_{2}^{0})}^2ds\Big )^{\frac{1}{2}}\\
 			\le& C\Big ( \sum_{i=0}^{m}\int_{t_i}^{t_{i+1}}e^{-c\lambda _{1}^{2}(t_{m+1}-t_i)}\|u(s)-u(t_i)\|_{L^p(\Omega;H)}^2ds\Big )^{\frac{1}{2}}
 			\le C\tau^{\frac{1}{2}}.
 	\end{align*}
 	Combing the above inequalities together, we complete the proof.
 	
 \end{proof}
 
 For the second term of \eqref{total error decomposition}, define $e_{m+1}^{N}=\tilde{u}_{m+1}^{N}-u_{m+1}^{N}$ which satisfies 
 \begin{equation}\label{error full-discrete scheme}
 	\left\{\begin{aligned}
 		&e_{m+1}^{N} - e_{m}^{N}
 		= -\tau A^2 e_{m+1}^{N}-\tau A P_N F(u(t_{m+1}))
 		+ \tau A P_N F(u_{m+1}^{N})\\
 		&\quad \quad + P_N G(u(t_m))\triangle_m W-P_N G(u_m^N)\triangle_m W, \\
 		&e^N_0 = 0. 
 	\end{aligned}\right.
 \end{equation}
 Following the same techniques in the proof of Lemma \ref{p-moment uniform in H}, we can deduce the following Lemma.
 
 \begin{lemma}\label{error for full discrete in H^-1 space}
 	Let $u_{0} \in L^{144}(\Omega;H^2)$. Suppose that Assumptions \ref{diffusion coeffient} and \ref{dissipative condition} hold. Then, for any $0<\tau<\frac{\lambda_{1}^{2}-\lambda_{1}-5L_{2}}{144L_2^2(\lambda_{1}^{2}+6)}$, it holds that
 	\begin{equation*}
 		\mathbb{E}[  \|A^{-\frac{1}{2} } e_{m+1}^{N}\|^4  ]\le C\big(\lambda _{N+1}^{-4}+\tau^{2} \big). 
 	\end{equation*}
 \end{lemma}
 \begin{proof}
 	Taking the inner product in $H$ on both sides of \eqref{error full-discrete scheme} with $A^{-1}e_{m+1}^N$ and using the Young inequality, we have
    \begin{equation}
    	\begin{aligned}
    		&\frac{1}{2}\big(\|A^{-\frac12} e_{m+1}^{N}\|^2-\|A^{-\frac12} e_{m}^{N}\|^2+ \|A^{-\frac12}(e_{m+1}^{N}-e_{m}^{N})\|^2\big) + \tau \|A^{\frac12} e_{m+1}^{N}\|^2 \\
    		=& -\tau \langle F(u(t_{m+1}))-F(u_{m+1}^{N}), e_{m+1}^{N} \rangle 
    		+ \langle (G(u(t_{m}))-G(u_{m}^{N}))\Delta_m W, A^{-1} e_{m+1}^{N}\rangle\\
    		=&-\tau \langle F(u(t_{m+1}))-F(\tilde{u}_{m+1}^{N}), e_{m+1}^{N} \rangle
    		-\tau \langle F(\tilde{u}_{m+1}^{N})-F(u_{m+1}^{N}), e_{m+1}^{N} \rangle \\
    		&+\langle (G(u(t_{m}))-G(u_{m}^{N}))\Delta_m W, A^{-1} (e_{m+1}^{N}-e_{m}^{N})\rangle + \langle (G(u(t_{m}))-G(u_{m}^{N}))\Delta_m W, A^{-1} e_{m}^{N}\rangle\\
    		\le& \varepsilon\tau\|e_{m+1}^{N}\|^2
    		+\frac{1}{4\varepsilon }\tau\|F(u(t_{m+1}))-F(\tilde{u}_{m+1}^{N})\|^2 +\tau \|e_{m+1}^{N}\|^2+\frac{1}{2}\|A^{-\frac12}(e_{m+1}^{N}-e_{m}^{N})\|^2 \\
    	    &+\frac{1}{2}\|A^{-\frac{1}{2}}(G(u(t_{m}))-G(u_{m}^{N}))\Delta_m W\|^2
    	    + \langle (G(u(t_{m}))-G(u_{m}^{N}))\Delta_m W, A^{-1} e_{m}^{N}\rangle
    	\end{aligned}
    \end{equation}
 	For the nonlinear term, we obtain
 	\begin{equation*}
 		\|F(u(t_{m+1}))-F(\tilde{u}_{m+1}^{N})\|^2
 		\le C\bigl(1+\|u(t_{m+1})\|_{L^{\infty}}^4+\|\tilde{u}_{m+1}^{N}\|_{L^{\infty}}^4\bigr)
 		\|u(t_{m+1})-\tilde{u}_{m+1}^{N}\|^2 .
 	\end{equation*}
 	Utilizing the inequalities
 	\begin{equation*}
 		\begin{aligned}
 			&\|e_{m+1}^{N}\|^2\le \frac{1}{2\lambda_{1}}\|A^{\frac{1}{2}}e_{m+1}^{N}\|^2+\frac{\lambda_{1}}{2}\|A^{-\frac{1}{2}}e_{m+1}^{N}\|^2, \\
 			&\|A^{\frac{1}{2}}e_{m+1}^{N}\|^2\ge \lambda_{1}^{2}\|A^{-\frac{1}{2}}e_{m+1}^{N}\|^2, 
 		\end{aligned}
 	\end{equation*}
 	we rewrite the above inequality as follows
 	\begin{equation}\label{first estimate in H^-1}
 		\begin{aligned}
 			&\frac{1}{2}\big(\|A^{-\frac12} e_{m+1}^{N}\|^2-\|A^{-\frac12} e_{m}^{N}\|^2\big) + \big(\lambda_1^2-(1+\varepsilon)\lambda_1\big )\tau \|A^{-\frac12} e_{m+1}^{N}\|^2 \\
 			\le &C_{\varepsilon}\tau\bigl(1+\|u(t_{m+1})\|_{L^{\infty}}^4+\|\tilde{u}_{m+1}^{N}\|_{L^{\infty}}^4\bigr)
 			\|u(t_{m+1})-\tilde{u}_{m+1}^{N}\|^2\\
 			&+\frac{1}{2}\|A^{-\frac{1}{2}}(G(u(t_{m}))-G(u_{m}^{N}))\Delta_m W\|^2 
 			+ \langle (G(u(t_{m}))-G(u_{m}^{N}))\Delta_m W, A^{-1} e_{m}^{N}\rangle \\
 			=:&I_1+I_2+I_3.
 		\end{aligned}
 	\end{equation}
 	Through Theorem \ref{continous uniform moment} and \eqref{auxilary solution stability} with $p=8$, Theorem \ref{error estimate with auxilary solution} with $p=4$, we have
 	\begin{equation}
 		\begin{aligned}
 			\mathbb{E}[I_1]\le C_{\varepsilon}\tau \left (\lambda_{N+1}^{-2}+\tau \right ).
 		\end{aligned}
 	\end{equation}
 	 Utilizing \eqref{global Lipschitz in H^-1} in Assumption \ref{diffusion coeffient}, together with $\left \langle (I-P_N)u_0,\phi_0 \right \rangle=0$,  one can deduce 
 	 \begin{equation*}
 	 	\begin{aligned}
 	 		\mathbb{E}[I_2]
 	 		&= \frac{1}{2}\tau \mathbb{E}[\|G(u(t_{m}))-G(u_{m}^{N})\|_{\mathcal{L}_{2}^{-1}}^2] \\
 	 		&\le L_{2}\tau\mathbb{E}[\|A^{-\frac12}(u(t_{m})-\tilde{u}_{m}^{N})\|^2]
 	 		+L_{2}\tau\mathbb{E}[\|A^{-\frac12}e_m^N\|^2]\\
 	 		&\le C\tau(\lambda_{N+1}^{-2}+ \tau)+L_{2}\tau\mathbb{E}[\|A^{-\frac12}e_m^N\|^2] .
 	 	\end{aligned}
 	 \end{equation*}
    Combining the above estimates, we obtain
    \begin{equation*}
    	\begin{aligned}
    		\Bigl[\frac{1}{2} + \big(\lambda_1^2-(1+\varepsilon)\lambda_1\big )\tau\Bigr]
    		\mathbb{E}[\|A^{-\frac12}e_{m+1}^{N}\|^2]
    		\le \Big( \frac{1}{2} + L_{2}\tau\Big)\mathbb{E}[\|A^{-\frac12}e_{m}^{N}\|^2]+C\tau(\lambda_{N+1}^{-2}+ \tau).
    	\end{aligned}
    \end{equation*}
    According to \eqref{dissipative for H^1} in Assumption \ref{dissipative condition}, we can choose $0<\varepsilon <\frac{\lambda_{1}^{2}-\lambda_{1}-L_{2}}{\lambda_{1}}$. 
 	Iterating the preceding inequality and using the fact $e_0^N=0$, a standard induction argument yields
 	\begin{equation*}
 		\begin{aligned}
 			\mathbb{E}[\|A^{-\frac12}e_{m+1}^{N}\|^2] \le C(\lambda_{N+1}^{-2}+ \tau).
 		\end{aligned}
 	\end{equation*}
 	Multiplying the inequality \eqref{first estimate in H^-1} by $\|A^{-\frac12} e_{m+1}^{N}\|^2$, applying the Young inequality and by simple calculations, we have
 	\begin{equation}\label{second estimate for error}
 		\begin{aligned}
 			&\frac{1}{4}\big(\|A^{-\frac12} e_{m+1}^{N}\|^4-\|A^{-\frac12} e_{m}^{N}\|^4\big) 
 			+ \big((1-\varepsilon)\lambda_1^2-\lambda_1\big )\tau \|A^{-\frac12} e_{m+1}^{N}\|^4 \\
 			\le &I_1\|A^{-\frac12} e_{m+1}^{N}\|^2+(I_2+I_3)\|A^{-\frac12} e_{m}^{N}\|^2+(I_2+I_3)^2\\
 			\le &\varepsilon\tau \|A^{-\frac12} e_{m+1}^{N}\|^4+ \frac{1}{4\varepsilon\tau} I_1^2+(I_2+I_3)\|A^{-\frac12} e_{m}^{N}\|^2+(I_2+I_3)^2.
 		\end{aligned}
 	\end{equation}
 	Similarly, by Theorem \ref{continous uniform moment} and \eqref{auxilary solution stability} with $p=16$, Theorem \ref{error estimate with auxilary solution} with $p=8$ and \eqref{global Lipschitz in H^-1} in Assumption \ref{diffusion coeffient}, we have
 		\begin{align*}
 			\mathbb{E}[I_1^2]
 			\le& C_{\varepsilon}\tau\mathbb{E}\left[\bigl(1+\|u(t_{m+1})\|_{L^{\infty}}^8+\|\tilde{u}_{m+1}^{N}\|_{L^{\infty}}^8\bigr)
 			\|u(t_{m+1})-\tilde{u}_{m+1}^{N}\|^4 \right ] \\
 			\le& C_{\varepsilon}\tau\left (\lambda_{N+1}^{-4}+\tau^{2} \right ), \\
 			\mathbb{E}[I_2\|A^{-\frac12}e_{m}^{N}\|^2] 
 			 \le &\frac{1}{2}\tau \mathbb{E}[\|G(u(t_{m}))-G(u_{m}^{N})\|_{\mathcal{L}_{2}^{-1}}^2\|A^{-\frac12}e_{m}^{N}\|^2] \\
 			 \le &(L_{2}+\varepsilon )\tau\mathbb{E}[\|A^{-\frac12}e_{m}^{N}\|^4]+C_{\varepsilon}\tau \left ( \lambda_{N+1}^{-4}+\tau^{2} \right ), \\
 			 \mathbb{E}[I_3^2]\le &\mathbb{E}\left[\|A^{-\frac{1}{2}}(G(u(t_{m}))-G(u_{m}^{N}))\Delta_m W\|^2\|A^{-\frac12}e_{m}^{N}\|^2\right ]\\
 			 =& \tau\mathbb{E}\left[\|G(u(t_{m}))-G(u_{m}^{N})\|_{\mathcal{L}_{2}^{-1}}^2\|A^{-\frac12}e_{m}^{N}\|^2\right ]\\
 			 \le &\left ( 2L_{2}+\varepsilon \right )\tau \mathbb{E}[\|A^{-\frac12}e_{m}^{N}\|^4]+C_{\varepsilon}\tau \left ( \lambda_{N+1}^{-4}+\tau^{2} \right ).
 		\end{align*}
 	Using the Burkh\"{o}lder-Davis-Gundy inequality and \eqref{global Lipschitz in H^-1} in Assumption \ref{diffusion coeffient}, we deduce that
 	\begin{equation*}
 		\begin{aligned}
 			\mathbb{E}[I_2^2]\le& \frac{1}{4}c_4\tau ^{2} \mathbb{E}\big [ \|G(u({t_m}))-G(u_{m}^{N})\|_{\mathcal{L}_{2}^{-1}}^{4}\big]
 			\le \frac{1}{4}c_4L_2^2\tau ^{2}\mathbb{E}\big [ \|u({t_m})-u _{m}^{N} \|_{-1}^{4}\big]\\
 			\le &2c_4L_2^2\tau ^{2}\mathbb{E}\big [\big \|A^{-\frac{1}{2}}(u({t_m})-\tilde{u}_{m}^{N})\big \|_{-1}^{4}\big]
 			+2c_4L_2^2\tau ^{2}\mathbb{E}\big [\big \|A^{-\frac{1}{2}}e_{m}^{N}\big \|_{-1}^{4}\big]\\
 			\le &72L_2^2\tau ^{2}\mathbb{E}\big [\big \|A^{-\frac{1}{2}}e_{m}^{N}\big \|_{-1}^{4}\big]+C\tau^{2}\left ( \lambda_{N+1}^{-4}+\tau^{2} \right ),
 		\end{aligned}
 	\end{equation*}
 	where $c_{4}$ denotes the constant in the Burkh\"{o}lder-Davis-Gundy inequality and $c_{p}=\big(\frac{p(p-1)}{2}\big)^{\frac{p}{2}}$.
 	Taking expectations on both sides of inequality \eqref{second estimate for error} and noting that the term $I_3\|A^{-\frac12} e_{m}^{N}\|^2$ vanishes upon taking expectation and $\mathbb{E}\left [ I_{2}^{2}\right ]\sim \tau^{2}$, we substituting the above inequality into \eqref{second estimate for error} and conclude that
 	\begin{equation*}
 		\begin{aligned}
 			\Bigl[\frac{1}{4} + \big((1-\varepsilon)\lambda_1^2-\lambda_1-\varepsilon \big )\tau\Bigr]
 			\mathbb{E}[\|A^{-\frac12}e_{m+1}^{N}\|^4]
 			\le \Big( \frac{1}{4} +(5L_{2}+3\varepsilon+144L_2^2\tau ) \tau\Big)\mathbb{E}[\|A^{-\frac12}e_{m}^{N}\|^4]+C\tau(\lambda_{N+1}^{-4}+ \tau^{2}).
 		\end{aligned}
 	\end{equation*}
 	From \eqref{dissipative for H}, we have $\varepsilon =\frac{\lambda_{1}^{2}-\lambda_{1}-5L_{2}}{\lambda_{1}^{2}+6}>0$. Then, for any $0<\tau<\frac{\lambda_{1}^{2}-\lambda_{1}-5L_{2}}{144L_2^2(\lambda_{1}^{2}+6)}$, 
 	it follows that $144L_2^2\tau<\varepsilon$.
 	Similarly, by employing a numerical iteration, the proof is complete.

 \end{proof}

 \begin{theorem}\label{the second error}
 	Let $u_{0} \in L^{144}(\Omega;H^2)$. Suppose that Assumptions \ref{diffusion coeffient} and \ref{dissipative condition} hold. Then, for every
 	$0<\tau <\min\big\{\tau_0^*,\tau_1^*, \frac{1}{\lambda _{1}^{2}},\frac{\lambda_{1}^{2}-\lambda_{1}-5L_{2}}{144L_2^2(\lambda_{1}^{2}+6)}\big\} $, there exists a constant $C>0$ such that  
 	\begin{equation*}
 		 \|e_{m+1}^{N}\|_{L^2(\Omega;H)} \le C\big(\lambda _{N+1}^{-1}+\tau^{\frac{1}{2}}\big). 
 	\end{equation*}
 \end{theorem}
 \begin{proof}
 	In order to prove the convergence estimates for $\|e_{m+1}^{N}\|_{L^2(\Omega;H)}$, we split the proof into two steps. \\
 	\textbf{Step 1.} For $\gamma_{3}\le \frac{\lambda _{1}^{2}}{2}$, we first show that 
 	\begin{equation}\label{middle process in error estimate}
 		\mathbb{E}\Big [ \Big (\tau \sum_{i=0}^{m}e^{-\gamma_{3}(t_{m+1}-t_{i+1})}\|A^{\frac12}e_{i+1}^{N}\|^2\Big )^2 \Big ]
 		\le C\big(\lambda _{N+1}^{-4}+\tau^{2}\big).
 	\end{equation}
 	By applying the estimate obtained in Lemma \ref{error for full discrete in H^-1 space}, we have
    \begin{equation}\label{main estimate in step1}
    	\begin{aligned}
    		&\|A^{-\frac12} e_{m+1}^{N}\|^2- \|A^{-\frac12} e_{m}^{N}\|^2+(2-2\varepsilon)\tau \|A^{\frac12} e_{m+1}^{N}\|^2\\
    		\le&\frac{1}{\varepsilon }\tau\|A^{-\frac{1}{2} } e_{m+1}^{N}\|^2+C_{\varepsilon}\tau \bigl(1+\|u(t_{m+1})\|_{L^{\infty}}^4+\|\tilde{u}_{m+1}^{N}\|_{L^{\infty}}^4\bigr)
    		\|u(t_{m+1})-\tilde{u}_{m+1}^{N}\|^2 \\
    		&+\|A^{-\frac{1}{2}}(G(u(t_{m}))-G(u_{m}^{N}))\Delta_m W\|^2+2\langle (G(u(t_{m}))-G(u_{m}^{N}))\Delta_m W, A^{-1} e_{m}^{N}\rangle .
    	\end{aligned}
    \end{equation}
    Multiplying both sides by $e^{\gamma_{3}t_{m+1}}$, utilizing the inequality 
    \begin{align*}
    	e^{\gamma_{3}t_{m+1}}\|A^{-\frac12}e_{m}^{N}\|^2
    	=& e^{\gamma_{3}t_{m}}\|A^{-\frac12}e_{m}^{N}\|^2+e^{\gamma_{3}t_{m+1}}(1-e^{-\gamma_{3}\tau})\|A^{-\frac12} e_{m}^{N}\|^2\\
    	\le& e^{\gamma_{3}t_{m}}\|A^{-\frac12}e_{m}^{N}\|^2+\gamma_{3}e^{\gamma_{3}t_{m+1}}\tau\|A^{-\frac12} e_{m}^{N}\|^2,
    \end{align*}
    replacing index $m$ by $i$ and summing over $i$ from 0 to $m$, we obtain
    	\begin{align}\label{summation of step 1}
    		&e^{\gamma_{3}t_{m+1}}\|A^{-\frac12} e_{m+1}^{N}\|^2+(2-2\varepsilon)\tau\sum_{i=0}^{m} e^{\gamma_{3}t_{i+1}}\|A^{\frac12} e_{i+1}^{N}\|^2 \notag\\
    		\le& \gamma_{3}\tau \sum_{i=0}^{m}e^{\gamma_{3}t_{i+1}}\|A^{-\frac12}e_{i}^{N}\|^2
    		+\frac{1}{\varepsilon}\tau\sum_{i=0}^{m} e^{\gamma_{3}t_{i+1}}\|A^{-\frac{1}{2} }e_{i+1}^{N}\|^2 \notag\\
    		&+C_{\varepsilon}\tau \sum_{i=0}^{m}e^{\gamma_{3}t_{i+1}}\bigl(1+\|u(t_{i+1})\|_{L^{\infty}}^4+\|\tilde{u}_{i+1}^{N}\|_{L^{\infty}}^4\bigr)\|u(t_{i+1})-\tilde{u}_{i+1}^{N}\|^2\\
    		&+\sum_{i=0}^{m}e^{\gamma_{3}t_{i+1}}\|A^{-\frac{1}{2}}(G(u(t_{i}))-G(u_{i}^{N}))\Delta_i W\|^2
    		+2\sum_{i=0}^{m}e^{\gamma_{3}t_{i+1}}\langle(G(u(t_{i}))-G(u_{i}^{N}))\Delta_i W, A^{-1} e_{i}^{N}\rangle.\notag
    	\end{align}
    By choosing $\tau<\frac{1}{\lambda _{1}^{2}}$, one can easily deduce that 
    \begin{align*}
    	\gamma_{3}\tau \sum_{i=0}^{m}e^{\gamma_{3}t_{i+1}}\|A^{-\frac12}e_{i}^{N}\|^2
    	=\gamma_{3}e^{\gamma_{3}\tau}\tau\sum_{i=0}^{m}e^{\gamma_{3}t_{i}}\|A^{-\frac12}e_{i}^{N}\|^2
    	\le \lambda_{1}^{2}\tau\sum_{i=0}^{m}e^{\gamma_{3}t_{i}}\|A^{-\frac12}e_{i}^{N}\|^2
    	\le\tau\sum_{i=0}^{m}e^{\gamma_{3}t_{i+1}}\|A^{\frac12}e_{i+1}^{N}\|^2
    \end{align*}
    Through simple calculation, we rewrite the inequality \eqref{summation of step 1} as  
 	\begin{align*}
 			&\|A^{-\frac12} e_{m+1}^{N}\|^2+(1-2\varepsilon)\tau\sum_{i=0}^{m}e^{-\gamma_{3}(t_{m+1}-t_{i+1})}\|A^{\frac12} e_{i+1}^{N}\|^2 \\
 			\le &\frac{1}{\varepsilon}\tau\sum_{i=0}^{m}e^{-\gamma_{3}(t_{m+1}-t_{i+1})}\|A^{-\frac{1}{2}}e_{i+1}^{N}\|^2\\
 			&+C_{\varepsilon}\tau\sum_{i=0}^{m}e^{-\gamma_{3}(t_{m+1}-t_{i+1})} \bigl(1+\|u(t_{i+1})\|_{L^{\infty}}^4+\|\tilde{u}_{i+1}^{N}\|_{L^{\infty}}^4\bigr)
 			\|u(t_{i+1})-\tilde{u}_{i+1}^{N}\|^2 \\
 			&+\sum_{i=0}^{m}e^{-\gamma_{3}(t_{m+1}-t_{i+1})}\|A^{-\frac{1}{2}}(G(u(t_{i}))-G(u_{i}^{N}))\Delta_i W\|^2\\
 			&+2\sum_{i=0}^{m}e^{-\gamma_{3}(t_{m+1}-t_{i+1})}\langle (G(u(t_{i}))-G(u_{i}^{N}))\Delta_i W, A^{-1} e_{i}^{N}\rangle \\
 			=:&I_1+I_2+I_3+I_4.
 	\end{align*}
 	By squaring the inequality, followed by taking expectations and applying Lemma \ref{error for full discrete in H^-1 space}, we derive
 	\begin{equation*}
 		\begin{aligned}
 			\mathbb{E}\big [I_1^2 \big ]
 			\le& C\mathbb{E}\Big [ \tau \sum_{i=0}^{m}e^{-\gamma_{3}(t_{m+1}-t_{i+1})}\|A^{-\frac12}e_{i+1}^{N}\|^4\cdot \tau \sum_{i=0}^{m}e^{-\gamma_{3}(t_{m+1}-t_{i+1})}\Big ]\\
 			\le& C\big(\lambda _{N+1}^{-4}+\tau^{2}\big).
 		\end{aligned}
 	\end{equation*}
 	Similarly, by applying Theorem \ref{continous uniform moment}, \eqref{auxilary solution stability} and Theorem \ref{error estimate with auxilary solution}, we have 
 	\begin{equation*}
 		\mathbb{E}\big [I_2^2 \big ]\le C\big(\lambda _{N+1}^{-4}+\tau^{2}\big).
 	\end{equation*}
 	For the third term, according to \eqref{global Lipschitz in H^-1} in Assumption \ref{diffusion coeffient}, Theorem \ref{error estimate with auxilary solution}, Lemma \ref{error for full discrete in H^-1 space}, the independence of $\Delta_i W$ and $\Delta_j W$ when $i\ne j$ and the It\^{o} isometry, we derive
 	\begin{equation*}
 		\begin{aligned}
 			\mathbb{E}\big [I_3^2 \big ]
 			\le& C\mathbb{E}\Big [\sum_{i=0}^{m}e^{-\gamma_{3}(t_{m+1}-t_{i+1})}\|A^{-\frac{1}{2}}(G(u(t_{i}))-G(u_{i}^{N}))\Delta_i W\|^4\cdot 
 			\sum_{i=0}^{m}e^{-\gamma_{3}(t_{m+1}-t_{i+1})}\Big ]\\
 			\le& C\sum_{i=0}^{m}e^{-\gamma_{3}(t_{m+1}-t_{i+1})}\mathbb{E} [\|A^{-\frac{1}{2}}(G(u(t_{i}))-G(u_{i}^{N}))\Delta_i W\|^4]
 			\cdot \sum_{i=0}^{m}e^{-\gamma_{3}(t_{m+1}-t_{i+1})}\\
 			\le& C\tau\sum_{i=0}^{m}e^{-\gamma_{3}(t_{m+1}-t_{i+1})}\mathbb{E}\big [\|A^{-\frac{1}{2}}(G(u(t_{i}))-G(u_{i}^{N}))\|_{\mathcal{L}_{2}^{0}}^4 \big ]\\
 			\le & C\big(\lambda _{N+1}^{-4}+\tau^{2}\big).
 		\end{aligned}
 	\end{equation*}
 	Similarly, for the last term, through \eqref{global Lipschitz in H^-1} in Assumption \ref{diffusion coeffient}, Theorem \ref{error estimate with auxilary solution}, Lemma \ref{error for full discrete in H^-1 space} and the Burkh\"{o}lder-Davis-Gundy inequality, we can deduce
 	\begin{align*}
 			\mathbb{E}\big [I_4^2 \big ]
 			\le& C\tau\mathbb{E}\Big [ \sum_{i=0}^{m}e^{-2\gamma_{3}(t_{m+1}-t_{i+1})}\|A^{-\frac{1}{2}}e_{i}^{N}\|^2\|A^{-\frac{1}{2}}(G(u(t_{i}))-G(u_{i}^{N}))\|_{\mathcal{L}_{2}^{0}}^2\Big ] \\
 			\le& C\tau\mathbb{E}\Big [ \sum_{i=0}^{m}e^{-2\gamma_{3}(t_{m+1}-t_{i+1})}\|A^{-\frac{1}{2}}e_{i}^{N}\|^2 \left (2L_{2}\|A^{-\frac{1}{2}}(u(t_i)-\tilde{u}_{i}^{N})\|^2+2L_{2}\|A^{-\frac{1}{2}}e_{i}^{N}\|^2\right )\Big ]\\
 			\le& C\tau\mathbb{E}\Big [ \sum_{i=0}^{m}e^{-2\gamma_{3}(t_{m+1}-t_{i+1})}\big(\|A^{-\frac{1}{2}}e_{i}^{N}\|^4+\|A^{-\frac{1}{2}}(u(t_i)-\tilde{u}_{i}^{N})\|^4\big )\Big ]\\
 			\le&C\big(\lambda _{N+1}^{-4}+\tau^{2}\big).
 	\end{align*}
 	Combining the above inequalities together, we complete the proof of \eqref{middle process in error estimate}. \\
 	\textbf{Step 2.} According to the definition of $e_{m+1}^{N}$ given by \eqref{error full-discrete scheme}, one can deduce 
 	\begin{equation*}
 		\begin{aligned}
 			\|e_{m+1}^{N}\|_{L^2(\Omega;H)}\le& \Big\|\tau\sum_{i=0}^{m}S_{\tau,N}^{m+1-i}AP_N(F(u(t_{i+1}))-F(\tilde{u}_{i+1}^{N}))\Big\|_{L^2(\Omega;H)}\\
 			&+\Big\|\tau\sum_{i=0}^{m}S_{\tau,N}^{m+1-i}AP_N(F(\tilde{u}_{i+1}^{N})-F(u_{i+1}^{N}))\Big\|_{L^2(\Omega;H)}\\
 			&+\Big\|\sum_{i=0}^{m}S_{\tau,N}^{m+1-i}P_N(G(u(t_i))-G(\tilde{u}_{i}^{N}))\Delta_iW\Big\|_{L^2(\Omega;H)}\\
 			&+\Big\|\sum_{i=0}^{m}S_{\tau,N}^{m+1-i}P_N(G(\tilde{u}_{i}^{N})-G(u_{i}^{N}))\Delta_iW\Big\|_{L^2(\Omega;H)}\\
 			=:&J_1+J_2+J_3+J_4.
 		\end{aligned}
 	\end{equation*}
 	By Lemma \ref{semigroup full discrte} with $\nu =1$, Theorem \ref{continous uniform moment}, \eqref{auxilary solution stability} and Theorem \ref{error estimate with auxilary solution} yields 
 		\begin{align*}
 			J_1\le& \tau\sum_{i=0}^{m}\big\|S_{\tau,N}^{m+1-i}AP_N(F(u(t_{i+1}))-F(\tilde{u}_{i+1}^{N}))\big\|_{L^2(\Omega,H)}\\
 			\le& C\tau\sum_{i=0}^{m}(t_{m+1}-t_i)^{-\frac{1}{2}}e^{-\frac{c\lambda_{1}^{2}}{2}(t_{m+1}-t_i)}\|F(u(t_{i+1}))-F(\tilde{u}_{i+1}^{N})\|_{L^2(\Omega,H)}\\
 			\le& C\tau\sum_{i=0}^{m}(t_{m+1}-t_i)^{-\frac{1}{2}}e^{-\frac{c\lambda_{1}^{2}}{2}(t_{m+1}-t_i)}\big (1+\|u(t_{i+1})\|_{L^6(\Omega,L^{{\infty}})}^2+\|\tilde{u}_{i+1}^{N}\|_{L^6(\Omega,L^{{\infty}})}^2\big )\|u(t_{i+1})-\tilde{u}_{i+1}^{N}\|_{L^{6}(\Omega,H)}\\
 			\le &C\big(\lambda _{N+1}^{-1}+\tau^{\frac{1}{2}}\big).
 		\end{align*}
 	By Lemma \ref{semigroup full discrte} with $\nu =\frac{1}{2}$, Theorem \ref{all regularity of full discrete}, \eqref{auxilary solution stability} and \eqref{middle process in error estimate}, one can derive, for $\delta>\frac{d}{2}$, that
 		\begin{align*}
 			J_2\le &C\Big(\mathbb{E}\Big [ \Big(\tau\sum_{i=0}^{m}\big\|S_{\tau,N}^{m+1-i}AP_N(F(\tilde{u}_{i+1}^{N})-F(u_{i+1}^{N}))\big\|\Big)^2 \Big]\Big)^{\frac{1}{2}}\\
 			\le &C\Big(\mathbb{E}\Big [ \Big(\tau\sum_{i=0}^{m}(t_{m+1}-t_i)^{-\frac{1}{4}}e^{-\frac{c\lambda_{1}^{2}}{2}(t_{m+1}-t_i)}\big\|A^{\frac{1}{2}}(F(\tilde{u}_{i+1}^{N})-F(u_{i+1}^{N}))\big\|\Big)^2 \Big]\Big)^{\frac{1}{2}}\\
 			\le& C\Big(\mathbb{E}\Big [ \Big(\tau\sum_{i=0}^{m}(t_{m+1}-t_i)^{-\frac{1}{4}}e^{-\frac{c\lambda_{1}^{2}}{2}(t_{m+1}-t_i)}\big (1+\|\tilde{u}_{i+1}^{N}\|_{\delta}^2+\|u_{i+1}^{N}\|_{\delta}^2\big )\|e_{i+1}^{N}\|_1 \Big)^2 \Big]\Big)^{\frac{1}{2}}\\
 			\le& C\Big(\mathbb{E}\Big [ \tau\sum_{i=0}^{m}(t_{m+1}-t_i)^{-\frac{1}{2}}e^{-\frac{c\lambda_{1}^{2}}{2}(t_{m+1}-t_i)}\big (1+\|\tilde{u}_{i+1}^{N}\|_{\delta}^4+\|u_{i+1}^{N}\|_{\delta}^4 \big) \cdot \tau\sum_{i=0}^{m}e^{-\frac{c\lambda_{1}^{2}}{2}(t_{m+1}-t_i)}\|A^{\frac{1}{2}}e_{i+1}^{N}\|^2 \Big]\Big)^{\frac{1}{2}}\\
 			\le &C\Big(\mathbb{E}\Big [ \Big (\tau\sum_{i=0}^{m}(t_{m+1}-t_i)^{-\frac{1}{2}}e^{-\frac{c\lambda_{1}^{2}}{2}(t_{m+1}-t_i)}\big (1+\|\tilde{u}_{i+1}^{N}\|_{\delta}^4+\|u_{i+1}^{N}\|_{\delta}^4 \big) \Big )^2 \Big]\Big)^{\frac{1}{4}}\\
 			 &\times\Big(\mathbb{E}\Big [ \Big (\tau\sum_{i=0}^{m}e^{-\frac{c\lambda_{1}^{2}}{2}(t_{m+1}-t_i)}\|A^{\frac{1}{2}}e_{i+1}^{N}\|^2 \Big )^2 \Big]\Big)^{\frac{1}{4}}
 			\le C\big(\lambda _{N+1}^{-1}+\tau^{\frac{1}{2}}\big).
 		\end{align*}
 	 By applying the Burkh\"{o}lder-Davis-Gundy inequality and \eqref{global Lipschitz in H} in Assumption \ref{diffusion coeffient}, we can deduce
 	 	\begin{align*}
 	 		J_3\le& C\Big(\tau\sum_{i=0}^{m}\mathbb{E} \Big[\big\|S_{\tau,N}^{m+1-i}P_N(G(u(t_i))-G(\tilde{u}_{i}^{N}))\big\|_{\mathcal{L}_{2}^{0}}^2 \Big ]\Big)^{\frac{1}{2}} \\
 	 		\le& C\Big(\tau\sum_{i=0}^{m}e^{-c\lambda_{1}^{2}(t_{m+1}-t_i)}\mathbb{E}\big[\big\|G(u(t_i))-G(\tilde{u}_{i}^{N})\big\|_{\mathcal{L}_{2}^{0}}^2 \big ]\Big)^{\frac{1}{2}}\\
 	 		\le& C\Big(\tau\sum_{i=0}^{m}e^{-c\lambda_{1}^{2}(t_{m+1}-t_i)}\mathbb{E}\big[\big\|u(t_i)-\tilde{u}_{i}^{N}\big\|^2 \big ]\Big)^{\frac{1}{2}}
 	 		\le C\big(\lambda _{N+1}^{-1}+\tau^{\frac{1}{2}}\big).
 	 	\end{align*}
    In a similar manner, utilizing Lemma \ref{error for full discrete in H^-1 space}, we can derive
 	\begin{align*}
 		J_4\le& C\Big(\tau\sum_{i=0}^{m}\mathbb{E} \Big[\big\|S_{\tau,N}^{m+1-i}P_N(G(\tilde{u}_{i}^{N})-G(u_{i}^{N}))\big\|_{\mathcal{L}_{2}^{0}}^2 \Big ]\Big)^{\frac{1}{2}} \\
 		\le& C\Big(\tau\sum_{i=0}^{m}(t_{m+1}-t_i)^{-\frac{1}{2}}e^{-c\lambda_{1}^{2}(t_{m+1}-t_i)}\mathbb{E}\big[\big\|G(\tilde{u}_{i}^{N})-G(u_{i}^{N})\big\|_{\mathcal{L}_{2}^{-1}}^2 \big ]\Big)^{\frac{1}{2}}\\
 		\le& C\Big(\tau\sum_{i=0}^{m}(t_{m+1}-t_i)^{-\frac{1}{2}}e^{-c\lambda_{1}^{2}(t_{m+1}-t_i)}\mathbb{E}\big[\big\|A^{-\frac{1}{2}}e_i^N\big\|^2 \big ]\Big)^{\frac{1}{2}}
 		\le C\big(\lambda _{N+1}^{-1}+\tau^{\frac{1}{2}}\big).
 	\end{align*}
 	Combining the above inequalities together, the proof is complete.
 	
 \end{proof}

\section{Numerical experiments}\label{numerical experiments}
  Some numerical tests are presented in this section to illustrate the convergence rates in time and space of the proposed scheme.

  \begin{example}[Strong convergence rates]
  	We consider the following stochastic Cahn--Hilliard equation driven by multiplicative noise
  	\begin{equation}\label{CH-equation}
  		{\small	\left\{
  			\begin{aligned}
  				&\frac{\partial u(t,x)}{\partial t}+\Delta^2u(t,x)-\Delta(u^3(t,x)-u(t,x))
  				= G(u(t,x))\frac{\partial W(t,x)}{\partial t},
  				&& t\in (0,T],\ x\in (0,1), \\
  				&u(0,x)
  				= 0.1+\sum_{k=1}^{N}\frac{\sqrt{2}}{10}k^{-3}\cos(k\pi x),
  				&& x\in (0,1), \\
  				&\frac{\partial u}{\partial x}\Big|_{x=0}
  				= \frac{\partial u}{\partial x}\Big|_{x=1}=\frac{\partial^3 u}{\partial x^3}\Big|_{x=0}
  				= \frac{\partial^3 w}{\partial x^3}\Big|_{x=1} = 0,
  				&& t\in (0,T]. \\
  			\end{aligned}
  			\right. }
  	\end{equation}
  \end{example}
  \noindent Here we choose the diffusion coefficient $G$ such that $G(u)Q^{\frac{1}{2}}\phi_k=\sqrt{q_k}A^{-\frac{1}{8}}\mathcal{P}(u\phi_k)$ for $k \ge 1$, where $\mathcal{P}$ denotes the orthogonal projection on $\dot{H}$. The driving process $W(t,x)$ is an $\dot H$-valued $Q$-Wiener process given by
  \begin{equation*}
  	W(t,x)
  	=
  	\sum_{j=1}^{\infty}
  	\frac{\sqrt{2}}{(j\pi)^2}
  	\cos(j\pi x)\beta_j(t),
  \end{equation*}
  where $\beta_{j}(t)$ are independent standard Brownian motions. 
  In what follows, the orthonormal eigen-pairs $\left \{\lambda_{k}, \phi_{k}\right\}_{k\in \mathbb{N}}$ of Neumann Laplacian on $H$ is 
   \begin{equation*}
   	\lambda _{k}=k^{2}\pi ^{2}, \quad \phi_{k}(x)=\sqrt{2}\cos(k\pi x), \quad k\ge 1, \quad \phi_{0}(x)=1.
   \end{equation*}
  Since the exact solution of Eq.\eqref{CH-equation} is not available, we choose numerical solutions with very small step size for reference. We take 1000 sample paths and choose $M_{ref}=2^{11}$ and $N_{ref}=2^{8}$. The error at the terminal time $T$ is approximated
  respectively by
  \begin{equation*} 
  	\begin{aligned}
  		&E\left ( \tau ,N \right )=\max_{0\le m\le M}\Big( \frac{1}{1000} \sum_{i=1}^{1000}\left \|u^{\tau,N}(t_m,\omega_{i})-u_{m}^{N}(\omega_{i})\right \|^{2} \Big )^{\frac{1}{2}},
  	\end{aligned}
  \end{equation*}
  where $u_{m}^{N}(\omega_{i})$ denote the numerical solution at $t_m$ and $i$-th paths with $N$ orthonormal bases and $M=\frac{T}{\tau}$. 
  
  We now illustrate the uniform-in-time strong convergence rates of the fully
  discrete scheme established in Theorem~\ref{main theorem}. To highlight the
  long-time behavior, we take $T=100$. Figure~\ref{FIG:1} displays the spatial
  and temporal errors for different choices of $N$ and $M$, respectively. The
  corresponding spatial errors and convergence orders are reported in
  Table~\ref{convergence rates}.
  
  \begin{figure}
  	\centering
  	\includegraphics[scale=0.5]{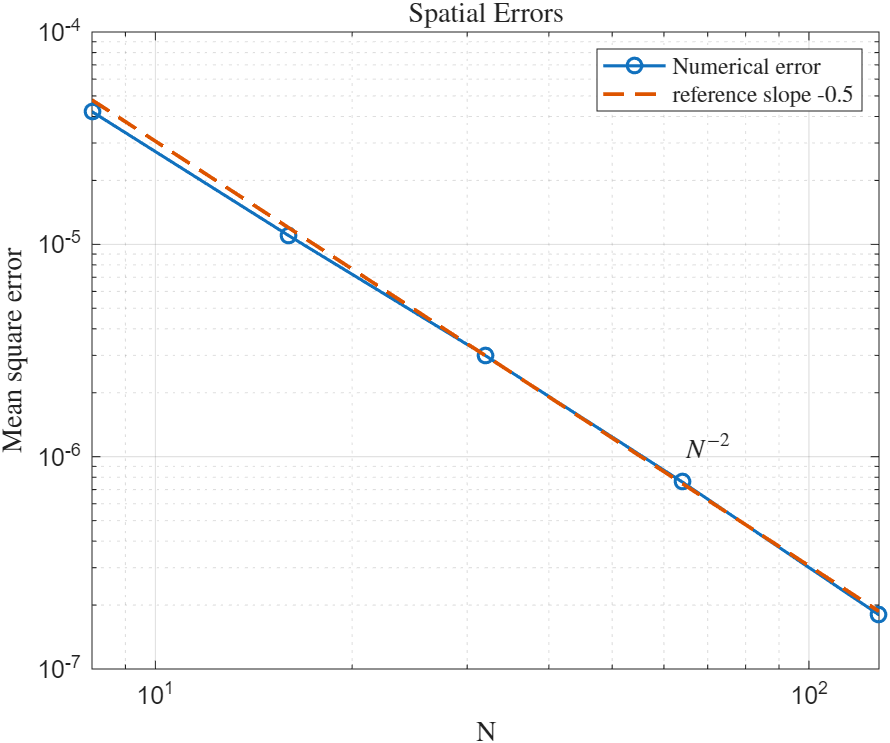}
  	\includegraphics[scale=0.5]{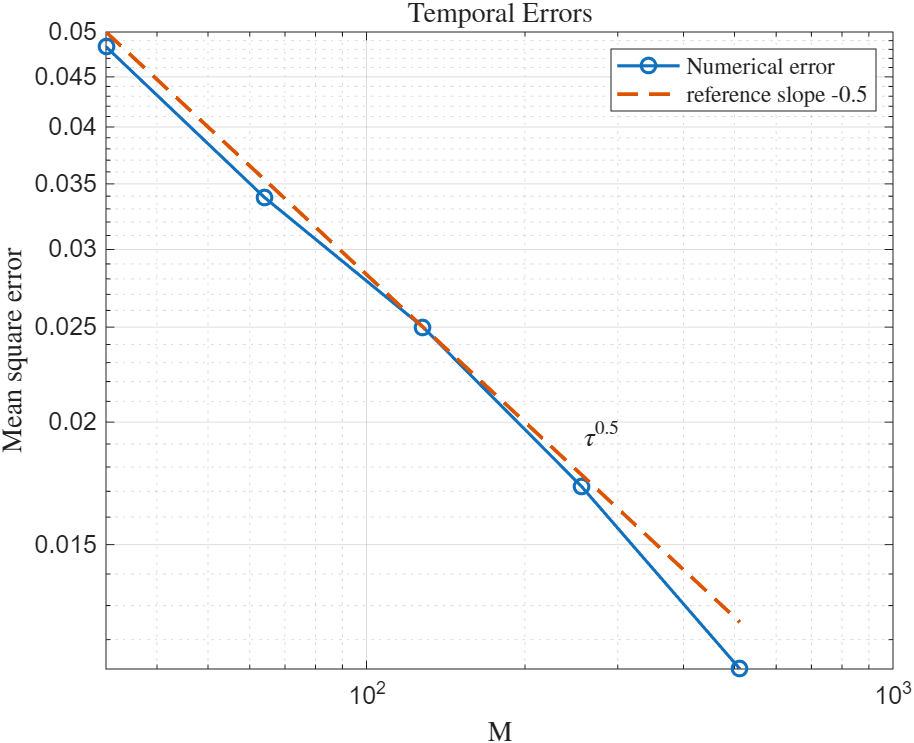}
  	\caption{Errors associated with 1000 samples and their variation with $N$ (left) and $M$ (right). }
  	\label{FIG:1}
  \end{figure}

  \begin{table}[h]
  	\caption{Errors and strong convergence rates of the numerical scheme with $T=100$.}
  	\label{Table1}
  	\centering
  	\renewcommand{\arraystretch}{1.2}
  	\setlength{\tabcolsep}{8pt}
  	\begin{tabular}{c c c c c c}
  		\hline
  		$\left(M,N\right)$ 
  		& $E\left (\tau,N\right)$ 
  		& order 
  		& $\left(M,N\right)$ 
  		& $E\left (\tau,N \right)$ 
  		& order \\
  		\hline
  		$\left(2^{11},8\right)$   
  		& $0.4216\times 10^{-4}$ 
  		& --   
  		& $\left(32,256\right)$   
  		& 0.0483 
  		& -- \\
  		
  		$\left(2^{11},16\right)$  
  		& $0.1099\times 10^{-4}$ 
  		& 1.9395 
  		& $\left(64,256\right)$  
  		& 0.0339 
  		& 0.5129  \\
  		
  		$\left(2^{11},32\right)$  
  		& $0.0299\times 10^{-4}$ 
  		& 1.8801 
  		& $\left(128,256\right)$  
  		& 0.0250 
  		& 0.4367 \\
  		
  		$\left(2^{11},64\right)$  
  		& $0.0076\times 10^{-4}$ 
  		& 1.9744 
  		& $\left(256,256\right)$  
  		& 0.0172 
  		& 0.5404 \\
  		
  		$\left(2^{11},128\right)$ 
  		& $0.0018\times 10^{-4}$ 
  		& 2.0477 
  		& $\left(512,256\right)$ 
  		& 0.0112
  		& 0.6148 \\
  		\hline
  		Expected order 
  		& 
  		& $2$ 
  		& 
  		& 
  		& $0.5$ \\
  		\hline
  	\end{tabular}\label{convergence rates}
  \end{table}

 \begin{example}
 	We consider the stochastic Cahn--Hilliard equation driven by multiplicative noise on the two-dimensional domain $\mathcal{D}=(0,1)^{2}$:
 	\begin{equation}\label{eq:SCH}
 		\left\{
 		\begin{aligned}
 			&du + A\left(Au + u^3-u\right)dt
 			= G(u)dW(t), 
 			&& t\in[0,T],\quad x\in\mathcal{D},\\
 			&u(0,x)=u_0,
 			&& x\in \mathcal{D}.
 		\end{aligned}
 		\right.
 	\end{equation}
 \end{example}
  \noindent Here, the diffusion coefficient is defined by $G(u)Q^{\frac{1}{2}}\phi_k=\sqrt{q_k}A^{-\frac{1}{4}}\mathcal{P}(u\phi_k)$ for $k \ge 1$, where $\mathcal{P}$ denotes the orthogonal projection onto the zero-mean space $\dot{H}$. 
  
  Let $M=2^{11}$, $N=2^{8}$ and $T=100$.
  We choose the bounded test function $\Phi(v)=\exp\left(-\|v\|_{-1}^2\right)$ and define the corresponding empirical time average
  \begin{equation*}
  	\mathcal A_m(u_0)
  	=\frac{1}{m+1}\sum_{j=0}^{m}
  	\Phi\!\left(u_j^{N}(u_0)\right).
  \end{equation*} 
  where $u_j^{N}(u_0)$ denotes the numerical solution starting from $u_0$.
  
 Figure~\ref{FIG:2} shows that the empirical averages corresponding to different initial conditions with the same conserved mass approach approximately the same value. These numerical solutions evolve in the same invariant affine space $H_{N,-1}^{\alpha}$. By contrast, initial conditions with different conserved spatial averages yield different limiting values of the chosen observable. Thus, for the test function $\Phi$ and the numerical parameters considered here, the long-time average appears to be insensitive to the initial spatial profile within a fixed mass space but depends on the conserved mass. These observations are consistent with the theoretical existence and uniqueness of the invariant measure $\pi_{\alpha}^{\tau,N}$ on each $H_{N,-1}^{\alpha}$。

  \begin{figure}
  	\centering
  	\includegraphics[scale=0.55]{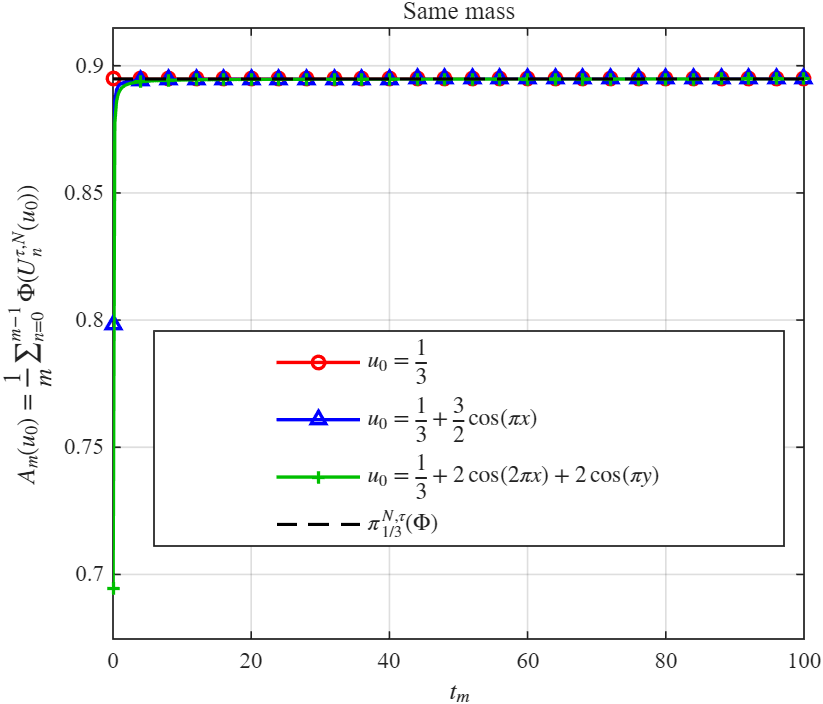}
  	\includegraphics[scale=0.55]{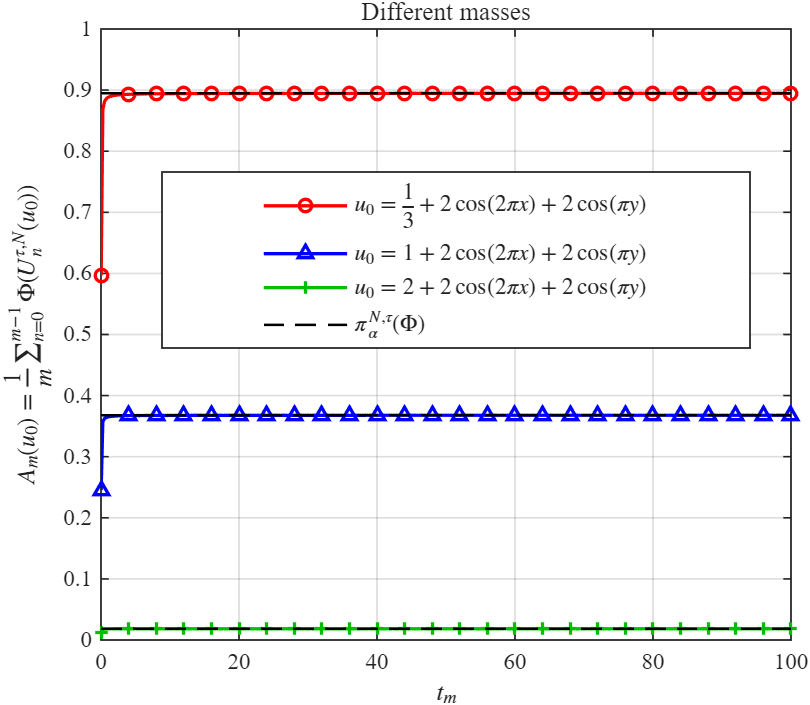}
  	\caption{Empirical time averages of $\Phi(v)$. Left: three initial conditions with the same conserved mass $\alpha=\frac{1}{3}$. Right: three initial conditions with conserved masses $\alpha=\frac{1}{3}$, $1$ and $2$, respectively.}
  	\label{FIG:2}
  \end{figure}

  \section*{Data Availability}
  Data will be made available on request.
  \section*{Declaration of competing interest}
  The authors declare that they have no known competing financial interests or personal relationships that could have appeared to influence the work reported in this paper.
  \section*{Statements and Declarations}
  {This work was supported by the National Natural Science Foundation of China (No. 12571420),
  	and the funding from the Jiangsu Provincial Scientific Research Center of Applied Mathematics
  	(No. BK20233002). }

  \appendix
  \section{Proof of Theorem \ref{continous uniform moment}}\label{detailed proof of Theorem 2.4}
       By the inequality \eqref{semigroup_1} with $\nu=1+\frac{\gamma}{2}$ and $\nu=\frac{\gamma}{2}$, we have
 		\begin{align*}
 			\|u(t)\|_{L^{p}(\Omega;H^{\gamma })}
 			\le	& \|S(t)u_0\|_{L^{p}(\Omega;H^{\gamma })}
 			+ \int_0^t \|AS(t-s) F(u(s))\|_{L^{p}(\Omega;H^{\gamma })} ds
 			\\
 			&+\Big \|\int_{0}^{t}S(t-s)G(u(s))dW(s)\Big \|_{L^{p}(\Omega ;H^{\gamma})}\\
 			\le& C\|u_0\|_{L^{p}(\Omega;H^{\gamma})}+C \int_0^t (t-s)^{-\frac{1}{2}-\frac{\gamma}{4}}e^{-\frac{\lambda _{1}^{2}}{2}(t-s) }
 			\left( 1+\|u(s)\|_{L^{3p}(\Omega;L^6)}^3\right) ds\\
 			&+C\left(\int_{0}^{t}(t-s)^{-\frac{\gamma}{2}}e^{-\lambda_{1}^{2}(t-s)}\Big(\mathbb{E}[\left \| G(u(s)) \right \|_{\mathcal{L}_2^0}^{p}]\Big )^{\frac{2}{p}} ds\right )^{\frac{1}{2}}.
 		\end{align*}
 		By Lemma \ref{uniform estimate of u(t) in H^1} and \eqref{linear growth in H} in Assumption \ref{diffusion coeffient}, we complete the proof of the inequality \eqref{uniform stability}.

 		For the second term \eqref{time regularity}, through the definition of $u(t)$, we get
 		\begin{equation*}
 			\begin{aligned}
 				&\|u(t)-u(s)\|_{L^{p}(\Omega;H)}\\ \le&\|(S(t-s)-I)S(s)u_0\|_{L^{p}(\Omega;H)}
 				+\int_{0}^{s}\|(S(t-s)-I)S(s-r)AF(u(r))\|_{L^{p}(\Omega;H)}dr\\
 				&+\int_{s}^{t}\|S(t-r)AF(u(r))\|_{L^{p}(\Omega;H)}dr
 				+\Big\|\int_{0}^{s}(S(t-s)-I)S(s-r)G(u(r))dW(r)\Big \|_{L^{p}(\Omega;H)}\\
 				&+\Big \|\int_{s}^{t}S(t-r)G(u(r))dW(r)\Big \|_{L^{p}(\Omega;H)}\\
 				=:&I_1+I_2+I_3+I_4+I_5.
 			\end{aligned}
 		\end{equation*}
 		Employing \eqref{semigroup_3} with $\mu=0, \rho=1, r=\frac{1}{2}$ yields
 		\begin{equation*}
 			I_1\le C(t-s)^{\frac{1}{2}}\|u_0\|_{L^{p}(\Omega;H^2)}.
 		\end{equation*}
 		By \eqref{semigroup_3} with $\mu=1, \rho=\frac{1}{2}, r=\frac{1}{2}$ and the inequality \eqref{uniform stability}, we have
 			\begin{align*}
 				I_2\le& C\int_{0}^{s}(s-r)^{-\frac{3}{4}}(t-s)^{\frac{1}{2}}e^{-\frac{\lambda_{1}^{2}}{2}(s-r)}\|A^{\frac{1}{2}} F(u(r))\|_{L^{p}(\Omega;H)}dr \\
 				\le& C(t-s)^{\frac{1}{2}}\sup_{t>0} \mathbb{E}\big [\|u(t)\|_{L^{\infty}}^{2p}\|\nabla u(t)\|^{p}\big ]\int_{0}^{s}(s-r)^{-\frac{3}{4}}e^{-\frac{\lambda_{1}^{2}}{2}(s-r)}dr \\
 				\le& C(t-s)^{\frac{1}{2}}.
 			\end{align*}
 		Similarly, for $I_3$, by \eqref{semigroup_1} with $\nu=1$ we have
 		\begin{equation*}
 			I_3 \le C\int_{s}^{t}(t-r)^{-\frac{1}{2}}e^{-\frac{\lambda_{1}^{2}}{2}(t-r)}\|F(u(r))\|_{L^{p}(\Omega,H)}dr \le C(t-s)^{\frac{1}{2}}.
 		\end{equation*}
 		Utilizing \eqref{semigroup_3} with $\mu=0, \rho=\frac{1}{2} , r=\frac{1}{2}$, the inequality \eqref{uniform stability}, \eqref{linear growth in H^1} in Assumption \ref{diffusion coeffient} and the Burkh\"{o}lder-Davis-Gundy inequality, one can deduce
 		\begin{equation}
 			\begin{aligned}
 				I_4\le& C\Big(\int_{0}^{s}\Big(\mathbb{E}\big [\big \|(S(t-s)-I)S(s-r)G(u(r))\big \|_{\mathcal{L}_2^0}^{p}\big ]\Big )^{\frac{2}{p}} dr\Big )^{\frac{1}{2}} \\
 				\le& C\Big(\int_{0}^{s}(t-s)(s-r)^{-\frac{1}{2}}e^{-\lambda _{1}^{2}(s-r) } \Big( \mathbb{E}\big [\left \|G(u(r))\right \|_{\mathcal{L}_2^{1}}^{p}\big ] \Big )^{\frac{2}{p}} dr\Big )^{\frac{1}{2}}\\
 				\le& C(t-s)^{\frac{1}{2}}\Big(\int_{0}^{s}(s-r)^{-\frac{1}{2} }e^{-\lambda _{1}^{2}(s-r) } \Big( \mathbb{E}\big [\left \|u(r)\right\|_{1}^{p}\big ]+1 \Big )^{\frac{2}{p}} dr\Big )^{\frac{1}{2}}\\
 				\le& C(t-s)^{\frac{1}{2}}.		
 			\end{aligned}
 		\end{equation}
 		For the last term $I_5$, utilizing \eqref{semigroup_1} with $\nu =0$, Lemma \ref{uniform boundness in L^2}, \eqref{linear growth in H} in Assumption \ref{diffusion coeffient} and the Burkh\"{o}lder-Davis-Gundy inequality, one can deduce
 		\begin{equation}
 			\begin{aligned}
 				I_5\le& C\left(\int_{s}^{t}\Big(\mathbb{E}\big [\big \|S(t-r)G(u(r))\big \|_{\mathcal{L}_2^0}^{p}\big ] \Big)^{\frac{2}{p}} dr\right )^{\frac{1}{2}} \\
 				\le& C\left(\int_{s}^{t} \big(\mathbb{E}[\left \|u(r)\right \|^{p}+1 ]\big)^{\frac{2}{p}} dr\right )^{\frac{1}{2}}\\
 				\le &C(t-s)^{\frac{1}{2}} .
 			\end{aligned}
 		\end{equation}
 		Combining the above inequalities all, we complete the proof.

  \section{Proof of Lemma \ref{semigroup full discrte}}\label{semigroup for full discrete}
  
  Set $r(z)=(1+z)^{-1}$. As shown in the proof of Theorem 7.1 in \cite{thomee2007}, there exists two positive constants $0<c<1$ and $C$ such that 
  \begin{equation}\label{basis inequality}
  	\begin{aligned}
  		&r(z) \le e^{-cz}, &&z\in [0,1],  \\
  		&\left | r(z)-e^{-z} \right |  \le C z^2,  &&z\in [0,1]. 
  	\end{aligned}
  \end{equation}
  By choosing $\tau$ small enough, one can deduce that $\tau \lambda _{1}^{2}\le 1$. For the case $\nu \in [0,1]$, it follows from the Parseval identity and the inequality \eqref{basis inequality} that
  	\begin{align*}
  		\|A^\nu S_{\tau,N}^m \varphi\|^2=&\sum_{k=0}^N \lambda_k^{2\nu} (1+\tau \lambda_k^2)^{-2m}\langle \varphi, \phi_k \rangle^2\\
  		=&(1+\tau \lambda_1^2)^{-m}\sum_{k=1}^N \lambda_k^{2\nu} (1+\tau \lambda_k^2)^{-m} \langle\varphi, \phi_k \rangle^2+k(\nu)\langle \varphi, \phi_0 \rangle^2\\
  		\le& e^{-c\lambda_1^2t_m}\sum_{k=1}^N\frac{\lambda_k^{2\nu}}{1+t_m \lambda_k^2}\langle\varphi, \phi_k \rangle^2+k(\nu)\langle \varphi, \phi_0 \rangle^2\\
  		\le &e^{-c\lambda_1^2t_m}\sum_{k=1}^N\big(\frac{t_m\lambda_k^{2}}{1+t_m \lambda_k^2}\big)^{\nu}\frac{1}{(1+t_m \lambda_k^2)^{1-\nu }}t_{m}^{-\nu }\langle\varphi, \phi_k \rangle^2+k(\nu)\langle \varphi, \phi_0 \rangle^2 \\
  		\le &t_{m}^{-\nu }e^{-c\lambda_1^2t_m}\|\varphi \|^2+k(\nu)\langle \varphi, \phi_0 \rangle^2.
  	\end{align*}
  Similarly, for the case $\nu \in (1,2)$, we have 
  \begin{equation*}
  	\begin{aligned}
  		\|A^\nu S_{\tau,N}^m \varphi\|^2
  		=&(1+\tau \lambda_1^2)^{-(2-\nu)m}\sum_{k=1}^N \lambda_k^{2\nu} (1+\tau \lambda_k^2)^{-\nu m} \langle\varphi, \phi_k \rangle^2+k(\nu)\langle \varphi, \phi_0 \rangle^2\\
  		\le &e^{-c(2-\nu)\lambda_1^2t_m}\sum_{k=1}^N\big(\frac{t_m\lambda_k^{2}}{1+t_m \lambda_k^2}\big)^{\nu}t_{m}^{-\nu }\langle\varphi, \phi_k \rangle^2+k(\nu)\langle \varphi, \phi_0 \rangle^2 \\
  		\le &t_{m}^{-\nu }e^{-c(2-\nu)\lambda_1^2t_m}\|\varphi \|^2+k(\nu)\langle \varphi, \phi_0 \rangle^2.
  	\end{aligned}
  \end{equation*}
  Thus the proof is complete.

  \section{Proof of Lemma \ref{error of semigrroup}}\label{solution operator difference}
  
   The proofs of \eqref{positive full semigroup} and \eqref{integral full semigroup} are given in \cite[Lemma 5.2]{qi2024} and are therefore omitted here. In what follows, we focus on proving the inequality \eqref{nonnegative full semigroup}.
  By the triangle inequality, \eqref{semigroup_1} with $\nu=\frac{\alpha -\beta}{2}$ and \eqref{semigroup_3} with $\mu=0, \rho=\frac{\beta}{2}, r=\frac{\alpha}{4}$ yields
  \begin{equation*}
  	\begin{aligned}
  		\|E_{\tau,N}(t)\varphi \|\le& \|S(t)(I-P_N)\varphi \|+\|S(t)(I-S(t_{m+1}-t))P_N\varphi \|
  		+\|(P_NS(t_{m+1})-S_{\tau,N}^{m+1})\varphi \|  \\
  		\le& C\lambda _{N+1}^{-\frac{\alpha }{2} }t^{-\frac{\alpha-\beta}{4}}e^{-\frac{\lambda _{1}^{2} }{2}t }\|\varphi\|_{\beta}
  		+Ct^{-\frac{\alpha-\beta}{4}}(t_{m+1}-t)^{\frac{\alpha }{4} }e^{-\frac{\lambda _{1}^{2} }{2}t }\|\varphi\|_{\beta}  
  		+\|(P_NS(t_{m+1})-S_{\tau,N}^{m+1})\varphi \|  \\
  		\le& C\big(\lambda _{N+1}^{-\frac{\alpha }{2} }+\tau^{\frac{\alpha}{4}}\big) t^{-\frac{\alpha-\beta}{4}}e^{-\frac{\lambda _{1}^{2} }{2}t }\|\varphi\|_{\beta}+\|(P_NS(t_{m+1})-S_{\tau,N}^{m+1})\varphi \|.
  	\end{aligned}
  \end{equation*}
  It remains to estimate $\|(P_NS(t_{m+1})-S_{\tau,N}^{m+1})\varphi \|$. By the the Parseval identity, it follows that
  \begin{equation*}
  	\begin{aligned}
  		\|(P_NS(t_{m+1})-S_{\tau,N}^{m+1})\varphi \|^2=\sum_{k=1}^{N}\big |e^{-\lambda _{k}^{2}t_{m+1}}-r(\tau\lambda _{k}^{2})^{m+1}\big |^2\left \langle \varphi ,\phi_k \right \rangle ^2=:I_1+I_2,
  	\end{aligned}
  \end{equation*}
  where 
  \begin{equation}\label{semigroup error}
  	\begin{aligned}
  		I_1:&=\sum_{k=1}^{n}\big |e^{-\lambda _{k}^{2}t_{m+1}}-r(\tau\lambda _{k}^{2})^{m+1}\big |^2\left \langle \varphi ,\phi_k \right \rangle ^2,\\
  		I_2:&=\sum_{k=n+1}^{N}\big |e^{-\lambda _{k}^{2}t_{m+1}}-r(\tau\lambda _{k}^{2})^{m+1}\big |^2\left \langle \varphi ,\phi_k \right \rangle ^2.
  	\end{aligned}
  \end{equation}
  We choose an appropriate $\tau$ such that $\tau\lambda_1^2\le 1$. Then we can find a positive integer $n$ such that $\tau \lambda _{n}^{2}\le 1 <\tau \lambda _{n+1}^{2}$. For the first case $\tau \lambda _{k}^{2}\le 1$ with $1\le k \le n$, utilizing the inequality \eqref{basis inequality}, we have
  	\begin{align}
  		\big |e^{-\lambda _{k}^{2}t_{m+1}}-r(\tau\lambda _{k}^{2})^{m+1}\big |
  		=&\big |(r(\tau\lambda _{k}^{2})-e^{-\lambda _{k}^{2}\tau })\sum_{i=0}^{m}r(\tau\lambda _{k}^{2})^{m-i} e^{-\lambda _{k}^{2}t_i} \big |\notag\\
  		\le& C\tau ^{2}\lambda _{k}^{4}(m+1)e^{-c\lambda _{k}^{2}t_{m+1}}\label{case 1}\\
  		\le &C(t_{m+1}\lambda _{k}^{2})^{1+\frac{\alpha -\beta }{4}}e^{-\frac{c\lambda _{k}^{2}}{2}t_{m+1}}
  		t_{m+1}^{-\frac{\alpha -\beta }{4}}(\tau \lambda _{k}^{2}) ^{1-\frac{\alpha}{4}}\tau ^{\frac{\alpha }{4}}\lambda _{k}^{\frac{\beta}{2}}e^{-\frac{c\lambda _{k}^{2}}{2}t_{m+1}}\notag\\
  		\le& Ct_{m+1}^{-\frac{\alpha -\beta }{4}}\tau ^{\frac{\alpha }{4}}\lambda _{k}^{\frac{\beta}{2}}e^{-\frac{c\lambda _{k}^{2}}{2}t_{m+1}}.\notag
  	\end{align}
  Substituting the inequality \eqref{case 1} into \eqref{semigroup error}, we can obtain
  \begin{equation}\label{third error}
  	I_1\le C\tau ^{\frac{\alpha }{2}}t_{m+1}^{-\frac{\alpha -\beta }{2}}e^{-c\lambda _{1}^{2}t_{m+1}}\|\varphi\|_{\beta}^2.
  \end{equation}
  For the case $\tau \lambda _{k}^{2}> 1$ with $n+1 \le k \le N$, we have 
  \begin{equation}\label{case 2}
  	\begin{aligned}
  		&\sup_{\tau \lambda _{k}^{2}> 1}e^{-\frac{m+1}{2}\tau\lambda _{k}^{2}}\le e^{-\frac{m+1}{2}}\le C(m+1)^{-\frac{\alpha}{2}}, &&\alpha \ge 0\\
  		&r(\tau\lambda _{k}^{2})^{\frac{m+1}{2}}\le \sup_{z\le 1}r(z)^{\frac{m+1}{2}}\le e^{-\frac{m+1}{2}}\le C(m+1)^{-\frac{\alpha}{2}}, &&\alpha \ge 0. 
  	\end{aligned}
  \end{equation}
  By simple calculation, we can divide $I_2$ into two parts
  \begin{equation*}
  	\begin{aligned}
  		I_2
  		\le& C\sum_{k=n+1}^{N}\big(e^{-2\lambda _{k}^{2}t_{m+1}}+r(\tau\lambda _{k}^{2})^{2m+2} \big)\left \langle \varphi ,\phi_k \right \rangle ^2\\
  		\le & Ce^{-\lambda _{1}^{2}t_{m+1}}\sum_{k=n+1}^{N}e^{-\lambda _{k}^{2}t_{m+1}}\left \langle \varphi ,\phi_k \right \rangle ^2
  		+C r(\tau\lambda _{1}^{2})^{m+1}\sum_{k=n+1}^{N}r(\tau\lambda _{k}^{2})^{m+1}\left \langle \varphi ,\phi_k \right \rangle ^2 
  		=:I_{21}+I_{22}.
  	\end{aligned}
  \end{equation*} 
  Through the inequality \eqref{case 2}, $\alpha \ge 0$ and $\beta \le 0$, we have
  \begin{equation*}
  	\begin{aligned}
  		I_{21}\le& Ce^{-\lambda _{1}^{2}t_{m+1}}\sum_{k=n+1}^{N}(t_{m+1}\lambda _{k}^{2})^{-\frac{\beta}{2}}e^{-\frac{\lambda _{k}^{2}}{2}t_{m+1}}(t_{m+1}\lambda _{k}^{2})^{\frac{\beta}{2}}(m+1)^{-\frac{\alpha }{2}}\left \langle \varphi ,\phi_k \right \rangle ^2 \\
  		\le& Ce^{-\lambda _{1}^{2}t_{m+1}}\tau ^{\frac{\alpha }{2}}t_{m+1}^{-\frac{\alpha -\beta }{2}}\sum_{k=n+1}^{N}\lambda _{k}^{\beta}\left \langle \varphi ,\phi_k \right \rangle ^2 \\
  		\le& Ce^{-\lambda _{1}^{2}t_{m+1}}\tau ^{\frac{\alpha }{2}}t_{m+1}^{-\frac{\alpha -\beta }{2}}\|\varphi\|_\beta^2.
  	\end{aligned}
  \end{equation*}
  Similarly, for the term $I_{22}$, through the inequality \eqref{case 2}, we deduce 
  \begin{equation*}
  	\begin{aligned}
  		I_{22}\le& Ce^{-c\lambda _{1}^{2}t_{m+1}}\sum_{k=n+1}^{N}\frac{1}{1+\frac{1}{2}t_{m+1}\lambda _{k}^{2}}(m+1)^{-\frac{\alpha }{2}}\left \langle \varphi ,\phi_k \right \rangle ^2 \\
  		\le &Ce^{-c\lambda _{1}^{2}t_{m+1}}\sum_{k=n+1}^{N}\Big(\frac{\frac{1}{2}t_{m+1}\lambda _{k}^{2}}{1+\frac{1}{2}t_{m+1}\lambda _{k}^{2}}\Big)^{-\frac{\beta}{2}}\big(1+\frac{1}{2}t_{m+1}\lambda _{k}^{2}\big)^{-1-\frac{\beta}{2}}(t_{m+1}\lambda _{k}^{2})^{\frac{\beta}{2}}(m+1)^{-\frac{\alpha }{2}}\left \langle \varphi ,\phi_k \right \rangle ^2\\
  		\le& Ce^{-c\lambda _{1}^{2}t_{m+1}}\tau ^{\frac{\alpha }{2}}t_{m+1}^{-\frac{\alpha -\beta }{2}}\|\varphi\|_\beta^2.
  	\end{aligned}
  \end{equation*}
  Combining the above inequalities together, we can deduce 
  \begin{equation*}
  	I_{2} \le Ce^{-c\lambda _{1}^{2}t_{m+1}}\tau ^{\frac{\alpha }{2}}t_{m+1}^{-\frac{\alpha -\beta }{2}}\|\varphi\|_\beta^2.
  \end{equation*}
  Thus the proof of the inequality \eqref{nonnegative full semigroup} is complete.

   \end{document}